\documentclass{conm-p-l}

\usepackage{}

\newtheorem{thm}{Theorem}[section]
\newtheorem{lemma}[thm]{Lemma}
\newtheorem{prop}[thm]{Proposition}
\newtheorem{cor}[thm]{Corollary}

\theoremstyle{definition}
\newtheorem{defin}[thm]{Definition}
\newtheorem{example}[thm]{Example}

\theoremstyle{remark}
\newtheorem{rem}[thm]{Remark}

\numberwithin{equation}{section}

\begin{document}

\title[Scissors Congruence]{Scissors Congruence with Mixed Dimensions}


\author{Thomas G. Goodwillie}
\address{Department of Mathematics, Brown University, Providence, RI  02912}
\email{tomg@math.brown.edu}


\date{May 31, 2016}

\begin{abstract}
We introduce a Grothendieck group $E_n$ for bounded polytopes in $\mathbb R^n$. It differs from the usual Euclidean scissors congruence group in that lower-dimensional polytopes are not ignored. We also define an analogous group $L_n$ using germs of polytopes at a point, which is related to spherical scissors congruence. This provides a setting for a generalization of the Dehn invariant. 
\end{abstract}

\maketitle

\tableofcontents

\parskip=4pt
\parindent=0cm



\section*{Introduction}

This project began as an exploration of the following definition. Let $E_n$ be the abelian group with a generator $\lbrack P\rbrack_n$ for each polytope (compact  PL subspace) in $\mathbb R^n$ and relations
$$
\lbrack P\cup Q\rbrack_n=\lbrack P\rbrack_n+\lbrack Q\rbrack_n-\lbrack P\cap Q\rbrack_n,
$$
$$ 
\lbrack \emptyset \rbrack_n=0,
$$
and 
$$
\lbrack g(P)\rbrack _n=\lbrack P\rbrack _n
$$
when $g:\mathbb R^n\to\mathbb R^n$ is an isometry. 

The $n$th Euclidean scissors congruence group, denoted here by $\hat E_n$, is the quotient of $E_n$ obtained by setting $\lbrack P\rbrack_n$ equal to $0$ whenever the dimension of $P$ is less than $n$. In other words $\hat E_n$ is the cokernel of the map $E_{n-1}\to E_n$ that sends $\lbrack P\rbrack_{n-1}$ to $\lbrack P\rbrack_n$. Thus the groups $E_n$ represent a refinement of the idea of scissors congruence. It seems to be worth looking into what other features they might have, and what light they might shed on the groups $\hat E_n$.

The groups $E_n$ form a graded commutative ring $E$. Let $p\in E_1$ be given by a point in $\mathbb R^1$, so that the map $E_{n-1}\to E_n$ mentioned above is multiplication by $p$. The graded ring $\hat E=E/pE$ is the usual scissors congruence ring. Euclidean volume gives a ring map $\mathcal V:E\to \mathbb R$ that factors through $\hat E$, and Euler characteristic gives a ring map $\chi:E\to \mathbb Z$ that does not. Another feature is a graded involution, which we call $I$ (for ``interior''), that sends $\lbrack P\rbrack_n$ to $(-1)^m\lbrack \text{int}P\rbrack_n=(-1)^m(\lbrack P\rbrack_n-\lbrack \partial P\rbrack_n)$ whenever $P$ is an $m$-dimensional convex polytope in $\mathbb R^n$. We have
$$
\mathcal V(I\xi)=(-1)^n\mathcal V(\xi)
$$$$
\chi(I\xi)=\chi(\xi)
$$
for $\xi\in E_n$. One can assign an element $\lbrack P\rbrack_n\in E_n$ to a polytope $P$ in any Euclidean space, as long as the dimension of $P$ is at most $n$.

Investigation of $E$ leads inevitably to a related object, which we call $L$. This is created from germs of polytopes at a point in the same way in which $E$ is created from compact polytopes. (The letter ``$L$'' stands for ``local''.) Technically we choose to work not with germs but with cones: closed PL subspaces of $\mathbb R^n$ that are invariant under dilatation by positive real numbers. The group $L_n$ has a generator $\langle P\rangle_n$ for every polytope germ (or cone), with defining relations analogous to those for $E_n$. Denote by $t$ the element of $L_1$ given by the germ,  at the origin, of the origin in $\mathbb R^1$. The quotient $\hat L_n=L_n/tL_{n-1}$ may be identified with the $(n-1)$st spherical scissors congruence group. There is a ring map $\mathcal U:L\to \mathbb R$ sending a germ in $\mathbb R^n$ to the $(n-1)$-dimensional volume of the corresponding subset of the unit sphere, normalized so that the volume of the entire sphere is $1$. There are also two ring maps $\epsilon$ and $e$ from $L$ to $\mathbb Z$. The first, the \emph{yes-no invariant}, simply takes the germ of a polytope $P$ at a point $v\in P$ to $1$ (and takes the empty germ, the germ of any polytope $P$ at a point $v\notin P$, to $0$) . The second, the \emph{relative Euler invariant}, takes the germ of $P$ at $v\in P$ to the relative Euler characteristic $\chi(P,P-v)$. Like the graded ring $E$, $L$ has an ``interior'' involution called $I$. There is also a duality involution $D:L\to L$, defined by sending a convex cone $P$ to the dual cone $DP$, the set of all points $w$ such that the nearest point to $w$ in $P$ is the origin. These invariants and involutions are related by 
$$
\mathcal U(I\xi)=(-1)^n\mathcal U(\xi)
$$$$
\epsilon (I\xi)=e(\xi)
$$$$
e(I\xi)=\epsilon(\xi)
$$$$
\epsilon(D\xi)=\epsilon(\xi)
$$$$
e(D\xi)=(-1)^ne(\xi)
$$$$
(D\circ I)(\xi)=(-1)^n(I\circ D)(\xi)
$$
for $\xi\in L_n$. The quantity $\mathcal W=\mathcal U\circ D:L\to \mathbb R$, which we call the \emph{dual volume}, plays an important role.

The relevance of $L$ to $E$ has to do with the following generalization of the Dehn invariant. Suppose that $k$ is a commutative ring, and that $F:L\to k$ and $G:E\to k$ are (ungraded) ring maps. We sometimes write $G_n(P)$ for $G(\lbrack P\rbrack_n)$ when $P$ is a polytope in $\mathbb R^n$, and write $F_n(P)$ for $F(\langle P\rangle_n)$ when $P$ is a cone in $\mathbb R^n$. If $F$ and $G$ are compatible in a certain way then they can be combined to make a new ring map $F\star G:E\to k$, as follows. The compatibility condition is $F(s)=G(p)$, where $p\in E_1$ is the point element defined above and $s$ is the \emph{line element} of $L_1$, the element given by the germ of the whole of $\mathbb R^1$ at the origin. The rule is:
$$
(F\star G)_n(P)=\sum_{0\le j\le n}\ \sum_{\text{dim}\sigma=j}\ F_{n-j}(\nu(\sigma,P))\cdot G_j(\text{int}\sigma).
$$
Here $P$ is a polytope in $\mathbb R^n$, the summation is over simplices of any triangulation of $P$, $\nu(\sigma,P)$ is (the germ of) the \emph{normal cone} of the simplex $\sigma$ in $P$ (a cone in an affine space orthogonal to the plane of $\sigma$ -- see \S\ref{local} and \S\ref{normal cones} below) and $G_j(\text{int}\sigma)$ is $G_j(\sigma)-G_j(\partial\sigma)$. 
Likewise, if $F:L\to k$ and $G:L\to k$ are ring maps such that $F(s)=G(t)$, then we define a ring map $F\star G:L\to k$ by essentially the same rule. This $\star$-product of multiplicative invariants is the central feature of this work. 

The $\star$-product is associative. In fact, in the second case, when both $F$ and $G$ have domain $L$, this product is the composition law in a groupoid, which we call $\mathcal G^O_k$. Its morphisms are the ring maps $L\to k$. The objects are the elements of $k$. The source and target of $F$ are $F(s)$ and $F(t)$. The identity morphism of $a\in k$ is the invariant that sends $\langle P\rangle_n$ to $a^n$ for every nonempty germ $P\subset \mathbb R^n$. The conical volume $U$ is an $\mathbb R$-valued morphism from $1$ to $0$. The inverse of a morphism $F$ in the groupoid is given by composing the ring map $F:L\to k$ with the duality involution $D:L\to L$; in particular $\mathcal W$ is the inverse of $\mathcal U$.

These ideas have been seen before on the scissors congruence level \cite{Sah}. If we restrict attention to those $F$ such that $F(t)=0=F(s)$, then the groupoid becomes a group and the ring $L$ is replaced by $L/(tL+sL)=\hat L/s\hat L$. This is a Hopf algebra, and the ring $\hat E$ is a comodule. (In the present more refined setting we may say that $L$ is a Hopf algebroid, representing the affine groupoid scheme $k\mapsto \mathcal G_k^O$. The $\star$-product corresponds to a ring map $L\to L\otimes_{\mathbb Z\lbrack u\rbrack}L$, where the tensor product is formed using the ring maps $L\leftarrow \mathbb Z\lbrack u\rbrack\to L$ that take the indeterminate $u$ to $s$ and $t$ respectively.)

In addition to yielding various kinds of generalized and refined Dehn invariants, the $\star$-product has other uses. It can be used to give an account of the \emph{intrinsic volume}, a canonical way of extending $n$-dimensional volume to polytopes of arbitrary dimension \cite{Had}. It can also be used to show that the quotient ring map $E\to \hat E$ has a right inverse that is also a graded ring map. This means that $E$ is isomorphic to the polynomial ring $\hat E\lbrack x\rbrack$, with $p\in E_1$ corresponding to the indeterminate $x$, and in particular that for each $n$ the group $E_n$ is isomorphic to the direct sum of the scissors groups $\hat E_j$ for $0\le j\le n$.

Because polytopes of lower dimension are not being ignored, our refined construction has a more topological flavor than scissors congruence has. We pay some attention to questions about which elements of $E_n$ and $L_n$ are represented by manifolds without boundary, and we make occasional reference to relations with differential geometry and characteristic classes. 

It was not until a supposedly final version of this paper had been written that the author became fully aware of the degree of overlap between his work and some earlier work of Peter McMullen. Before we investigated this refinement of Euclidean scissors congruence, McMullen had investigated the corresponding refinement of translational scissors congruence. This led to a thorough revision of the paper, largely because we decided to adapt the $\star$-product so as to be able to apply it also in the setting of \cite{McM1}.

The paper is organized as follows:

\S 1 is concerned not with polytopes in $\mathbb R^n$ modulo Euclidean symmetry, nor yet with polytopes in $\mathbb R^n$ modulo translational symmetry, but rather with polytopes in an arbitrary piecewise linear space $X$ (not modulo any symmetry). We introduce a group $\mathcal P(X)$, a sort of Grothendieck group generated by compact PL subspaces of $X$. We show that $\mathcal P$ is a functor, and we exhibit several kinds of structure on it, including an external product $\mathcal P(X)\times \mathcal P(Y)\to \mathcal P(X\times Y)$, a filtration by dimension, and an ``interior'' involution $I$. We also introduce a boundary operator that becomes useful in \S 8.

It is not until the end of \S 1 that we specialize to the case of $\mathcal P(V)$ for a (real, finite-dimensional) vector space. The main object of study in \cite{McM1}, here denoted $\Pi(V)$, is the group of coinvariants for the action of translations on $\mathcal P(V)$, while our $E_n$ is the group of coinvariants for the action of all affine isometries on $\mathcal P(\mathbb R^n)$. The translational scissors congruence group of $V$, denoted here by $\hat \Pi(V)$, is the quotient of $\Pi(V)$ by the subgroup generated by polytopes of positive codimension in $V$.

In \S 2 we continue to work with a vector space $V$, introducing cones and germs, and (following the notation of \cite{McM1}) the cone group $\Sigma(V)$. Again $\Sigma$ is a functor, with external products, a dimension filtration, and an interior involution $I$. It also has a duality involution $D$. Our $L_n$ will be the group of coinvariants for the action of linear isometries on $\Sigma(\mathbb R^n)$.

\S 3 introduces the $\star$-product, but in a more general and more complicated form than what is described above. Now we consider functions $F$ which, rather than simply assigning an element $F_n(P)\in k$ to each cone $P\subset \mathbb R^n$, now assign an element $F_V(P)\in k$ to each cone $P\subset V$ where $V$ is any finite-dimensional vector subspace of $\mathbb R^\infty$ (subject to certain additivity and multiplicativity conditions, Definition \ref{morphism} below). The compatibility condition for defining $F\star G$ is $F_V(V)=G_V(0)$ for every $V$. These multiplicative cone functions are the morphisms of a groupoid $\mathcal G_k$. The objects are functions $O$ assigning an element $O_V\in k$ to each $V$ in such a way that orthogonal direct sums go to products. (The groupoid $\mathcal G_k^O$ discussed above is isomorphic to a subcategory of $\mathcal G_k$,  consisting of those objects and morphisms which are invariant under linear isometries.) There is also a product $F\star G$ defined when $F$ is such a multiplicative function of cones and $G$ is a suitable multiplicative function of polytopes: $G$ assigns an element $G_V(P)\in k$ to each polytope $P\subset V$ subject to additivity and multiplicativity conditions. Neither these functions of cones nor these functions of polytopes are required to be invariant under isometry, but the latter are required to be invariant under translation. 

In \S 4 we recall certain constructions of McMullen and use the $\star$-product to give a unified description of them. In particular we give a brief proof of one of the main results of \cite{McM1}: a canonical splitting of $\Pi(V)$ as the direct sum of $\hat \Pi(W)$ over all vector subspaces $W\subset V$. The later parts of this paper do not depend on \S 4.

The brief \S 5 investigates the properties of the relative Euler invariant. It is a morphism from $-1$ to $1$, and it is closely connected with the interior involution $I$.

In \S 6 we come at last to the Euclidean-invariant part of the story. We use the $\star$-product to define the intrinsic volume and various generalizations of the Dehn invariant. 

It is not until \S 7 that we introduce the graded rings $E$ and $L$. We name various elements and invariants, and we obtain the splitting result for $E$ that was mentioned above.

In \S 8 we continue to investigate $E$ and $L$, now with emphasis on elements represented by manifolds without boundary.

The author thanks Inna Zakharevich for sparking his interest in scissors congruence in the first place, Ben Wieland for many stimulating conversations at the beginning of this project, and the anonymous referee for making him acquainted with McMullen's work.

\section{The polytope group of a PL space}

Recall that two locally finite triangulations of a space are said to determine the same PL structure if they have a common refinement, and that a space equipped with such a structure is called a PL space. A map between PL spaces is called a PL map if in a neighborhood of each point in its domain it is simplicial for some triangulations within the PL structures. A subset $Y$ of a PL space $X$ is called a PL subspace if it has a (necessarily unique) PL structure such that the inclusion map $Y\to X$ is PL. An open subset of $X$ is always a PL subspace. A closed subset is a PL subspace if and only if it is a subcomplex for some triangulation within the PL structure of $X$.

To a PL space $X$ we will associate an abelian group $\mathcal P(X)$, a sort of Grothendieck group generated by the compact PL subspaces of $X$. We will show that $\mathcal P$ is a functor, and that it has several kinds of structure: an external multiplication, a filtration by dimension, and an involution $I$ that is related on the one hand to interiors of manifolds and on the other to local Euler characteristic.

This section provides a foundation for the rest of the paper and establishes some notation. Very little here is new. Many proofs could have been omitted by giving references, for example to \cite{McM1} or \cite {McM2}, but instead we have tried to give a streamlined and mostly self-contained treatment.

\subsection{Valuations of polytopes}

The compact PL subspaces of a PL space $X$ will be called the \emph{polytopes} in $X$.

Suppose that $A$ is an abelian group, and that $F$ is a function assigning an element of $A$ to every polytope $P\subset X$. $F$ is a \emph{valuation} if it satisfies
\begin{equation}\label{k2}
F(P\cup Q)=F(P)+F(Q)-F(P\cap Q)
\end{equation}
and
\begin{equation}\label{k0}
F(\emptyset)=0.
\end{equation}
One basic example is the Euler characteristic. Another, when $X=\mathbb R^n$, is the volume.

Equations (\ref{k2}) and (\ref{k0}) are the special cases $k=2$ and $k=0$ of the following more general statement, which they imply: 

If $P_1,\dots ,P_k$ are polytopes in $X$ then
\begin{equation}\label{genk}
F(P_1\cup\dots \cup P_k)=\sum_{S\neq\emptyset}(-1)^{|S|-1}F(\cap_{i\in S}P_i).
\end{equation}
Here $S$ runs through all nonempty subsets of $\lbrace 1,\dots ,k\rbrace$ and $|S|$ is the number of elements of $S$. This statement for $k>2$ follows from (\ref{k2}) by an easy induction with respect to $k$.

\subsection{The group $\mathcal P(X)$}

Let $P\mapsto (P)$ be the universal valuation for polytopes in $X$. It takes values in an abelian group $\mathcal P(X)$, the \emph{polytope group} of $X$. Thus $\mathcal P(X)$ has a presentation in which there is a generator $(P)$ for each polytope $P\subset X$, with defining relations
$$
(P\cup Q)=(P)+(Q)-(P\cap Q)
$$$$
(\emptyset)=0.
$$
It is clear that $\mathcal P(X)$ is generated by the elements $(\sigma)$ where $\sigma$ is a simplex of a triangulation of $X$ (in the given PL structure). There is the following precise formula. Recall that, for a simplex $\sigma$ in a simplicial complex $\mathcal K$, the \emph{star} is the subcomplex $\text{St}(\sigma, \mathcal K)$ consisting of all simplices having $\sigma$ as a face (including $\sigma$ itself) together with the faces of these, and the \emph{link} is the subcomplex $\text{Lk}(\sigma,\mathcal K)$ consisting of those simplices of $\text{St}(\sigma,\mathcal K)$ which are disjoint from $\sigma$. Thus $\text{St}(\sigma,\mathcal K)$ is the join of $\sigma$ and $\text{Lk}(\sigma,\mathcal K)$. 

\begin{prop}
For a triangulated polytope $P\subset X$ we have
\begin{equation}\label{triang}
(P)=\sum_{\sigma}\ (1- \chi \text{Lk}(\sigma,P))(\sigma)\in \mathcal P(X),
\end{equation}
where the sum is over all simplices $\sigma$ of the given triangulation of $P$ and $\chi \text{Lk}(\sigma,P)$ is the Euler characteristic of the link. 
\end{prop}
\begin{proof}This is by induction with respect to the number of simplices of the triangulation of $P$. The essential point is that the equation holds for $P\cup Q$ if it holds for subcomplexes $P$ and $Q$ and for their intersection. This uses the fact that for a simplex $\sigma$ in $P\cap Q$ we have 
$$
\text{Lk}(\sigma,P\cap Q)=\text{Lk}(\sigma,P)\cap \text{Lk}(\sigma, Q)
$$
and 
$$
\text{Lk}(\sigma,P\cup Q)=\text{Lk}(\sigma,P)\cup \text{Lk}(\sigma, Q),
$$
together with the fact that when $\sigma$ is in $P$ but not in $Q$ (resp. in $Q$ but not in $P$) then $\text{Lk}(\sigma,P\cup Q)$ is equal to $\text{Lk}(\sigma,P)$ (resp. $\text{Lk}(\sigma,Q)$). The induction begins with the easy case when $P$ is a simplex triangulated by its faces and the trivial case when $P$ is empty. 
\end{proof}
It can be useful to identify $\mathcal P(X)$ with a subgroup of the group $\text{Fun}(X,\mathbb Z)$ of all integer-valued functions on the point set of $X$. For a subset $Y\subset X$  let $i_Y:X\to\mathbb Z$ be the characteristic function of $Y$, defined by $i_Y(x)=1$ if $x\in Y$ and $i_Y(x)=0$ if $x\notin Y$. Since the mapping $P\mapsto i_P$ is a valuation, it determines a homomorphism
$$
\mathcal P(X)\xrightarrow{i} \text{Fun}(X,\mathbb Z)
$$
$$
(P)\mapsto i_P.
$$
\begin{lemma}\cite{Gr}\label{function}
The homomorphism $i$ is injective. 
\end{lemma}
\begin{proof}Suppose that $\xi\in \mathcal P(X)$ is not zero. For some triangulation of some polytope $P\subset X$ we have $\xi=\sum_\sigma k_\sigma\cdot (\sigma)$, where $\sigma$ runs through the simplices of the triangulation and the coefficients $k_\sigma$ are integers. Choose a simplex $\tau$ of maximal dimension among all those for which $k_\tau\neq 0$. The function $i(\xi)= \sum_\sigma k_\sigma i_\sigma$ takes the value $k_\tau$ at points in the interior of $\tau$ and is therefore different from zero.
\end{proof}
We may consider a broader class of subsets of $X$ than the polytopes:
\begin{defin}
A subset of $X$ is \emph{constructible} if it is the union of the interiors of a finite set of simplices in some triangulation of $X$ (in the given PL structure).
\end{defin}
Note that when a subset of $X$ is a union of simplex interiors for one triangulation then it is also a union of simplex interiors for any finer triangulation. From this it follows that the union or intersection of two constructible sets, or the relative complement of one constructible set in another, is again constructible. We conclude:
\begin{prop}
The class of constructible sets in $X$ is the smallest class that contains the polytopes and is closed under the operations of pairwise union and intersection and relative complement. 
\end{prop}
Clearly the characteristic function $i_Y$ of any constructible set $Y$ belongs to $i\mathcal P(X)$. If $Y$ is constructible, write $(Y)$ for the element of $\mathcal P(X)$ corresponding to $i_Y$. In particular we have $(P-A)=(P)-(A)$ when $P$ and $A$ are polytopes and $A\subset P$. Any valuation of polytopes in $X$ extends uniquely to a valuation of constructible sets in $X$. 
\begin{rem}When the Euler characteristic $\chi$ is extended in this way to constructible sets, $\chi(Y)$ is not what is usually called the Euler characteristic of the space $Y$. For polytopes $A\subset P$, $\chi(P-A)$ is the relative Euler characteristic of the pair $(P,A)$; this can also be described in terms of compactly supported cohomology of $P-A$. 
\end{rem}
\begin{rem}
A constructible set is the relative complement of some pair of polytopes if and only if it is relatively open in its closure.
\end{rem}

The following is clear:
\begin{prop}
The constructible subsets of $X$ are precisely those sets $Y\subset X$ such that $i_Y\in i\mathcal P(X)$. The elements of $i\mathcal P(X)$ are precisely those bounded functions \newline
$f:X\to\mathbb Z$ such that for each integer $k\neq 0$ the level set $f^{-1}(k)$ is constructible.
\end{prop}
Call the elements of $i\mathcal P(X)$ the \emph{constructible functions}.

We will sometimes use (\ref{genk}) in the following form. Suppose that the polytope $P$ contains polytopes $P_1,\dots ,P_k$. Then
\begin{equation}\label{exin}
F(P-\cup_{i=1}^k P_i)=\sum_S\ (-1)^{|S|}F(\cap_{i\in S}P_i).
\end{equation}
Here in contrast to (\ref{genk}) the summation is over \emph{all} subsets of $\lbrace 1,\dots ,k\rbrace$ (with the convention that $\cap_{i\in \emptyset}P_i=P$). Equations (\ref{genk}) and (\ref{exin}) are valid more generally when $P$ and the $P_j$ are constructible sets.
\begin{rem}\label{basis}
For any triangulation $\mathcal K$ in the PL structure of $X$, the subgroup of $\mathcal P(X)$ generated by the subcomplexes of $\mathcal K$ has a $\mathbb Z$-basis given by the simplices of $\mathcal K$. It has another $\mathbb Z$-basis given by the interiors of those simplices.
When $\mathcal L$ is a refinement of $\mathcal K$ then the subgroup of $\mathcal P(X)$ generated by simplices of $\mathcal K$ is a summand of the subgroup generated by simplices of $\mathcal L$. Thus $\mathcal P(X)$ can be expressed as a direct limit of free abelian groups with respect to split injections, and it follows that $\mathcal P(X)$ itself is free abelian (although there is no canonical $\mathbb Z$-basis).
\end{rem}

\subsection{Functoriality}

\begin{prop}
For a PL map $f:X\rightarrow Y$ there is a unique homomorphism \newline
$f_\ast:\mathcal P(X)\rightarrow \mathcal P(Y)$ such that whenever $f$ maps a simplex $\sigma\subset X$ linearly onto a simplex in $Y$ then $f_\ast(\sigma)=(f(\sigma))$. With this rule, $\mathcal P$ is a functor.
\end{prop}
\begin{proof}
For a polytope $P\subset X$, define $f_\ast(P)$ by defining $if_\ast(P)\in \text{Fun}(X,\mathbb Z)$ to be the function whose value at a point $y\in Y$ is the Euler characteristic of $P\cap f^{-1}(y)$. It is clear that this is indeed a constructible function and that the rule associating this function to $P$ is a valuation. This yields a homomorphism $f_\ast$. It is clear that $f_\ast$ behaves as asserted when a simplex is mapped linearly onto a simplex. Certainly if $f$ is an identity map then $f_\ast$ is an identity map. To see that $(g\circ f)_\ast(P)=g_\ast( f_\ast(P))$, first consider the case when $f$ maps a simplex $P$ linearly onto a simplex $f(P)$ and $g$ maps $f(P)$ linearly onto a simplex $g(f(P))$. The statement is clear in this case, and the general case follows easily. The uniqueness is also clear.
\end{proof}

Note that when the PL space $P$ is compact  then a PL map $g:P\to X$ produces an element of $\mathcal P(X)$ by applying the homomorphism $g_\ast:\mathcal P(P)\to \mathcal P(X)$ to the element $(P)\in \mathcal P(P)$. Thus a polytope that is \emph{over} $X$ but not necessarily \emph{in} $X$ still determines an element of $\mathcal P(X)$.

We denote this element by $(g:P\to X)$. Sometimes we simply write $(P\to X)$ if there is no need to name the map from $P$ to $X$. We may even casually write $(P)\in \mathcal P(X)$ if the map $P\to X$ is understood.
For example, if $M\subset X$ is a compact PL manifold then the double $\mathcal DM=M\cup_{\partial M}M$ determines an element $(\mathcal DM)\in \mathcal P(X)$, equal to $2(M)-(\partial M)$.

\subsection{The product}

There is an external multiplication 
$$
\mathcal P(X)\times \mathcal P(Y)\to \mathcal P(X\times Y),
$$
distributive over addition, such that
$$
(P)\times (Q)= (P\times Q).
$$
This follows from the fact that for any polytope $Q\subset Y$ (resp. $P\subset X$) the function $P\mapsto (P\times Q)$ (resp. $Q\mapsto (P\times Q)$) is a valuation of polytopes in $X$ (resp. $Y$).

The product is natural: for maps $f:X\to X'$ and $g:Y\to Y'$ and elements $\xi\in \mathcal P(X)$ and $\eta\in \mathcal P(Y)$ we have
$$
f_\ast(\xi)\times g_\ast(\eta)=(f\times g)_\ast(\xi\times \eta)\in \mathcal P(X'\times Y').
$$ 
The product is commutative and associative in the appropriate senses. For example, for $\xi\in \mathcal P(X)$ and $\eta\in \mathcal P(Y)$ the elements $\xi\times \eta$ and $\eta\times \xi$ are related by the canonical isomorphism $\mathcal P(X\times Y)\cong\mathcal P(Y\times X)$. It is also unital: if $\ast $ is a one-point space and $1$ is the generator $(\ast)\in\mathcal P(\ast)\cong\mathbb Z$, then $\xi\times 1$ is related to $\xi$ by the canonical isomorphism $\mathcal P(X\times \ast)\cong \mathcal P(X)$.

 \subsection{The extension to pairs}\label{pairs}

The map $\mathcal P(A)\to \mathcal P(X)$ induced by inclusion of a PL subspace is injective, and we view it as an inclusion.  The \emph{support} of $\xi\in \mathcal P(X)$ is the closure of the set of all points $x\in X$ such that $(i\xi)(x)\neq 0$. This is a polytope. For a PL subspace $A$ of $X$ and an element $\xi$ of $\mathcal P(X)$, $\xi$ belongs to $\mathcal P(A)$ if and only if the support of $\xi$ is contained in $A$. It follows that
\begin{equation}\label{AB}
\mathcal P(A)\cap \mathcal P(B)=\mathcal P(A\cap B).
\end{equation}

Call $(X,A)$ a \emph{PL pair} if $X$ is a PL space and $A$ is a PL subspace. The pairs are the objects of a category in the usual way, and the category of PL spaces may be identified with the full subcategory given by the pairs $(X,\emptyset)$. The functor $\mathcal P$ extends to pairs by writing
$$
\mathcal P(X,A)=\mathcal P(X)/\mathcal P(A).
$$
The product extends to a product for pairs
$$
\mathcal P(X,A)\times \mathcal P(Y,B)\to \mathcal P(X\times Y,A\times Y\cup X\times B).
$$
\begin{prop}\label{excision}
When $X=A\cup B$ then the canonical map
$$
\mathcal P(B,A\cap B)\to \mathcal P(X,A)
$$
is an isomorphism if the $PL$ subspaces $A$ and $B$ are either both open or both closed.
\end{prop}
\begin{proof}It is always injective, by (\ref{AB}). It is surjective if every polytope in $X$ is the union of a polytope in $A$ and a polytope in $B$.
\end{proof}
The following will be used later in the case when $C$ is a point.
\begin{prop}\label{germ}
If $C\subset X$ is a closed set, then two polytopes $P$ and $Q$ in $X$ determine the same element of $\mathcal P(X,X-C)$ if and only if they have the same germ at $C$, that is, if and only if there is a neighborhood $N$ of $C$ such that $P\cap N=Q\cap N$. 
\end{prop}
\begin{proof}
The condition $(P)-(Q)\in \mathcal P(X-C)$ is equivalent to the statement that the support of $(P)-(Q)$ is disjoint from $C$. Since $C$ is closed, this is equivalent to the equality of the germs. 
\end{proof}

\subsection{The dimension filtration}

Let $\mathcal P_n(X)$ be defined in the same way as $\mathcal P(X)$, but using only polytopes of dimension $\le n$. This injects into $\mathcal P(X)$ because by the proof of Lemma \ref{function} the composed map $\mathcal P_n(X)\to Fun(X,\mathbb Z)$ is an injection. Thus we can identify $\mathcal P_n(X)$ with the subgroup of $\mathcal P(X)$ generated by all elements $(P)$ where $P\subset X$ is a polytope of dimension $\le n$. An element of $\mathcal P(X)$ belongs to $\mathcal P_n(X)$ if and only if the dimension of its support is at most $n$.

$\mathcal P_n$ is a subfunctor of $\mathcal P$. That is, the map $\mathcal P(X)\xrightarrow{f_\ast} \mathcal P(Y)$ induced by $X\xrightarrow{f} Y$ carries $\mathcal P_n(X)$ into $\mathcal P_n(Y)$.

The filtration is compatible with the multiplication and with the unit:
$$
\mathcal P_p(X)\times\mathcal P _q(Y)\subset\mathcal P_{p+q}(X\times Y)
$$
$$
1\in \mathcal P_0(\ast).
$$
The filtration may be extended to a functorial filtration for pairs with $\mathcal P_n(X,A)\cong\mathcal P_n(X)/\mathcal P_n(A)$.

\subsection{The involution $I$}
If $M$ is a PL manifold, we write $\partial M$ for its boundary and $\text{int} M$ for its interior $M-\partial M$. Thus if the polytope $M\subset X$ is a PL manifold of any dimension then $(\text{int}M)=(M)-(\partial M)\in \mathcal P(X)$. We will never refer to the interior of a point set except in the PL manifold sense.
\begin{prop}\label{makeI}
There is an involution $\mathcal P(X)\xrightarrow{I} \mathcal P(X)$ such that if $M\subset X$ is a compact PL manifold of any dimension $d$ then 
\begin{equation}\label{intM}
I(M)=(-1)^d(\text{int}M). 
\end{equation}
Moreover, for any polytope $P$ in $X$ we have
\begin{equation}\label{IVtriang}
I( P)=\sum_{\sigma}\ (-1)^{|\sigma|}(\sigma)
\end{equation}
where $\sigma$ ranges over the simplices of a triangulation of $P$ and $|\sigma|$ is the dimension of $|\sigma|$. 
\end{prop}
\begin{rem}
One could take (\ref{IVtriang}) as the starting point for defining $I$. Alternatively, one could begin by using (\ref{intM}) to specify $I(\sigma)$ for any simplex $\sigma$. In either case one would then need to verify compatibility with refinement (subdivision). It will be slightly easier to use the isomorphism $\mathcal P(X)\cong i\mathcal P(X)$ in defining $I$. 
\end{rem}
\begin{proof}
For any polytope $P\subset X$, define $I(P)$ by letting $iI(P)$ be the constructible function $X\to \mathbb Z$ that takes each point $x\in P$ to the relative Euler characteristic 
$\chi(P,P-x)$ and vanishes at all points outside $P$. This map $P\mapsto iI(P)$ is a valuation, and therefore it yields a group map $I:\mathcal P(X)\to \mathcal P(X)$.

When the polytope $M\subset X$ is a $d$-dimensional manifold, then Equation (\ref{intM}) holds because $iI(M)$ is the characteristic function of the interior of $M$ multiplied by $(-1)^d$.

To see that $I$ is an involution, evaluate $I\circ I$ on a simplex $\sigma$ of any dimension $d$, using (\ref{intM}) three times:
$$
(-1)^dI(I(\sigma))=I(\sigma)-I(\partial \sigma)=(-1)^d((\sigma) - (\partial \sigma))-(-1)^{d-1}(\partial \sigma)=(-1)^d( \sigma).
$$
For (\ref{IVtriang}), first note that
$$
(P)=\sum_\sigma\ ( \text{int}\sigma)=\sum_\sigma\ (-1)^{|\sigma|} I( \sigma)
$$ 
since $P$ is the union of the disjoint sets $\text{int} \sigma$. Then apply $I$ to both sides and use that $I\circ I$ is the identity. (Alternatively, (\ref{IVtriang}) can be proved by induction just like Proposition \ref{triang}.)
\end{proof}
In particular $I$ takes a simplex to the alternating sum of its faces:
\begin{equation}\label{Isimp}
I(\sigma)=\sum_{\tau\subset \sigma}\ (-1)^{|\tau|}(\tau).
\end{equation}
For any triangulation $\mathcal K$ of $X$, the involution $I$ preserves the subgroup of $\mathcal P(X)$ generated by the subcomplexes of $\mathcal K$, and up to sign it interchanges the two bases mentioned in Remark \ref{basis}.
\begin{rem}
In the special case when $X$ is a vector space, McMullen \cite{McM1} uses the term \emph{Euler map} for $I$ and for related involutions of quotient groups of $\mathcal P(X)$, and he denotes $I\xi$ by $\xi^*$. In a slightly different context \cite{Al}, such an involution has been called Euler-Verdier duality.
\end{rem}
\begin{prop}
The map $I$ is natural. It preserves the external product and its unit. It preserves the dimension filtration. 
\end{prop}
\begin{proof}
Let $X\xrightarrow{f} Y$ be a PL map. To prove that $f_\ast(I(P)))=I(f_\ast (P))$ for all polytopes $P\subset X$, it suffices to consider the case when $P$ is a simplex such that $f$ maps $P$ linearly onto a simplex of $Y$ with vertices mapping to vertices. The reader may verify it in this case, perhaps either by considering Euler characteristics of the fibers of $\partial P\to f(P)$ or by using (\ref{Isimp}).

To see that
$$
I(\xi\times \eta)=I(\xi)\times I(\eta)\in \mathcal P(X\times Y)
$$
for $\xi\in \mathcal P(X)$ and $\eta\in \mathcal P(Y)$, it is enough to consider the case when $\xi$ and $\eta$ are given by simplices. Since simplices are manifolds, it holds in this case by (\ref{intM}), as the interior of the product of two manifolds is the product of the interiors. 

The rest is clear.
\end{proof}
Because $I$ is natural, it extends to a natural  involution $\mathcal P(X,A)\xrightarrow{I} \mathcal P(X,A)$ for $PL$ pairs.

\begin{rem}\label{boundary double}
If $M$ is a closed $d$-dimensional manifold then $I(M)=(-1)^d(M)$. If $M$ is any compact $d$-dimensional manifold then $(M)-(-1)^dI(M)=(\partial M)$ and \newline
$(M)+(-1)^dI(M)=(\mathcal D M)$ where $\mathcal DM=M\cup_{\partial M}M$ is the double of $M$. 
\end{rem}
\begin{prop}\label{topI}
The involution of $\mathcal P_n(X)/\mathcal P_{n-1}(X)$ induced by $I$ is multiplication by $(-1)^n$.
\end{prop}
\begin{proof}
$\mathcal P_{n}(X)$ is generated modulo $\mathcal P_{n-1}(X)$ by elements $(\sigma)$ where $\sigma$ is an $n$-simplex. For these we have
$$
(\sigma)-(-1)^nI(\sigma)=(\partial\sigma)\in \mathcal P_{n-1}(X).
$$
\end{proof}
\begin{rem}\label{singularity}
If the polytope $M\subset X$ is such that outside a set of dimension $n$ it is a PL $d$-dimensional manifold without boundary  (that is, if there is an $n$-dimensional polytope $K\subset M$ such that $M-K$ is a $d$-dimensional manifold without boundary), then $(M)-(-1)^dI(M)\in \mathcal P_n(X)$. One way to see this is by comparing Equations (\ref{IVtriang}) and (\ref{triang}) term by term.
\end{rem}

\subsection{The boundary operator}\label{boundaries}

In view of Proposition \ref{topI} we may define a natural map
$$
\delta_n:\mathcal P_n(X)\to \mathcal P_{n-1}(X)
$$
by
$$
\delta_n(\xi)= \xi-(-1)^nI(\xi).
$$
This operator will be useful in \S\ref{int elem}. 

Since $I\circ I=\mathbb I$, we have $(\mathbb I-I)\circ (\mathbb I+I)=0=(\mathbb I+I)\circ(\mathbb I-I)$, and this gives 
\begin{equation}\label{dd}
\delta_n\circ\delta_{n+1}=0,
\end{equation}
\begin{equation}\label{Idelta_n}
I\circ\delta_n=(-1)^{n+1}\delta_n,
\end{equation}
and
\begin{equation}\label{delta_nI}
\delta_n\circ I=(-1)^{n+1}\delta_n.
\end{equation}
If $\xi\in \mathcal P_i(X)$ and $\eta\in \mathcal P_j(Y)$, then
\begin{equation}\label{prodrule}
\delta_{i+j}(\xi\times \eta)=\delta_i(\xi)\times \eta+(-1)^iI(\xi)\times \delta_j(\eta).
\end{equation}
All of this extends routinely to $PL$ pairs. 

By Remark \ref{boundary double}, if $M\to X$ is a PL map and $M$ is a compact manifold (with boundary) of dimension $d\le n+1$, then we have either
\begin{equation}\label{boundary}
\delta_{n+1}(M\to X)=(\partial M\to X),
\end{equation}
if $n-d$ is odd, or
\begin{equation}\label{double}
\delta_{n+1}(M\to X)=(\mathcal DM\to X),
\end{equation}
if $n-d$ is even. Thus in both cases this element of $\mathcal P_n(X)$ is given by a closed manifold over $X$ whose dimension is congruent $\text{mod}\ 2$ to $n$. Since $\mathcal P_{n+1}(X)$ is generated by simplices (and simplices are manifolds) it follows that the subgroup $\delta_{n+1}\mathcal P_{n+1}(X)\subset \mathcal P_{n}(X)$ is generated by (some) closed manifolds over $X$ whose dimensions are at most $n$ and congruent to $n$ $\text{mod}\ 2$.

Let $h_n(X,A)$ be the $n$th homology group of the chain complex 
$(\mathcal P_\bullet(X,A),\delta_\bullet)$. The remainder of this subsection is devoted to investigating this (not very extraordinary) homology theory. We show that it is a direct sum of suspensions of ordinary $\text{mod}\ 2$ homology, we exhibit it as a quotient of unoriented PL bordism, and we determine its multiplicative structure. In doing so we find (Corollary \ref{cycles}) that $ker(\delta_n)$ is generated by \emph{all} closed manifolds over $X$ with dimension at most $n$ and congruent to $n$  $\text{mod}\ 2$; in fact it is generated by those of dimension $n$.

None of the material in \S\ref{boundaries} after equation (\ref{prodrule}) will be used in the rest of the paper, and the impatient reader who is more interested in refinements of scissors congruence may skip to \S\ref{V polytope}. The impatient author has left many steps for the reader to fill in in the following arguments.

Note that $h_n(X,A)$ is killed by $2$, because if $\delta_n(\xi)=0$ then $2\xi=\xi+(-1)^nI(\xi)=\delta_{n+1}(\xi)$. 
\begin{prop}
$h_\bullet$ satisfies the axioms for a multiplicative homology theory on the category of PL pairs. 
\end{prop}
\begin{proof}
Excision holds by Proposition \ref{excision}. There is an external product 
$$
h_i(X,A)\times h_j(Y,B)\to h_{i+j}(X\times Y, A\times Y\cup X\times B),
$$
well defined because of (\ref{prodrule}) and (\ref{delta_nI}). This product is compatible with long exact sequences of pairs. One can establish homotopy invariance by showing that a PL homotopy $X\times \Delta^1\to Y$ gives a suitable chain homotopy; for this, by (\ref{prodrule}) and (\ref{delta_nI}), it is enough to verify that 
there is an element $c\in \mathcal P_1(\Delta^1)$ such that $\delta_1(c)\in \mathcal P_0(\Delta^1)$ is the difference between the endpoints. 
\end{proof}

The coefficient ring $h_\bullet(\ast)$, the homology ring of a point, is a polynomial ring $(\mathbb Z/2\mathbb Z)\lbrack t\rbrack$ with $t$ in degree $2$.

Let $\Omega^{PL}_\bullet$ be unoriented PL bordism theory, another multiplicative homology theory. There is a rather obvious map 
\begin{equation}\label{bord map}
\Omega^{PL}_n(X)\to h_n(X)
\end{equation}
sending the bordism class of $M\to X$ to the homology class of the cycle \newline
$(M\to X)\in \mathcal P_n(X)$. This extends to pairs, it is natural, and it is compatible with product and unit, and with connecting homomorphisms of pairs. In other words, it is a map of multiplicative homology theories.

The map $\Omega^{PL}_\bullet (\ast)\to h_\bullet(\ast)$ of coefficient rings is surjective, since the element of $\Omega^{PL}_2(\ast)$ given by $\mathbb RP^2$ goes to the polynomial generator $t\in h_2(\ast)$. Let $K$ be the kernel of this graded ring map. This graded ideal consists of those bordism classes whose representatives have even Euler characteristic. 

Consider the multiplicative homology theory $(\Omega^{PL}/K)_\bullet$ given by 
$$
(\Omega^{PL}/K)_\bullet(X,A)=(\Omega^{PL}_\bullet(\ast)/K)\otimes_{\Omega^{PL}_\bullet(\ast)}\Omega^{PL}_\bullet(X,A).
$$
The fact that this is a homology theory relies on the fact that $\Omega^{PL}_\bullet(X,A)$ is flat (in fact, free) as a module over $\Omega^{PL}_\bullet(\ast)$. This in turn relies on a theorem of Thom \cite{Thom}. For a relatively elementary treatment of this small part of Thom's work on cobordism, see Proposition \ref{Thom} below.
\begin{prop}\label{bord iso}
The map $(\Omega^{PL}/K)_\bullet \to h_\bullet$ induced by (\ref{bord map}) is an isomorphism of multiplicative homology theories.
\end{prop}
\begin{proof}
It is a map of such theories, and the map of coefficient rings is an isomorphism. Since both theories satisfy the limit axiom (homology of an infinite complex is the direct limit of homology of finite subcomplexes), this suffices.
\end{proof}
We now give another description of $h_n(X,A)$. Adjoin an indeterminate $u$ in degree one to ordinary $\text{mod}\ 2$ homology, obtaining the multiplicative homology theory
$$
H_n(-;(\mathbb Z/2)\lbrack u\rbrack)=H_n(-;\mathbb Z/2)\oplus H_{n-1}(-;\mathbb Z/2)u\oplus H_{n-2}(-;\mathbb Z/2)u^2\oplus \dots .
$$
Let $\beta $ be the $\text{mod}\ 2$ Bockstein operator. Let 
$$
h_n^\beta(X,A)\subset H_n(X,A;(\mathbb Z/2)\lbrack u\rbrack)
$$
be the subgroup consisting of all elements $\sum_j a_ju^j$ such that $\beta(a_{2k})=a_{2k+1}$ for all $k$. This gives a homology theory $h^\beta_\bullet(-)\subset H_\bullet(-;(\mathbb Z/2)\lbrack u\rbrack)$, isomorphic to $\bigoplus_{k\ge 0}H_{\bullet-2k}.$ It is closed under the product and contains the unit, so it is a multiplicative theory. The coefficient ring $h_\bullet^\beta(\ast)$ is $(\mathbb Z/2)\lbrack u^2\rbrack$.
\begin{prop}
$h_\bullet$ is isomorphic to $h^\beta_\bullet$ as a multiplicative homology theory.
\end{prop}
\begin{proof}
It suffices to show that $(\Omega^{PL}/K)_\bullet$ is isomorphic to $h^\beta_\bullet$. For this we can argue as in the proof of Proposition \ref{bord iso} as soon as we have a map of multiplicative homology theories $\Omega^{PL}_\bullet\to h^\beta_\bullet$ such that the map of coefficient rings is surjective with kernel $K$. 

Map $\Omega^{PL}_n(X)$ to $h_n^\beta(X)$ by
$$
(f:M\to X)\mapsto \sum_j f_\ast(w_j(M)\cap \lbrack M\rbrack)u^j,
$$
where $w_j(M)\in H^j(M;\mathbb Z/2)$ is the tangent Stiefel-Whitney class. We must verify the assertion that the right-hand side belongs to $h_\bullet^\beta(X)$. In other words, we must verify that
$$
\beta f_\ast(w_{2k}\cap \lbrack M\rbrack)=f_\ast(w_{2k+1}\cap \lbrack M\rbrack).
$$
This follows from $\beta (w_{2k}\cap \lbrack M\rbrack)=w_{2k+1}\cap \lbrack M\rbrack$, which in turn follows from the standard identities
\begin{equation}\label{w}
\beta (w_{2k})=w_{2k+1}+w_1\cup w_{2k}
\end{equation}
and 
\begin{equation}\label{w1}
\beta\lbrack M\rbrack=w_1(M)\cap \lbrack M\rbrack ,
\end{equation}
since $\beta(a\cap x)=\beta (a)\cap x+a\cap \beta (x)$.

This map $\Omega^{PL}_\bullet\to h^\beta_\bullet$ extends routinely to pairs and gives a map of multiplicative homology theories. In the case of a point it takes the bordism class of the $n$-dimensional manifold $M$ to $\chi(M)u^n$, since the value of $w_n(M)$ on the fundamental class is the $\text{mod}\ 2$ Euler characteristic. Thus the map of coefficient rings is again surjective with kernel $K$. 
\end{proof}
\begin{rem}
We have used the Stiefel-Whitney classes of (the tangent microbundle of) a PL manifold. Let us define Stiefel-Whitney classes for spherical fibrations, therefore in particular for microbundles, and then justify (\ref{w}). If $E\to B$ is a fibration with fibers homotopy equivalent to $S^{d-1}$, then the $\text{mod}\ 2$ Thom class $\varphi\in H^d(E\to B)$ gives the Thom isomorphism
$$
H^j(B)\to H^{j+d}(E\to B).
$$$$
\alpha\mapsto \varphi\cup\alpha
$$
Define $w_j$ by $\varphi\cup w_j=Sq^j(\varphi)$. The identity (\ref{w}) follows from
$$
\varphi\cup w_{2k+1}  =  Sq^{2k+1}\varphi=Sq^1Sq^{2k}\varphi=Sq^1(\varphi\cup w_{2k}) 
 =  Sq^1\varphi\cup w_{2k}+\varphi\cup Sq^1 w_{2k}
 $$$$
 =\varphi\cup w_1\cup w_{2k}+\varphi\cup\beta w_{2k}=\varphi\cup(w_1\cup w_{2k}+\beta w_{2k}).
$$
We leave (\ref{w1}) as an exercise.
\end{rem}
Here is a consequence of Proposition \ref{bord iso}, in which the homology groups $h_n(X)$ are not explicitly mentioned.
\begin{cor}\label{cycles}
The subgroup of $\mathcal P_n(X)$ consisting of all elements $\xi$ that are fixed by the involution $(-1)^nI$ is the same as the subgroup $\mathcal M_n(X)$ generated by all elements $(M\to X)\in \mathcal P_n(X)$ such that $M$ is a closed PL $n$-manifold.
\end{cor}
\begin{proof}The first of these subgroups is the cycle group $ker(\delta_n)$. Equation (\ref{boundary}) implies that $\mathcal M_n(X)\subset ker(\delta_n)$. By Proposition \ref{bord iso} the projection \newline
$\mathcal M_n(X)\to ker(\delta_n)/im(\delta_{n+1})$ is surjective. It remains to show that $\mathcal M_n(X)$ contains $im(\delta_{n+1})$. By (\ref{boundary}) and (\ref{double}), $im(\delta_{n+1})$ is generated by (some) elements $(M\to X)$ where $M$ is a closed manifold of dimension $n-2k\le n$, namely the boundary of an $(n+1-2k)$-simplex or the double of an $(n-2k)$-simplex. Replacing $M$ by $M\times (\mathbb RP^2)^k$ does not change the element $(M\to X)$, since $(\mathbb RP^2)=(\ast)\in \mathcal P(\ast)$, and therefore we may take $n-2k$ to be $n$.
\end{proof}
We now offer a self-contained justification for the step above that relied on Thom's work on smooth unoriented bordism. We use nothing about the mod $2$ Steenrod algebra except the relatively simple fact that any stable operation on $\text{mod}\ 2$ (co)homology is determined by what it does in the case of products of one or more real projective spaces. 

Suppose that $k_\bullet$ is a connective multiplicative homology theory. If $k_0(\ast)=\mathbb Z/2$, then the horizontal edge of a first-quadrant Atiyah-Hirzebruch spectral sequence gives a map $k_n(X,A)\to H_n(X,A;\mathbb Z/2\mathbb Z)$, the \emph{Hurewicz map} for $k_\bullet$. Let $\Omega^O_\bullet$ be smooth unoriented bordism.
\begin{prop}\label{Thom}
Let $k_\bullet$ be any connective multiplicative homology theory such that $k_0(\ast)=\mathbb Z/2$, and such that $k_\bullet$ admits a map $\Omega^O_\bullet\to k_\bullet$. Then for every pair $(X,A)$ the Hurewicz map for $k_\bullet$ is surjective and $k_\bullet(X,A)$ is a free module over $k_\bullet(\ast)$.
\end{prop}
\begin{proof}(Sketch) 
For any given $(X,A)$, surjectivity of the Hurewicz map for all $n$ is equivalent to the statement that the spectral sequence has $E^2_{n,0}=E^\infty_{n,0}$ for all $n$. We prove by induction on $r$ that $E^2_{n,j}=E^r_{n,j}$. Assume that this holds for $r$ (for all pairs) and consider the differential $d_{n,0}^r:E_{n,0}^r\to E_{n-r,r-1}^r$. This map 
$$
d_{\bullet,0}^r:H_{\bullet}(-; k_0(\ast))\to H_{\bullet-r}(-;k_{r-1}(\ast))
$$
is a stable homology operation, so essentially an element of the Steenrod algebra (tensored with the $\mathbb Z/2\mathbb Z$ vector space $k_{r-1}(\ast)$). Therefore it is determined by its behavior on spaces that are products of real projective spaces. For such a space all of the $\text{mod}\ 2$ homology comes from smooth submanifolds, so that the Hurewicz map for $\Omega^O_\bullet$ and therefore also for $k_\bullet$ is surjective. This forces $d_{\bullet,0}^r$ to be zero for these spaces and therefore for all pairs. Multiplicativity implies that more generally $d_{n,j}^r=0$ for all $j$, and thus $E^2_{n,j}=E^{r+1}_{n,j}$.

Since $E^2=E^\infty$, we see that $k_\bullet(X,A)$ is isomorphic  as a module over $k_\bullet(\ast)$ to\newline
$H_\bullet(X,A;\mathbb Z/2\mathbb Z)\otimes_{\mathbb Z/2\mathbb Z}k_\bullet(\ast)$ and is therefore free.
\end{proof}

\subsection{Polytopes in a vector space}\label{V polytope}

We now specialize from PL spaces to vector spaces. Let $V$ be a finite-dimensional real vector space. 

A valuation of polytopes in $V$ is determined by its values on nonempty \emph{convex} polytopes. Furthermore, a function of nonempty convex polytopes in $V$ extends to a valuation of all polytopes in $V$ if and only if it satisfies (\ref{k2}) whenever $P$, $Q$, and $P\cup Q$ are convex. In fact, it is enough if it satisfies it in the special case when a convex polytope is cut in two pieces by a hyperplane:

\begin{lemma}\label{cutlemma}(\cite{McM1} \text{Lemma 2})
A function $F$ of nonempty convex polytopes in $V$ extends uniquely to a valuation of polytopes in $V$ if it satisfies 
\begin{equation}\label{loccut}
F(P)=F(P\cap H^+)+F(P\cap H^-)-F(P\cap H)
\end{equation}
whenever $P\subset V$ is a nonempty convex polytope of any dimension, $H\subset V$ is a codimension one affine subspace intersecting $P$, and $H^+$ and $H^-$ are the two closed half-spaces determined by $H$. 
\end{lemma}

In other words, $\mathcal P(V)$ has an alternative, more economical, presentation using only the nonempty convex polytopes in $V$ as generators and using only the relations
$$
(P)=(P\cap H^+)+(P\cap H^-)-(P\cap H).
$$ 
\begin{rem}\label{d}
$\mathcal P(V)$ is generated by elements $(\sigma)$ where $\sigma\subset V$ is a (linear) simplex of any dimension up to $d=dim\ V$. In fact it is generated by $d$-dimensional simplices, because for $n<d$ any $n$-simplex is the intersection of two $(n+1)$-simplices whose union is an $(n+1)$-simplex.
\end{rem}
Note that the homomorphism $\mathcal P(V)\xrightarrow{f_\ast} \mathcal P(W)$ induced by a linear map 
$V\xrightarrow{f}W$ takes $(P)$ to $(f(P))$ if the polytope $P\subset V$ is convex.

For a non-zero real number $\lambda$, write 
$$
\mathcal P(V)\xrightarrow{\Delta(\lambda)} \mathcal P(V)
$$
for the group automorphism induced by the dilatation $v\mapsto \lambda v$ of $V$.

The group of translations of $V$ acts on $\mathcal P(V)$. The group of coinvariants for this action, denoted $\Pi(V)$, was the central object of study in \cite{McM1}. Our work here will place some constructions of McMullen in a broader context. However, our main focus will be on a smaller quotient $E(V)$, the group of coinvariants for the action of Euclidean isometries.

Both $\Pi(V)$ and $E(V)$ inherit external products, dimension filtrations, and interior involutions $I$ from $\mathcal P(V)$, as well as dilatation operators $\Delta(\lambda)$.

If $V$ is equipped with an inner product then the associated volume gives a homomorphism 
$$
\mathcal P(V)\xrightarrow{\mathcal V} \mathbb R.
$$
If more than one vector space is in play then we will sometimes write $\mathcal V_V$ to avoid ambiguity. 

By Proposition \ref{topI} the volume map satisfies 
\begin{equation}\label{VIpoly}
\mathcal V\circ I=(-1)^{d}\mathcal V
\end{equation}
where $d$ is the dimension of $V$, since $\mathcal V$ vanishes on the subgroup $\mathcal P_{d-1}(V)$ generated by polytopes of positive codimension.

The Euler characteristic map
$$
\mathcal P(V)\xrightarrow{\chi}\mathbb Z.
$$
satisfies 
$$
\chi\circ I=\chi.
$$

\begin{rem}
Although this will not play much of a role here, the group $\mathcal P(V)$ has a ring structure, which is inherited by the quotient $\Pi(V)$ (but not by the smaller quotient $E(V)$). The product is given by 
$$
(P)(Q)=(P+Q)
$$
when the polytopes $P$ and $Q$ are convex. Here $P+Q$ denotes $\lbrace x+y\ |\   x\in P, y\in Q\rbrace$, the \emph{Minkowski sum} of $P$ and $Q$. In other words, this internal product can be described by combining the external product
$$
\mathcal P(V)\times \mathcal P(V)\to \mathcal P(V\times V)
$$
with the map 
$$
\mathcal P(V\times V)\to \mathcal P(V)
$$
induced by vector addition.
\end{rem}

\subsubsection{Cell structures}\label{cells}

We will have occasion to consider cell structures on polytopes in $V$ for which the cells are convex polytopes but not necessarily simplices. Here are the relevant definitions.

A \emph{face} of a nonempty convex polytope $P\subset V$ is the subset of $P$ on which some linear function $\ell:V\to \mathbb R$ is maximized. In particular $P$ itself is a face. Zero-dimensional faces are called \emph{extreme points}. 

A \emph{cell complex} in $V$ is a finite set $\mathcal K$ of nonempty convex polytopes (the \emph{cells}) such that every face of every cell is the union of a set of cells and such that no point belongs to the interior of more than one cell. The union of the cells is then a polytope $P$, and we call $\mathcal K$ a \emph{cell structure} on $P$. When $\mathcal K$ and $\mathcal K'$ are two cell structures on $P$, we call $\mathcal K'$ a \emph{refinement} of $\mathcal K$ if each cell of $\mathcal K'$ is contained in some cell of $\mathcal K$. Any two cell structures on $P$ have a common refinement. A nonempty convex polytope has a cell structure in which the faces are the cells. 

A simplicial complex is a cell complex in which every cell is a simplex and every face of a cell is a cell. The resulting cell structure is a triangulation. Every cell structure has a refinement that is a triangulation.

\section{Cones}

We now consider PL cones in a vector space $V$. These arise in connection with polytopes in $V$, because they correspond to (translation classes of) germs of polytopes at points. There is a group analogous to $\mathcal P(V)$, defined using cones instead of polytopes. We follow McMullen in calling it $\Sigma(V)$. In addition to some features analogous to those of $\mathcal P(V)$ it has some features of its own, notably a duality involution. 

We adopt the convention that each vector space is equipped with an inner product. If $V$ and $W$ are two vector spaces then $V\times W$ is the product vector space equipped with the inner product determined by those of $V$ and $W$. If $V$ and $W$ occur as vector subspaces of a larger vector space $U$, then the vector space spanned by $V$ and $W$ will be denoted by $V\times W$ if (and only if) $V$ and $W$ are orthogonal in $U$. We extend this notation by writing $P\times Q$ for a subset of $V\times W$ determined by subsets $P\subset V$ and $Q\subset W$ of orthogonal vector spaces.

\subsection{Convex cones}

Call a set $P\subset V$ a \emph{convex cone} if, for some $k\ge 0$ and vectors $v_1,\dots ,v_k$, $P$ is the set of all linear combinations $\sum_{j=1}^kt_jv_j$ with $t_j\ge 0$. Convex cones may also be characterized as those sets of vectors definable by a finite set of conditions of the form $\ell (v)\ge 0$ with $\ell:V\to \mathbb R$ a linear map.

If the generating vectors $v_1,\dots ,v_k$ can be chosen to be linearly independent, then we call $P$ a \emph{conical simplex}. If $P$ contains no one-dimensional vector subspace of $V$ then we call $P$ an \emph{ordinary convex cone}. In general for a convex cone $P$ in $V$ there is a largest vector subspace $W$ contained in $P$, and $P$ is $W\times Q$ for some ordinary convex cone $Q$ in the orthogonal complement $W^\perp$. 

A subset of a convex cone $P$ is a \emph{face} of $P$ if it is the subset of $P$ where some linear function $\ell:V\to \mathbb R$ is maximized. Every face is again a convex cone. The largest face of $P$ is $P$, and the smallest is the largest vector space contained in $P$. 

Every convex cone is a PL manifold. Its boundary is nonempty except when the cone is a vector space.

\subsection{Duals of convex cones}(Except for minor notational differences, most of this can be found in \cite{McM1}.) For a subset $P\subset V$, define its \emph{dual}, $DP$, to consist of those vectors $w\in V$ such that for every $v\in P$ we have $\langle w,v\rangle\le 0$. If $P$ is a convex cone then $DP$ is again a convex cone and $DDP$ is $P$. In this case a point $w$ belongs to $DP$ if and only if, among all the points in $P$, the one nearest to $w$ is $0$.

We sometimes write $D_VP$ instead of $DP$ to avoid ambiguity. 

When $0\in P\subset V$ and $0\in Q\subset W$ with $V$ orthogonal to $W$, then
\begin{equation}\label{Dprod}
D_{V\times W}(P\times Q)=D_VP\times D_WQ.
\end{equation}

The dual cone of a vector subspace of $V$ is its orthogonal complement. In particular $D_VV=0$ and $D_V0=V$. If $V$ has dimension $d$ then the duals of the $d$-dimensional convex cones are the ordinary convex cones. In particular the dual of a $d$-dimensional ordinary convex cone is again a $d$-dimensional ordinary convex cone. The dual of a conical $d$-simplex (in a $d$-dimensional vector space) is a conical $d$-simplex.

If $P$ and $Q$ are two convex cones then $P\cap Q$ is again a convex cone. 

\begin{prop}\label{Dadd}
If $P$, $Q$, and $P\cup Q$ are convex cones, then $DP\cup DQ$ is also a convex cone and we have
\begin{equation}\label{Dcup}
D(P\cup Q)=DP\cap DQ
\end{equation}
\begin{equation}\label{Dcap}
D(P\cap Q)=DP\cup DQ.
\end{equation}
\end{prop}
\begin{proof}
Certainly (\ref{Dcup}) holds in general, as does the inclusion $DP\cup DQ\subset D(P\cap Q)$. If $v\notin DP\cup DQ$ then there exist $w_0\in P$ and $w_1\in Q$ such that $\langle v,w_0\rangle>0$ and $\langle v,w_1\rangle>0$. Since $P\cup Q$ is convex and $P$ and $Q$ are closed, there exists $w\in P\cap Q$ on the line segment joining $w_0$ and $w_1$. Then $\langle v,w\rangle>0$, so that $v\notin D(P\cap Q)$. This shows that (\ref{Dcap}) holds. It follows that $DP\cup DQ$ is a convex cone.
\end{proof}
 
\subsection{Cones}\label{cone cells}

Call a subset of $V$ a \emph{cone} if it is the union of finitely many convex cones. (With this convention the empty set is a cone, but not a convex cone.)

A \emph{conical cell complex} in $V$ is a finite set of convex cones (the \emph{cells}) such that every face of every cell is the union of a set of cells and no point belongs to the interior of more than one cell. The union of all the cells is then a cone, and we call the complex a \emph{conical cell structure} on that cone. When $\mathcal K$ and $\mathcal K'$ are two conical cell structures on $P$, we call $\mathcal K'$ a \emph{refinement} of $\mathcal K$ if each cell of $\mathcal K'$ is contained in some cell of $\mathcal K$. Any two conical cell structures on $P$ have a common refinement. A convex cone has a cell structure in which the cells are the faces. A \emph{conical simplicial complex} is a conical cell complex in which every cell is a conical simplex and every face of a cell is a cell. The resulting conical cell structure is called a \emph{conical triangulation}. Every conical cell structure has a refinement that is a conical triangulation.

\subsection{The local viewpoint}\label{local}

Say that subsets $A,B\subset V$ have the same \emph{germ} at $0$ if there is a neighborhood $N$ of $0$ such that $A\cap N=B\cap N$. The germ of any polytope at $0$ is the germ of a unique cone, and every cone arises in this way. Thus we could, if we preferred, work entirely with germs of polytopes at $0$ (or germs of closed PL subspaces of $V$ at $0$) instead of cones. In fact, the main reason for considering cones is that they correspond to germs. 

We may also consider the germ of a polytope $P\subset V$ at $v\in V$ if $v\neq 0$. Again this determines a cone in $V$, namely the cone corresponding to the germ at $0$ of the polytope obtained by translating $P$ by the vector $-v$. Let us call this the \emph{normal cone} of $v$ in $P$ and denote it by $\nu(v,P)$.

We may speak of the germ of the polytope $P$ at $v$ even in the case when $v\notin P$. In this case it is the empty germ, the germ of the empty set, and corresponds to the empty cone.

The convex cones are those cones which occur as $\nu(v,P)$ when $P$ is convex and $v\in P$. The ordinary convex cones are those which occur as $\nu(v,P)$ when $P$ is convex and $v$ is an extreme point of $P$. The conical simplices are those which occur as $\nu(v,P)$ when $P$ is a simplex and $v$ is one of its vertices.

When a polytope $P\subset V$ is given a cell structure in the sense of \S\ref{cells}, then the cone $\nu(v,P)$ acquires a conical cell structure whose conical cells are the convex cones $\nu(v,\sigma)$ corresponding to the germs at $v$ of those cells $\sigma$ which contain $v$. If the cell structure is a triangulation of $P$ and $v$ is a vertex, then the conical cell structure is a conical triangulation.

\subsection{Normal cones of cells}\label{normal cones}
We extend the notion of normal cone, defining $\nu(\sigma,P)$ whenever $\sigma $ is a nonempty convex polytope contained in $P$, for example a cell of some cell structure of $P$. Choose a point $v$ in the interior of $\sigma$. Let $W$ be the affine span of $\sigma$ and let the affine space $W^\perp$ be the orthogonal complement of $W$ at $v$. Define $\nu(\sigma,P)=\nu(v,P\cap W^\perp)$. This is independent of which point $v\in \text{int}\ \sigma$ is chosen. Note that $\nu(v,P)$ is the orthogonal product of $\nu(\sigma,P)$ and (the vector space parallel to) $W$. When $\sigma=v$ is a point, then $\nu(\sigma,P)$ is the cone $\nu(v,P)$ that was defined in \S\ref{local}.

If the polytope $P$ is convex then each $\nu(\sigma,P)$ is a convex cone. The normal cone of a face in a convex polytope is an ordinary convex cone.

\begin{rem}
For a face $F$ of a convex polytope $P$, McMullen refers to the cone denoted here by $\nu(F,P)$ as the \emph{inner cone} of the face. Its dual $D\nu(F,P)$ is his \emph{outer cone} $N(F,P)$.
\end{rem}

In the same way we may speak of $\nu(\sigma,P)$ when $P$ is a PL subspace of $V$ but not necessarily closed or bounded, and when $\sigma\subset P$ is convex but not necessarily closed or bounded. In particular when $P$ is a cone and $\sigma$ is a cell of a conical cell structure of $P$ then we define the normal cone $\nu(\sigma,P)$. 

If $\sigma$ and $\tau$ are convex sets in $P$ with $\sigma\subset \tau$, then $\nu(\sigma,\tau)$ is a convex set in $\nu(\sigma,P)$ and its normal cone $\nu(\nu(\sigma,\tau),\nu(\sigma,P))$ may be identified with $\nu(\tau,P)$.

\subsection{The spherical viewpoint}

A nonempty cone $P$ is determined by its intersection $P\cap S(V)$ with the unit sphere. Thus the nonempty cones correspond to certain subsets of $S(V)$, which will be called \emph{spherical polytopes}, and among these the convex cones correspond to what will be called \emph{spherical convex polytopes}. Call a set $A\subset S(V)$ \emph{spherically convex} if whenever it contains non-antipodal points $a$ and $b$ then it contains the shorter of the two great circle arcs connecting $a$ and $b$, and let the \emph{spherical convex hull} of a set in $S(V)$ be the smallest spherically convex set that contains it. The spherical convex polytopes are the spherical convex hulls of the finite subsets of $S(V)$. An ordinary convex cone corresponds to the spherical convex hull of a finite subset of the sphere that is contained in some open hemisphere. A set is a spherical polytope if it is the union of finitely many spherical convex polytopes.

A conical simplex corresponds to a spherical simplex, and a conical triangulation of a cone yields a triangulation of the corresponding subset of the sphere. A spherical polytope does not have a canonical PL structure (because spherical simplices cannot be systematically parametrized by linear simplices in a way that is compatible with subdivision), but it does have such a structure defined up to PL homeomorphism. With this structure, a spherical \emph{convex} polytope is always a PL disk of some dimension, except when it is the unit sphere $S(W)$ of some vector subspace $W\subset V$.

\subsection{The cone group $\Sigma(V)$}

This is defined exactly like $\mathcal P(V)$, but using cones instead of polytopes. We consider valuations of cones in $V$, meaning functions satisfying (\ref{k2}) and (\ref{k0}). We write $P\mapsto (P)\in \Sigma(V)$ for the universal example of such a valuation.

There is an alternative presentation of $\Sigma(V)$ using only \emph{convex} cones as generators and only a hyperplane cutting relation. That is, there is the following analogue of Lemma \ref{cutlemma}.
\begin{lemma}\label{cone cutlemma}
A function $F$ of convex cones in $V$ extends uniquely to a valuation of cones in $V$ if it satisfies 
\begin{equation}\label{cone loccut}
F(P)=F(P\cap H^+)+F(P\cap H^-)-F(P\cap H)
\end{equation}
whenever $P$ is a convex cone, $H\subset V$ is a codimension one vector subspace intersecting $P$, and $H^+$ and $H^-$ are the two closed half-spaces determined by $H$.
\end{lemma}
In analogy with Remark \ref{d} we have:
\begin{rem}\label{0d}
$\Sigma(V)$ is generated by elements $(\sigma)$ where $\sigma\subset V$ is a conical $n$-simplex for some $n\le d=\text{dim}\ V$. In fact it is generated by conical $d$-simplices and the conical $0$-simplex, because for $0<n<d$ any conical $n$-simplex is the intersection of two conical $(n+1)$-simplices whose union is a conical $(n+1)$-simplex. (This argument fails for $n=0$.)
\end{rem}
There is a conical analogue of Equation (\ref{triang}): for a cone $P$ in $V$,
\begin{equation}\label{conetriang}
(P)=\sum_{\sigma}\ (1- \chi Lk(\sigma,P))(\sigma)\in \Sigma(X),
\end{equation}
where the sum is now over all conical simplices $\sigma$ of a conical triangulation (not forgetting the conical $0$-simplex). 

The group $\Sigma(V)$ may be described as a subgroup of $Fun(V,\mathbb Z)$ by identifying a cone with its characteristic function. We may extend the definition of $(P)$ to conical constructible sets, sets obtainable from cones by finite set operations. We omit the details.
\begin{rem}
When taking the spherical point of view, one must bear in mind that every ``spherical constructible set'' corresponds to not one but two conical constructible sets, one containing the origin and one not.
\end{rem}
There is the option of identifying $\Sigma(V)$ with $\mathcal P(V,V-0)$, using Proposition \ref{germ} and the bijection between cones and germs of polytopes at $0$. 
(Note, however, that when a linear map $f:V\to W$ is not injective then the map $f_\ast:\Sigma(V)\to \Sigma(W)$ is not induced by a map of pairs $(V,V-0)\to (W,W-0)$.)

$\Sigma(V)$ has a dimension filtration 
$$
\Sigma_0(V)\subset \Sigma_1(V)\subset \Sigma_2(V)\subset \dots \subset\Sigma_d(V)=\Sigma(V)
$$
where $d$ is the dimension of $V$. Let $\hat\Sigma (V)=\Sigma_d(V)/\Sigma_{d-1}(V)$ be the top quotient in the filtration.

$\Sigma$ is functorial with respect to linear maps, and $\Sigma_n$ is a subfunctor.

We will not find it particularly useful here to view $\Sigma(V)$ as a ring, but we will make much use of the external product 
$$
\Sigma(V)\times \Sigma(W)\to \Sigma(V\times W)
$$
given by 
$$
(P)\times (Q)=(P\times Q).
$$
This is natural, and it respects the dimension filtration.

If $\Sigma(V)$ is identified with $\mathcal P(V,V-0)$ then the subgroup $\Sigma_n(V)$ is $\mathcal P_n(V,V-0)$. The external product for cones is an instance of the product for pairs as defined in \S\ref{pairs}:
$$
\mathcal P(V,V-0)\times \mathcal P(W,W-0)\to \mathcal P(V\times W,V\times W-(0,0)).
$$
From the spherical point of view the product is given by a join operation.

\begin{rem}
For each $n$-dimensional vector subspace $W\subset V$ there is an obvious map $\hat \Sigma(W)\to \Sigma_n(V)/\Sigma_{n-1}(V)$. The resulting map
$$
\bigoplus_{\text{dim}(W)=n} \hat \Sigma(W)\to \Sigma_n(V)/\Sigma_{n-1}(V)
$$
is rather clearly an isomorphism. The dimension filtration of $\Sigma(V)$ must split because these groups are all free abelian, so that there is an isomorphism (non-canonical, and therefore not all that interesting)
$$
\Sigma(V)\cong \bigoplus_{W\subset V}\hat \Sigma(W).
$$
\end{rem}

\subsubsection{Some maps out of $\Sigma(V)$}
There is a group homomorphism 
$$
\Sigma(V)\xrightarrow{\mathcal U} \mathbb R,
$$
the \emph{conical volume invariant}, defined as follows. For a cone $P\subset V$ let $\mathcal U(P)$ be the spherical volume of the intersection $P\cap S(V)$, normalized so that $S(V)$ has volume $1$. (In the special case when $V$ is zero-dimensional and the sphere $S(V)$ is empty a special convention is needed; we agree that in this case $\mathcal U(V)=1$ and $\mathcal U(\emptyset)=0$, just as in every other case.) This function is a valuation and so defines a homomorphism as above. Of course $\mathcal U$ factors through the projection $\Sigma(V)\to \hat\Sigma(V)$. Conical volume is multiplicative: we have
\begin{equation}\label{Vmult}
\mathcal U(P\times Q)=\mathcal U(P)\mathcal U(Q)
\end{equation}
for cones $P\subset V$ and $Q\subset W$ (with $V\perp W$).

There is a modest but useful group homomorphism 
$$
\Sigma(V)\xrightarrow{\epsilon} \mathbb Z
$$
taking every convex cone to $1$. More generally it takes every nonempty cone to $1$. We call it the \emph{yes/no invariant} because it takes a conical constructible set $P$ to $1$ or $0$ according to whether $0\in P$ or $0\notin P$.

There is a somewhat more interesting homomorphism 
$$
\Sigma(V)\xrightarrow{e} \mathbb Z,$$
the \emph{local Euler invariant}. It takes a nonempty cone $P$ to 
$$
e(P):=\chi(P,P-0)=1-\chi(P-0)=1-\chi(P\cap S(V)).
$$
Thus if $P$ is a convex cone then $e(P)=0$ unless $P$ is a vector subspace $W\subset V$, in which case $e(W)=(-1)^{\text{dim}(W)}$.

\subsubsection{The duality involution}

As a consequence of Proposition \ref{Dadd} there is a group homomorphism $D:\Sigma(V)\to\Sigma(V)$ such that whenever $P$ is a convex cone we have $D(P)=(DP)$. Clearly $D\circ D=\mathbb I$. By (\ref{Dprod}) the involution $D$ is compatible with the external product on $\Sigma$.

The yes/no invariant satisfies
$$
\epsilon\circ D=\epsilon.
$$
The local Euler invariant satisfies
\begin{equation}\label{eD}
e(D\xi)=(-1)^de(\xi),
\end{equation}
where $d$ is the dimension of $V$. To see this, it suffices to consider the case where $\xi=(P)$ for a convex cone $P\subset V$, in other words to show that for every such $P$ we have
$$
e(D_VP)=(-1)^{\text{dim}(V)}e(P).
$$
If $P$ is not a vector subspace of $V$ then $e(P)=0$ and $e(D_VP)=0$. For a vector subspace $W$ we have $e(W)=(-1)^{\text{dim}(W)}$ and $e(D_VW)=e(W^\perp)=(-1)^{\text{dim}(V)-\text{dim}(W)}$.

Composition of conical volume with duality gives an important homomorphism, the \emph{dual volume invariant},
$$
\Sigma(V)\xrightarrow{\mathcal W=\mathcal U\circ D} \mathbb R.
$$
By (\ref{Dprod}) and (\ref{Vmult}) it is multiplicative:
$$
\mathcal W(P\times Q)=\mathcal W(P)\mathcal W(Q).
$$
For the trivial subspace $0\subset V$ we have 
$$
\mathcal W(0)=\mathcal U(D0)=\mathcal U(V)=1.
$$
It follows that the number $\mathcal W(P)$ does not depend on the ambient vector space: if $P$ is a cone in $V$ then for any $W$ the cone $P\times 0\subset V\times W$ satisfies
$$
\mathcal W(P\times 0)=\mathcal W(P)\times \mathcal W(0)=\mathcal W(P).
$$
\begin{example}
For the conical $0$-simplex (the origin), $\mathcal W(0)=1$. For a conical $1$-simplex $\sigma^1$ (a half-line), $\mathcal W(\sigma^1)=\frac{1}{2}$. For a conical $2$-simplex $\sigma^2_\theta$, a sector with angle $\theta$, $\mathcal W(\sigma^2_\theta)=\frac{\pi-\theta}{2\pi}$. 
\end{example}
\begin{rem}
It should be emphasized that, whereas the dual volume of a \emph{convex} cone $P$ is defined geometrically as the conical volume of $DP$, the extension to general cones is an algebraic process; for a general cone $P$ the dual volume is not simply the conical volume of some geometric object associated with $P$.
\end{rem}
\begin{example}For a cone $M$ that is a $2$-manifold without boundary (corresponding therefore to the germ of a PL surface at an interior point), $\mathcal W(M)=1-\frac{\theta}{2\pi}$ where $\theta$ is the total angle around the point.
\end{example}

\subsubsection{The involution $I$}\label{interior cone}

There is an involution of $\Sigma(V)$ analogous to the interior involution of $\mathcal P(V)$, and again denoted by $I$. This may be defined either by using the isomorphism $\Sigma(V)\cong \mathcal P(V,V-0)$ or by adapting the proof of Proposition \ref{makeI} to cones. It respects external products and the dimension filtration. There is the following analogue of (\ref{IVtriang}):
For any cone $P$ in $V$,
\begin{equation}\label{IVconetriang}
I( P)=\sum_{\sigma}\ (-1)^{|\sigma|}(\sigma).
\end{equation}
This looks like (\ref{IVtriang}), but now $\sigma$ ranges over all of the conical simplices of a conical triangulation of $P$ (including the $0$-dimensional conical simplex if $P\neq\emptyset$). We omit the proof, which is essentially the same as that of (\ref{IVtriang}).

For any $\xi\in \Sigma(V)$
\begin{equation}\label{VIcone}
\mathcal U(I(\xi))=(-1)^{d}\mathcal U(\xi),
\end{equation}
where $d$ is the dimension of $V$. This holds because  $\mathcal U$ vanishes on cones of positive codimension and $I$ acts like $(-1)^{d}$ on the quotient $\hat \Sigma(V)$.

Let us work out the conical analogue of (\ref{intM}). That is, let us work out the relationship between the map $I$ and the interiors of manifolds in the conical (or local, or spherical) setting.

The nonempty cones $P$ that are $m$-dimensional PL manifolds are precisely those such that $P$ has the same germ at $0$ as some polytope that is itself an $m$-dimensional PL manifold. Such cones are of two distinct kinds. If $\partial P=\emptyset$ then the germ of $P$ at $0$ is the germ of a manifold at an interior point and $P\cap S(V)$ is an $(m-1)$-dimensional PL sphere. If $\partial P\neq\emptyset$ then the germ of $P$ is the germ of a manifold at a boundary point and $P\cap S(V)$ is an $(m-1)$-dimensional PL disk. In the first case 
$$
I(P)=(-1)^m(\text{int} P)=(-1)^m(P).
$$
In the second case 
$$
I(P)=(-1)^m(\text{int} P)=(-1)^m(P)+(-1)^{m-1}(\partial P).
$$

The yes/no invariant and the local Euler invariant are related by 
$$
\epsilon\circ I=e,
$$
or equivalently
$$
e\circ I=\epsilon.
$$
We verify this by evaluating both sides of the first equation on a $d$-dimensional convex cone $P$. If $P$ is a vector space then its boundary is empty and we have 
$$
\epsilon(I(P))=(-1)^d\epsilon(P)=(-1)^d=e(P).
$$
If $P$ is not a vector space then its boundary is nonempty and we have
$$
\epsilon(I(P))=(-1)^d(\epsilon(P)-\epsilon(\partial P))=(-1)^d(1-1)=0=e(P).
$$
\begin{rem}\label{singular}
We can also look at the case when the cone $P$ is not necessarily a manifold but $P-0$ is an $m$-dimensional manifold. This corresponds to the germ of a ``singular manifold'' at an isolated interior singularity. In spherical terms, it means that $M=P\cap S(V)$ is a closed manifold of dimension $m-1$. A spherical triangulation of $M$ corresponds to a conical triangulation of $P$. Each conical simplex $\sigma$ of $P$ except the $0$-simplex corresponds to a spherical simplex $\tau=\sigma\cap S(V)$ of $M$ of one dimension less. Rewriting (\ref{conetriang}) and (\ref{IVconetriang}) in terms of the spherical triangulation, we have
$$
(P)=(1-\chi(M))(0)+\sum_\tau\ (1-\chi \text{Lk}(\tau,M))(\sigma)
$$
$$
I(P)=(0)+\sum_\tau\ (-1)^{1+ \text{dim}(\tau)}(\sigma).
$$
Because $M$ is a closed $(m-1)$-dimensional manifold, $1-\chi \text{Lk}(\tau,M)=(-1)^{m-1-\text{dim}(\tau)}$ and we obtain
$$
(P)-(-1)^mI(P)=(1-\chi(M)-(-1)^m)(0)=(\chi(S^{m-1})-\chi(M))(0).
$$
\end{rem}

\subsubsection{The duality/interior relation}\label{DIDIsection}

Let $\Delta(-1):\Sigma(V)\to\Sigma(V)$ be the action of the $-1$ dilatation of $V$. (Note that the effect of the dilatation $v\mapsto \lambda v$ on $\Sigma(V)$ depends only on the sign of $\lambda$.)  

\begin{lemma}\label{DI}
The operators $I$, $D$, and $\Delta(-1)$ on $\Sigma(V)$ are related by 
$$
D\circ I=(-1)^d\Delta(-1)\circ I\circ D,
$$
where $d$ is the dimension of $V$.
\end{lemma}
\begin{proof}
By Remark \ref{0d} it suffices if the two sides agree on the element $(\sigma)\in \Sigma(V)$ for every conical $d$-simplex $\sigma\subset V$ and also on $(0)$. 

We have 
$$
DI(0)=D(0)=(V)
$$
and 
$$
(-1)^d\Delta(-1)ID(0)=(-1)^d\Delta(-1)I(V)=\Delta(-1)(V)=(V).
$$

Now let $\sigma$ be a conical $d$-simplex. The set $\sigma$ consists of the nonnegative linear combinations of some basis $v_1,\dots ,v_d$ of $V$, and $\text{int}\sigma$ consists of the linear combinations with strictly positive coefficients. For $S\subset \lbrace 1,\dots ,d\rbrace$ let $\sigma_S$ be the face generated by the $v_i$ with $i\in S$. Because $\partial \sigma$ is the union of the $(n-1)$-dimensional faces, and because the faces satisfy $\sigma_{S\cap T}=\sigma_S\cap \sigma_T$, the inclusion-exclusion formula (\ref{exin}) gives
\begin{equation}\label{formulaint}
( \text{int}\sigma)=( \sigma -\partial\sigma)=\sum_{S\subset \lbrace 1,\dots ,n\rbrace}(-1)^{|S|}( \sigma_S),
\end{equation}
whence
\begin{equation}\label{Dformulaint}
D( \text{int}\sigma)=\sum_{S}\ (-1)^{|S|}( D\sigma_S).
\end{equation}
The set $D\sigma_S$ consists of all those vectors $v\in V$ such that $\langle v,v_i\rangle\le 0$ for all $i\in S$. Note that $D\sigma_{S\cup T}=D\sigma_S\cap D\sigma_T$.
Let $X$ be the union of $D\sigma_S$ over all nonempty $S$. Inclusion-exclusion gives
$$
(V-X)=\sum_S\ (-1)^{d-|S|}(D\sigma_S),
$$
and hence by (\ref{Dformulaint})
$$
D( \text{int}\sigma)=(-1)^d( V-X).
$$
But $V-X$, the set of all $v$ such that $\langle v,v_i\rangle>0$ for all $i$, is related to $\text{int}D\sigma$ by $\Delta(-1)$.
\end{proof}
Using this result and (\ref{VIcone}) we find that
$$
\mathcal W (I\xi)=\mathcal U(DI\xi)=(-1)^d\mathcal U(ID\xi)=\mathcal U(D\xi)=\mathcal W(\xi)
$$
for every $\xi\in \Sigma(V)$, or
\begin{equation}
\mathcal W\circ I=\mathcal W.
\end{equation}
This in turn implies:
\begin{cor}\label{oddbend}
When the cone $P$ is an odd-dimensional manifold without boundary then $\mathcal W(P)=0$.
\end{cor}\label{}
\begin{proof}
$\mathcal W(P)=\mathcal W(I(P))=\mathcal W(-(P))=-\mathcal W(P)$.
\end{proof}
\begin{rem}
This last argument can be adapted to the case of an isolated singularity; in the setting of Remark \ref{singular}, with $m$ odd, we find
$$
\mathcal W(P)=1-\frac{\chi(M)}{2}.
$$
\end{rem}


\section{The $\star$-product}\label{star}

We continue to study valuations of polytopes or cones in a vector space $V$, but now the valuations take values in a ring and are required to satisfy a multiplicativity condition. In order to state the condition, we place all of our vector spaces in $\mathbb R^\infty$, a real vector space of countably infinite dimension equipped with a positive definite inner product. We will show that the set of all multiplicative functions of cones (with values in a given ring $k$) carries an extra structure that can be neatly expressed by saying that these functions form (the morphism set of) a groupoid $\mathcal G_k$. The objects of the groupoid are multiplicative functions of vector spaces. The same kind of rule that gives the composition law in the groupoid also gives an ``action'' of the groupoid on the set of all translation-invariant multiplicative functions of polytopes. (Later on we will specialize to the case of Euclidean valuations: valuations that are invariant under isometry. These form a smaller groupoid that can be described more simply, without using $\mathbb R^\infty$.)

\subsection{Multiplicative functions of vector spaces, cones, and polytopes}

Here $V$ or $W$ always denotes a finite-dimensional vector subspace of $\mathbb R^\infty$, and $V\times W$ always means the subspace spanned by orthogonal subspaces $V$ and $W$.

For a commutative ring $k$ we will define the set $\mathcal O_k$ of \emph{objects}, which are multiplicative functions of vector spaces; the set $\mathcal G_k$ of \emph{morphisms}, which are multiplicative families of valuations of cones in vector spaces; and the set $\mathcal I_k$ of \emph{multiplicative polytope invariants}, which are multiplicative families of translation-invariant valuations of polytopes in vector spaces. 

\subsubsection{Objects}
\begin{defin}
A $k$-valued \emph{object} is a function $O$ assigning an element $O_V\in k$ to each finite-dimensional vector subspace $V$ of $\mathbb R^\infty$ in such a way that \begin{itemize}
\item $O_{V\times W}=O_VO_W$ when $V\perp W$
\item $O_0=1$.
\end{itemize} 
\end{defin}
The set $\mathcal O_k$ of all such functions is an abelian monoid: it has a commutative, associative, and unital multiplication given by
$$
(OP)_V=O_VP_V.
$$

The most prominent objects are those which depend only on the dimension of the vector space:
\begin{example}\label{scalar object}
Each element $a\in k$ determines a $k$-valued object
$$
V\mapsto a_V:=a^{\text{dim}(V)},
$$
which will simply be denoted by $a$. We call these the \emph{scalar objects}.
\end{example}
Note that $(ab)_V=a_Vb_V$ and $1_V=1$, so that the multiplicative monoid of $k$ is embedded as a submonoid of $\mathcal O_k$. The object called $0$ satisfies $0_V=0$ for every vector subspace of positive dimension, but $0_0=1$. We have $0O=0$ for every object $O$.

\begin{example}\label{space object}
Each vector subspace $U\subset \mathbb R^\infty$ (not necessarily finite-dimensional) determines a $\mathbb Z$-valued object $O^U$, defined by $O^U_V=1$ if $V\subset U$ and otherwise $O^U_V=0$. These satisfy $O^UO^{U'}=O^{U\cap U'}$, $O^{\mathbb R^\infty}=1$, and $O^0=0$.
\end{example}
\begin{rem}
A $\mathbb Z$-valued object determines a $k$-valued object for every ring $k$. More generally a ring map $k\to k'$ turns a $k$-valued object into a $k'$-valued object. We will sometimes use the same notation for the new object, especially if $k$ is a subring of $k'$. The same applies to morphisms and multiplicative polytope invariants.
\end{rem}

All of the objects mentioned up to now satisfy the following stronger form of the multiplicativity condition: $O_{V+W}=O_VO_W$ if $V\cap W=0$ (even if $V$ and $W$ are not orthogonal). The next result says that these are essentially the only examples of that kind. 
\begin{prop}
If $k$ is an integral domain and the object $O\in \mathcal O_k$ satisfies this stronger multiplicativity, then either $O=0$ or else $O=aO^U$ for some $a\in k$ different from $0$ and some $U\subset \mathbb R^\infty$ different from $0$.
\end{prop}
\begin{proof}
Suppose that $O\neq 0$. There must be at least one line (one-dimensional subspace) $L$ in $\mathbb R^\infty$ such that $O_L\neq 0$. Note that if $L$ and $M$ are two such lines then for any third line $N$ that is contained in $L+M$ we have 
$$
O_LO_N=O_{L+N}=O_{L+M}=O_LO_M,
$$
and therefore $O_N=O_M$. By the same token $O_N=O_L$. This implies that $O_L=O_M$, and it also implies that $O_N\neq 0$. 

Now let $a\in k$ be $O_L$ for some (therefore any) line $L$ such that $O_L\neq 0$, and let $U$ be the vector space spanned by all such lines. We have $O_M=a$ for every line $M$ contained in $U$ and $O_M=0$ for every line $M$ not contained in $U$. That is, $O$ agrees with $aO^U$ on every line. It follows that $O=aO^U$.
\end{proof}
\begin{example}
The following class of objects will not be used below, but it seems worthy of notice. If $T$ is a self-adjoint linear operator on $\mathbb R^\infty$, let $tr_VT$ be the trace of the composition
$$
V\to \mathbb R^\infty\xrightarrow{T} \mathbb R^\infty\to V
$$
of $T$ with inclusion and orthogonal projection. The equation $O_T(V)=e^{tr_VT}$ defines an $\mathbb R$-valued object $O_T$.
We have $O_TO_{T'}=O_{T+T'}$. When $T=c\mathbb I$ is a scalar multiple of the identity then $O_T$ is the scalar object $e^c$.
\end{example}

\subsubsection{Morphisms}

\begin{defin}\label{morphism}
A $k$-valued \emph{morphism} is a collection $F=\lbrace F_V\rbrace$ indexed by the finite-dimensional subspaces of $\mathbb R^\infty$, such that
\begin{itemize}
\item $F_V$ is a valuation of cones in $V$ with values in (the additive group of) $k$
\item $F_{V\times W}(P\times Q)=F_V(P)F_W(Q)$ whenever $V\perp W$ 
\item $F_0(0)=1$, where $0$ is the trivial subspace. 
\end{itemize}
\end{defin}
We express the last two of these three conditions by saying that the collection $\lbrace F_V\rbrace$ of valuations is \emph{multiplicative}.

Let $\mathcal G_k$ be the set of all $k$-valued morphisms.

The four key examples are the yes/no invariant $\epsilon$ and the local Euler invariant $e$ (both $\mathbb Z$-valued) and the conical volume $\mathcal U$ and dual volume $\mathcal W$ (both $\mathbb R$-valued). 

There is a trivial $\mathbb Z$-valued example called $0$, given by $0_V(P)=0$ for every cone in every vector space with the exception of the nonempty cone in the trivial vector space: $0_0(0)=1$.

Of course for any such $F$ and $V$ the valuation $F_V$ may be extended from cones to more general conical constructible sets. We often do so without comment.

\subsubsection{Multiplicative polytope invariants}

\begin{defin}
A $k$-valued \emph{multiplicative polytope invariant} is a collection $F=\lbrace F_V\rbrace$ indexed by the finite-dimensional subspaces of $\mathbb R^\infty$, such that 
\begin{itemize}
\item $F_V$ is a valuation of polytopes in $V$ with values in $k$
\item $F_V$ is invariant under translations in $V$
\item $F_{V\times W}(P\times Q)=F_V(P)F_W(Q)$ whenever $V\perp W$ 
\item $F_0(0)=1$, where $0$ is the trivial subspace. 
\end{itemize}
\end{defin}
Again we express the last two conditions by saying that the collection $\lbrace F_V\rbrace$ of valuations is \emph{multiplicative}.

Let $\mathcal I_k$ be the set of all $k$-valued multiplicative polytope invariants.

The key examples are the Euler invariant $\chi$ ($\mathbb Z$-valued) and the volume $\mathcal V$ ($\mathbb R$-valued).

Again there is a trivial example $0$, given by  $0_V(P)=0$ for every polytope in every vector space, with the sole exception $0_0(0)=1$ for the nonempty polytope in the trivial vector space.

Again we routinely extend $F_V$ from polytopes to constructible sets.

\subsubsection{Scaling}
The monoid $\mathcal O_k$ acts on the sets $\mathcal G_k$ and $\mathcal I_k$ by 
$$
(O\bullet F)_V(P)=O_VF_V(P).
$$
We refer to $O\bullet F$ as the \emph{scaling} of $F$ by $O$. Thus in the case when $O$ is given by an element $a\in k$ we have
$$
(a\bullet F)_V(P)=a^{\text{dim}(V)}F_V(P).
$$
For example, if $F$ is volume measured in feet, square feet, cubic feet, etc., then $12\bullet F$ is volume measured in inches, square inches, cubic inches, etc.

We economize on punctuation by writing $-\bullet F$ for $(-1)\bullet F$ and $-O\bullet F$ for $(-O)\bullet F=-(\mathcal O\bullet F)$.

\subsection{The groupoid}

We now make a groupoid whose  morphism set is $\mathcal G_k$ and whose object set is $\mathcal O_k$. The groupoid itself will also be called $\mathcal G_k$. We first specify sources, targets, identity morphisms, composition, and inverses, and then verify that these definitions make a groupoid.

The \emph{source} and \emph{target} of $F\in \mathcal G_k$ are the objects $s(F)$ and $t(F)$ given by 
$$
s(F)_V=F_V(V)
$$$$
t(F)_V=F_V(0).
$$

\begin{rem}
The morphisms $F$ such that $t(F)=1$ are precisely those which are \emph{absolute} in the sense that $F_V(P)$ does not depend on $V$. Indeed, when $P$ is a cone in $V$ then $F_{V\times W} (P\times 0)=F_V(P)F_W(0)$, and to say that this always equals $F_V(P)$ is to say that $F_W(0)=1$ for all $W$. Among the absolute morphisms are $\epsilon$, $e$, and $\mathcal W$.
At the other extreme, to say that $t(F)=0$ is to say that $F_V(P)$ vanishes whenever $P$ has positive codimension in $V$, in other words, that $F_V:\Sigma(V)\to k$ factors through the projection $\Sigma(V)\to \hat\Sigma(V)$. (In the terminology of \cite{McM1} $F_V$ is \emph{simple}.) An example is $\mathcal U$. 
\end{rem}

The \emph{identity morphism} of the object $O$ is the morphism $1_O$ given by $(1_O)_V(P)=O_V$ for every nonempty cone $P\subset V$. In other words, it is $O\bullet\epsilon$.

The \emph{composition} $F\star G$ of two morphisms is defined as follows. Assume that $s(F)=t(G)$. For any finite-dimensional $V\subset \mathbb R^\infty$ and any cone $P\subset V$ we must define $(F\star G)_V(P)\in k$. Choose a conical cell structure for $P$. Each cell $\sigma$ spans a vector space $\text{span}(\sigma)$, and this has an orthogonal complement $\text{span}(\sigma)^\perp$ in $V$. In the equation below we use the abbreviations $G_\sigma$ for $G_{\text{span}(\sigma)}$ and $F_{\sigma^\perp}$ for $F_{\text{span}(\sigma)^\perp}$. Define
\begin{equation}\label{comp}
(F\star G)_V(P)=\sum_{\sigma}\ F_{\sigma^\perp}(\nu(\sigma,P))\cdot G_{\sigma}(\text{int}\sigma).
\end{equation}
(The symbol ``$\cdot$'' on the right-hand side simply means multiplication in $k$. It is inserted to make some formulas easier to read.)
\begin{rem}
The assumption that $s(F)=t(G)$ will be needed in proving that the right hand side of (\ref{comp}) is well defined (independent of the choice of conical cell structure of $P$). 
\end{rem}
The \emph{inverse} of a morphism $F$ is a morphism $FD$ defined by duality: for a convex cone $P$ in $V$,
$$
(FD)_V(P)=F_V(D_VP)).
$$

\begin{thm}\label{groupoidthm}
With the definitions above, $\mathcal G_k$ and $\mathcal O_k$ form a groupoid.
\end{thm}
\begin{proof}
We verify the following statements in the following order: 
\begin{itemize}
\item $(F\star G)_V(P)$ is well-defined. 
\item For each $V$ the function $(F\star G)_V$ is a valuation.
\item The family $F\star G=\lbrace (F\star G)_V\rbrace$ is multiplicative.
\item$s(F\star G)=s(G)$.
\item$t(F\star G)=t(F)$.
\item The $\star$-product is associative.
\item$1_O$ is a two-sided identity.
\item $FD$ is a morphism, and it is an inverse to $F$.
\end{itemize}

Suppose that we have two conical cell structures of $P$, one a refinement of the other. Denote the cells of the coarser structure by $\sigma$ and those of the finer one by $\tau$. For each $\tau$ there is a unique $\sigma$ such that $\text{int}\tau\subset \text{int}\sigma$; say that $\tau$ \emph{belongs to} $\sigma$. Let us rewrite each term in the sum 
$$
\sum_{\tau}\ F_{\tau^\perp}(\nu(\tau,P))\cdot G_{\tau}(\text{int}\tau),
$$
using the simplex $\sigma$ that $\tau$ belongs to. The cone $\nu(\tau,P)$ is the orthogonal product $\nu(\sigma,P)\times W(\tau,\sigma)$, where the vector space $W(\tau,\sigma)$ is the orthogonal complement of $\text{span}(\tau)$ in $\text{span}(\sigma)$. Therefore we have 
$$
F_{\tau^\perp}(\nu(\tau,P))\cdot G_{\tau}(\text{int}\tau)=F_{\sigma^\perp}(\nu(\sigma,P))\cdot F_{W(\tau,\sigma)}(W(\tau,\sigma))\cdot G_{\tau}(\text{int}\tau).
$$
Because $s(F)=t(G)$, this becomes
$$
F_{\sigma^\perp}(\nu(\sigma,P))\cdot G_{W(\tau,\sigma)}(0)\cdot G_{\tau}(\text{int}\tau)=F_{\sigma^\perp}(\nu(\sigma,P))\cdot G_{\sigma}(\text{int}\tau).
$$
The sum of these terms over all those $\tau$ which belong to a given $\sigma$ is 
$$
F_{\sigma^\perp}(\nu(\sigma,P))\cdot G_{\sigma}(\text{int}\sigma),
$$
because $\text{int}\sigma$ is the union of the disjoint sets $\text{int}\tau$. Therefore the sum over all $\tau$ is equal to the sum of $F_{\sigma^\perp}(\nu(\sigma,P))\cdot G_{\sigma}(\text{int}\sigma)$ over all $\sigma$, as asserted.

We now verify that $(F\star G)_V$ is a valuation of cones in $V$. It is trivially true that
$$
(F\star G)_V(\emptyset)=0,
$$
because the empty cone has no conical cells. To prove that 
$$
(F\star G)_V(P\cup Q)=(F\star G)_V(P)+(F\star G)_V(Q)-(F\star G)_V(P\cap Q),
$$
choose a conical cell structure of $P\cup Q$ that makes $P$ and $Q$ subcomplexes. When a cell $\sigma$ is contained in both $P$ and $Q$ then the cones $\nu(\sigma, P)$ and $\nu(\sigma,Q)$ have union $\nu(\sigma,P\cup Q)$ and intersection $\nu(\sigma, P\cap Q)$, so that 
$$
F_{\sigma^\perp}(\nu(\sigma,P\cup Q))=F_{\sigma^\perp}(\nu(\sigma,P))+F_{\sigma^\perp}(\nu(\sigma,Q))-F_{\sigma^\perp}(\nu(\sigma,P\cap Q)).
$$
When $\sigma$ is contained in $P$ but not in $Q$ then $\nu(\sigma,P\cup Q)=\nu(\sigma,P)$ and
$$
F_{\sigma^\perp}(\nu(\sigma,P\cup Q))=F_{\sigma^\perp}(\nu(\sigma,P))+0-0.
$$
When $\sigma$ is contained in $Q$ but not in $P$ then $\nu(\sigma,P\cup Q)=\nu(\sigma,Q)$ and
$$
F_{\sigma^\perp}(\nu(\sigma,P\cup Q))=0+F_{\sigma^\perp}(\nu(\sigma,Q))-0.
$$
Multiply all terms in the last three equations by $G_\sigma(\text{int}\sigma)$ and then sum over all $\sigma$ to obtain the conclusion.

We verify that $F\star G$ is multiplicative. It is trivially true that $(F\star G)_0(0)=1$. To see that 
\begin{equation}\label{starmult}
(F\star G)_{V\times W}(P\times Q)=(F\star G)_V(P)\cdot (F\star G)_W(Q)
\end{equation}
when $V\perp W$, choose conical cell structures for $P$ and $Q$. As $\sigma$ ranges over the cells of $P$ and $\tau$ ranges over those of $Q$, the products $\sigma\times \tau$ are the cells of a conical cell structure for $P\times Q$. Moreover, 
$$
\nu( \sigma\times \tau,P\times Q)=\nu(\sigma,P)\times \nu( \tau,Q)
$$
and 
$$
\text{int}(\sigma\times\tau)=\text{int}\sigma\times \text{int}\tau.
$$
The left-hand side of (\ref{starmult}) is equal to the sum of
$$
F_{(\sigma\times\tau)^\perp}(\nu(\sigma\times \tau,{P\times Q}))\cdot G_{\sigma\times\tau}(\text{int}(\sigma\times\tau))
$$
over all ordered pairs $(\sigma,\tau)$. This is the same as the sum of 
$$
F_{\sigma^\perp}(\nu(\sigma,P))\cdot G_{\sigma}(\text{int}\sigma)\cdot F_{\tau^\perp}(\nu(\tau,Q))\cdot G_{\tau}(\text{int}\tau)
$$
and is therefore equal to the right hand side of (\ref{starmult}).

To verify that
$$s(F\star G)=s(G),
$$
evaluate both sides on a line $L$. (Two objects must be equal if they agree on lines.) We must show that $(F\star G)_L(L)=G_L(L)$. The only conical triangulation of $L$ has one $0$-simplex (the origin) and two $1$-simplices (half-lines). The $0$-simplex contributes the term 
$$
F_L(\nu(0,L) )\cdot G_0(0)=F_L(L)\cdot 1=G_L(0).
$$
A $1$-simplex $\sigma$ contributes the term 
$$
F_0(\nu(\sigma,L))\cdot G_L(\text{int}\sigma)=F_0(0)\cdot G_L(\sigma -0)=G_L(\sigma -0).
$$
The sum of the three terms is $G_L(L)$, since the line $L$ is the disjoint union of $0$ and the two open half-lines $\sigma-0$.

To verify that
$$
t(F\star G)=t(F),
$$
again evaluate both sides on $L$. We must show that $(F\star G)_L(0)=F_L(0)$. The only conical cell structure of $0$ has just one ($0$-dimensional) simplex. In this case the one term is 
$$
F_L(\nu(0,0))\cdot G_0(\text{int}0)=F_L(0)\cdot G_0(0)=F_L(0).
$$

For associativity, first recall from \S\ref{normal cones} that when $\tau$ is a cell in a conical cell structure of $P$ then $\nu(\tau,P)$ has a conical cell structure whose simplices are the normal cones $\nu(\tau,\sigma)$ of $\tau$ in all the cells $\sigma$ that contain it. 
Now suppose that $F$, $G$, and $H$ are multiplicative $k$-valued invariants with $s(F)=t(G)$ and $s(G)=t(H)$. Both $((F\star G)\star H)_V(P)$ and $(F\star (G\star H))_V(P)$ are equal to the sum, over all pairs $\tau\subset \sigma$ of cells in a conical cell structure of $P$, of 
$$
F_{\sigma^\perp}(\nu(\sigma,P))\cdot G_{W(\tau,\sigma)}(\text{int}\nu(\tau,\sigma)\cdot H_{\tau}(\text{int}\tau),
$$
where again $W(\tau,\sigma)$ is the orthogonal complement of $\text{span}(\tau)$ in $\text{span}(\sigma)$. 

To verify that $1_{O}\star G=G$ for every morphism $G$ such that $t(G)=O$, evaluate $(1_{O}\star G)_V$ on a conical polytope $P\subset V$. For every cell $\sigma$ in a conical cell structure, the term $(1_{ O})_{\sigma^\perp}(\nu(\sigma,P))\cdot G_{\sigma}(\text{int}\sigma)$ is 
$$
 O_{\sigma^\perp}\cdot G_{\sigma}(\text{int}\sigma)=G_{\sigma^\perp}(0)\cdot G_{\sigma}(\text{int}\sigma)=G_V(\text{int}\sigma).
$$
The sum of this over all $\sigma$ is $G_V(P)$ because $P$ is the union of the disjoint sets $\text{int}\sigma$.

To verify that $F\star 1_{ O}=F$ for every morphism $F$ with $s(F)=O$, evaluate $F\star 1_{ O}$ on $P$. For each cell $\sigma$ except $0$, the term 
$$
F_{\sigma^\perp}(\nu(\sigma,P))\cdot (1_{ O})_{\sigma}(\text{int}\sigma)
$$
is zero, because 
$$
(1_{ O})_{\sigma}(\text{int}\sigma)=O_\sigma\cdot \epsilon(\text{int}\sigma)
$$
and 
$$
\epsilon(\text{int}\sigma)=\epsilon(\sigma)-\epsilon(\partial\sigma)=1-1=0.
$$ 
The remaining term is 
$$
F_V(\nu(0,P))\cdot  (1_O)_0(\text{int}0)=F_V(P)\cdot 1=F_V(P)
$$
because $\epsilon(\text{int}0)=\epsilon(0)=1$.

To show that $FD$ is an inverse for $F$, first note that $FD$ is a morphism: by Proposition \ref{Dadd} each valuation $(FD)_V$ is well-defined, and by (\ref{Dprod})  the family $FD$ is multiplicative.
Source and target are reversed:
$$
s(FD)_V=(FD)_V(V)=F_V(D_VV)=F_V(0)=t(F)_V
$$
$$
t(FD)_V=(FD)_V(0)=F_V(D_V0)=F_V(V)=s(F)_V
$$
As soon as we have shown that $FD\star F=1_{s(F)}$, it will follow that $F\star FD=1_{t(F)}$ as well, because when every morphism in a category has a left inverse then left inverses are two-sided inverses. 

We must show that 
$$
(FD\star F)_V(P)=s(F)_V
$$
for every convex cone $P\subset V$. Using the canonical conical cell structure on $P$ in which the faces are the cells, we compute 
$$
(FD\star F)_V(P)=\sum_\sigma\ F_{\sigma^\perp}(D_V\nu(\sigma,P))\cdot F_{\sigma}(\text{int}\sigma)=\sum_\sigma\ F_V(D_V\nu(\sigma,P)\times \text{int}\sigma).
$$
The set 
$$
D_V\nu(\sigma,P)\times\text{int}\sigma\subset V
$$
consists of those vectors $v\in V$ such that whatever point of $P$ is nearest to $v$ belongs to the interior of the face $\sigma$. Since $V$ is the union of these disjoint conical constructible sets, the sum is $F_V(V)=s(F)_V.$
\end{proof}

It is clear that scaling is compatible with source, target, identity, and composition. Explicitly,
$$
s(O\bullet F)= Os(F)
$$
$$
t(O\bullet F)= Ot(F)
$$
$$
1_{OP}= O\bullet 1_{ P}
$$
$$
(O\bullet F)\star  (O\bullet G)=O\bullet (F\star G).
$$
In other words, when the monoid $\mathcal O_k$ acts on the set $\mathcal O_k$ by multiplication and on the set $\mathcal G_k$ by scaling, this constitutes an action of $\mathcal O_k$ on the groupoid. 
\begin{example}\label{zeroscale}
The scaling $0\bullet F$ of any morphism $F$ by the object $0$ is the morphism $0$. In particular the identity morphism $1_0=0\bullet\epsilon$ is $0$.
\end{example}
\begin{rem}
By calling a certain object $0$ and calling a certain morphism $0$ we are creating potential confusion. Of course these are not a zero object and a zero morphism in the category-theoretic sense.\end{rem}

\begin{example}
$\mathcal U$ is an $\mathbb R$-valued morphism from $1$ to $0$, and $\mathcal W$ is its inverse.
\end{example}

\begin{example}\label{minus e}
$e$ is a $\mathbb Z$-valued morphism from $-1$ to $1$. Thus for every $k$-valued object $O$ we have a $k$-valued morphism $O\bullet e$ from $-O$ to $O$ and a $k$-valued morphism $-O\bullet e$ from $O$ to $-O$.
These are inverses:
\begin{equation}\label{-e}
-O\bullet e=(O\bullet e)^{-1}.
\end{equation}
To see this, it suffices (because scaling by $O$ is a groupoid map) to consider the case $O=1$. One of many ways of seeing that $-\bullet e$ is the inverse of $e$ is Equation (\ref{eD}).
\end{example}

\subsection{The $\star$-product of a morphism and a polytope invariant}

A multiplicative polytope invariant $G\in \mathcal I_k$ has, so to speak, a target but no source. (A point in $V$ is a polytope, but $V$ itself is not.) The object $p(G)\in \mathcal O_k$ given by 
$$
p(G)_V=G_V(\ast)
$$
will be called the \emph{location} of $G$.

\begin{defin}
If $F\in \mathcal G_k$, $G\in \mathcal I_k$, and $s(F)=p(G)$, then $F\star G\in \mathcal I_k$ is defined by equation (\ref{comp}). Here the summation is over the cells in a cell structure of a polytope $P\subset V$.
\end{defin}
\begin{thm}
$F\star G$ is an element of $\mathcal I_k$. We have
\begin{itemize}
\item$(F\star G)\star H=F\star (G\star H)$ when $s(F)=t(G)$ and $s(G)=p(H)$
\item $p(F\star G)=t(F)$
\item$1_{p(G)}\star G=G$.
\end{itemize}
\end{thm}
\begin{proof}
This is essentially identical to the proof of Theorem \ref{groupoidthm}. The fact that the right-hand side of (\ref{comp}) is independent of the cell structure is proved just as in the conical case. $(F\star G)_V$ is translation-invariant because $G$ is. The proof that it is a valuation and the proof of multiplicativity are essentially the same as in the conical case, as is the proof of associativity. The last two identities are proved like the identities $t(F\star G)=t(F)$ and $1_{t(G)}\star G=G$.
\end{proof}
This piece of structure could be summarized by saying that we have a functor from $\mathcal G_k$ to the category of sets, taking the object $O$ to the set of all $G\in \mathcal I_k$ such that $p(G)=O$, and taking the morphism $F$ to the map $G\mapsto F\star G$. Instead, we will use the following language. Note that if $F$ is a morphism from $ O$ to $P$ and $G$ is located at $O$ then $F\star G$ is located at $ P$. We speak of using $F$ to \emph{transport} $G$ from $O$ to $P$.

Location and transport are compatible with scaling:
$$
p(O\bullet F)=Op(F)
$$
$$
(O\bullet F)\star  (O\bullet G)= O\bullet (F\star G).
$$

Observe that the $\mathbb R$-valued invariant $\mathcal V$ is located at $0$ and the $\mathbb Z$-valued invariant $\chi$ is located at $1$.

\begin{example}\label{Vchi}
$$\mathcal W\star 0=\chi
$$
Here $0$ is the trivial polytope invariant and $\mathcal W$ is the dual volume.
To see what the equation says, evaluate the left hand side on a nonempty convex polytope $P\subset\mathbb R^n$ using the cell structure in which faces are cells:
$$
(\mathcal W\star 0)(P)=\sum_{F}\ \mathcal W_{F^\perp}(\nu(F,P))\cdot 0_{F}(\text{int}F).
$$
The quantity $0_F(\text{int}F)$ vanishes when the face $F$ has positive dimension, and it equals $1$ when $F$ is a point. Thus $(\mathcal W\star 0)(P)$ is the sum, over extreme points $v\in P$, of the conical volume of the dual cone $D\nu(v,P)$. The fact that the sum is $1$ essentially says that for an observer standing very far away from a nonempty convex polytope $P\subset \mathbb R^n$ the probability of the nearest point in $P$ being an extreme point is approximately $1$.

Because $\mathcal W$ is the inverse of $\mathcal U$, an equivalent statement is
$$
\mathcal U\star \chi=0.
$$
\end{example}

\section{Relation with McMullen's work}

The aim here is to recall several constructions in \cite{McM1} and to remark that they can be seen as instances of the $\star$-product. The rest of the paper does not depend on this material.

We begin by defining two rings $\hat\Sigma$ and $\hat\Pi$, and an element $\tau^\Sigma$ of $\mathcal G_{\hat\Sigma}$, and an element $\tau^\Pi$ of $\mathcal I_{\hat\Pi}$. 
\subsection{Some universal constructions}\label{universal}

Recall the cone group $\Sigma(V)$ of a finite-dimensional vector space, and recall that we denote the top quotient in its dimension filtration by $\hat \Sigma(V)$. Write $\hat \Sigma$ for the direct sum of $\hat \Sigma(V)$ over all finite-dimensional $V$ in $\mathbb R^\infty$. We make $\hat \Sigma$ into a ring as follows. If $V$ is orthogonal to $W$ then the product of an element of $\hat \Sigma(V)$ and an element of $\hat \Sigma(W)$ is the element of  $\hat \Sigma(V\times W)$ given by the map 
$$
\hat\Sigma(V)\times  \hat\Sigma(W)\to\hat\Sigma(V\times W)
$$
induced by the external product map $ \Sigma(V)\times  \Sigma(W)\to\Sigma(V\times W)$. If $V$ is not orthogonal to $W$ then the product of an element of $\hat \Sigma(V)$ and an element of $\hat \Sigma(W)$ is zero. 

Recall that $\Pi(V)$ is the group of coinvariants for the action of the translation group on $\mathcal P(V)$. Let $\Pi_n(V)\subset \Pi(V)$ be the image of $\mathcal P_n(V)$, so that there is a filtration
\begin{equation}\label{Mc filtration}
\Pi_0(V)\subset \Pi_1(V)\subset \Pi_2(V)\subset \dots \subset \Pi_d(V)=\Pi(V),
\end{equation}
where $d$ is the dimension of $V$, and write $\hat \Pi(V)$ for the top quotient $\Pi_d(V)/\Pi_{d-1}(V)$. (This quotient is also called the translational scissors congruence group of $V$.) Write $\hat \Pi$ for the direct sum of $\hat \Pi(V)$ over all $V\subset\mathbb R^\infty$. This becomes a ring in the same way as $\hat \Sigma$. 

Define a morphism $\tau^\Sigma$ with values in $\hat \Sigma$ by declaring that, for a cone $P$ in $V$, the element $\tau^\Sigma_V(P)$ is the image of $(P)\in\Sigma(V)$ under projection to $\hat \Sigma(V)\subset \hat \Sigma$. The target $t(\tau^\Sigma)$ is $0$, because if $\text{dim}V>0$ then the point element $(0)\in\Sigma(V)$ projects to $0\in\hat\Sigma(V)$. The source $s(\tau^\Sigma)$ is the object that sends $V$ to the element of $\hat\Sigma(V)\subset \hat\Sigma$ given by the entire vector space $V$.

Similarly define a multiplicative $\hat \Pi$-valued polytope invariant $\tau^\Pi$ with values in $\hat \Pi$ by declaring that, for a polytope $P$ in $V$, the element $\tau^\Pi_V(P)$ is the image of $(P)\in\Pi(V)$ under projection to $\hat \Pi(V)\subset \hat \Pi$. It is located at $0$.

\subsection{The injection $\sigma$}\label{sigma}

McMullen (\S 13 of \cite{McM1}) defines a map
$$
\sigma_V:\Pi(V)\to \bigoplus_{W\subset V} \hat\Sigma(W)\otimes \mathbb R
$$
and proves that it is injective. Let us describe it as a $\star$-product. 

Consider the ring $\hat \Sigma\otimes_{\mathbb Z} \mathbb R$. This is the coproduct of $\hat\Sigma$ and $\mathbb R$ in the category of (commutative unital) rings, via the ring maps 
$$
\xi\mapsto \xi\otimes 1
$$$$
x\mapsto 1\otimes x.
$$
Applying the first of these ring maps to $\tau^\Sigma$ we obtain a morphism with values in $\hat \Sigma\otimes_{\mathbb Z} \mathbb R$ which may be written as $\tau^\Sigma\otimes 1$:
$$
(\tau^\Sigma\otimes 1)_V(P)=\tau^\Sigma_V(P)\otimes 1.
$$
Its target is $0\otimes 1=0$. 

Applying the second ring map to $\mathcal V$, we obtain a multiplicative polytope invariant which may be written as $1\otimes \mathcal V$:
$$
(1\otimes \mathcal V)_V(P)=1\otimes \mathcal V_V(P).
$$
It is located at $1\otimes 0=0$.

Let $\sigma$ be $(\tau^\Sigma \otimes 1)^{-1}\star (1 \otimes \mathcal V)$, a multiplicative polytope invariant. 
\begin{rem}
Although $(\tau^\Sigma )^{-1}$ has source $0$ and $ \mathcal V$ is located at $0$, their $\star$-product is undefined because they take values in two different rings. The function of the tensor product is to force them both into the same larger ring.
\end{rem}
The value of the map 
$$
\sigma_V:\Pi(V)\to \hat\Sigma\otimes \mathbb R
$$
on a convex polytope $P$ is given by
$$
\sigma_V(P)=\sum_F\  \tau_{F^\perp}(D\nu(F,P))\otimes \mathcal V_F(F),
$$
where $F$ runs through the faces of $P$. The range of this map $\sigma_V$ is contained in the subgroup $\bigoplus_{W\subset V}\hat \Sigma(W)\otimes \mathbb R$, and when regarded as a map of $\Pi(V)$ into that subgroup this is McMullen's $\sigma_V$.

\subsection{The splitting of the dimension filtration}\label{M split}

One of the main results of \cite{McM1} is that for a finite-dimensional vector space $V$ there is a canonical splitting of the dimension filtration (\ref{Mc filtration}) of $\Pi(V)$. For each $n$-dimensional vector subspace $W\subset V$ there is an obvious map $\hat \Pi(W)\to \Pi_n(V)/\Pi_{n-1}(V)$, and these may be combined to make a map
\begin{equation}\label{surj}
\bigoplus_{\text{dim}W=n} \hat \Pi(W)\to \Pi_n(V)/\Pi_{n-1}(V)
\end{equation}
which is clearly surjective. McMullen in effect produces a map 
\begin{equation}\label{Msplitting}
\Pi(V)\to \bigoplus_{W\subset V}\hat \Pi(W)
\end{equation}
and shows that it induces, for each $n$, a left inverse for (\ref{surj}), thus splitting the filtration (\ref{Mc filtration}) and also establishing that (\ref{surj}) is an isomorphism. It follows that (\ref{Msplitting}) is itself an isomorphism of groups.

We can explain and perhaps simplify the proof using $\star$. The idea is to form the $\star$-product of $\mathcal W$ and $\tau^\Pi$. Again this is nonsense on the face of it because $\mathcal W$ and $\tau^\Pi$ take values in two different rings, namely $\hat\Pi$ and $\mathbb R$, but in this case one does not need anything so drastic as a tensor product to fit the two rings into a larger ring. 

One of the deeper facts about translational scissors congruence (see \cite{JT}, or Chapter 3 of \cite{Sah}) is that when the vector space $V$ has positive dimension then the group $\hat \Pi(V)$ has a canonical real vector space structure. With respect to this structure the external product $\hat\Pi(V)\times \hat\Pi(W)\to \hat \Pi(V\times W)$ is $\mathbb R$-bilinear. Now enlarge the ring $\hat\Pi$ by replacing the summand $\mathbb Z=\hat \Pi(0)\subset\hat \Pi$ by a copy of $\mathbb R$. Call the result $\hat\Pi_{\mathbb R}$. This is again a ring, and in fact an $\mathbb R$-algebra. Therefore we can view both $\mathcal W$ and $\tau^\Pi$ as taking values in $\hat\Pi_{\mathbb R}$ and so combine them to make $\mathcal W\star \tau^\Pi$, a multiplicative polytope invariant with values in $\hat\Pi_{\mathbb R}$. 

Explicitly, for a polytope $P\subset V$
$$
(\mathcal W\star \tau^\Pi)_V(P)=\sum_{\sigma}\ \mathcal W(\nu(\sigma,P))\tau^\Pi(\sigma),
$$
where $\sigma$ ranges over the cells in a cell structure of the polytope $P$. (If $P$ is convex, as in \cite{McM1}, then these can be taken to be the faces of $P$.)

To repeat, the reason why it is possible to multiply the real number $\mathcal W(\nu(\sigma,P))$ by the scissors congruence class $\tau^\Pi(\sigma)$ is the deep fact mentioned above. 

In the exceptional case when $\sigma$ is a $0$-cell, the expression $\mathcal W(\nu(\sigma,P))\tau^\Pi(\sigma)$ denotes an element of $\hat\Pi_{\mathbb R}(0)=\mathbb R$ rather than $\hat\Pi(0)=\mathbb Z$. However, the sum of this expression over all $0$-cells is
$$
\sum_{\text{dim}\sigma =0} \mathcal W(\nu(\sigma,P)),
$$
since $\tau^\Pi(\sigma)=1$ for every $0$-cell $\sigma$, and by (\ref{Vchi}) this is the integer $\chi(P)$. Thus in fact $\mathcal W\star\tau^\Pi$ takes values in the smaller ring $\hat\Pi$ after all!

The map $(\mathcal W\star \tau^\Pi)_V:\Pi(V)\to \hat \Pi$ defined above is (except for a change in codomain) the required map (\ref{Msplitting}).

\subsection{Frame invariants}

McMullen (\cite{McM1}, \S10) defines a collection of real valuations $f_U$ for polytopes in a finite-dimensional vector space $V$ and shows that when taken together these invariants separate $\Pi(V)$. There is one of these for each frame $U=(u_1,\dots ,u_k)$ in $V$, i.e. for each finite orthonormal sequence of vectors. The definition can be given as follows.

Let $P\subset V$ be a nonempty convex polytope (of any dimension). The frame $U$ determines a sequence of faces of $P$. Let $F_1^UP$ be the face of $P$ on which the inner product $\langle -,u_1\rangle$ is maximized, let $F_2^UP$ be the face of $F_1^UP$ on which the inner product $\langle -,u_2\rangle$ is maximized, and so on. The face $F_j^UP$ has dimension at most $d-j$ where $d$ is the dimension of $V$, because it is contained in an affine space orthogonal to the vectors $u_1,\dots ,u_j$. The number $f_U(P)$ is defined to be the $(d-k)$-dimensional volume $\mathcal V_{d-k}(F_k^UP)$.

We now indicate how this can be placed in the context of multiplicative invariants and the $\star$-product.

Fix a frame $U=(u_1,\dots ,u_k)$ in $\mathbb R^\infty$, spanning a vector space $W$. The frame determines a notion of positivity for vectors in $\mathbb R^\infty$, as follows. Say that $v\ge 0$ (with respect to $U$) if either the inner products $\langle v,u_i\rangle$ are all zero or the first one that is not zero is positive. This property of $v$ is invariant under positive dilatation and is closed under addition. For every $v$, either $v\ge 0$ or $-v\ge 0$. The vectors satisfying both $v\ge 0$ and $-v\ge 0$ are those belonging to $W^\perp$.

For a (finite-dimensional) convex cone $P\subset\mathbb R^\infty$ we write $P\ge 0$ if $v\ge 0$ for all $v\in P$.

Make a $\mathbb Z$-valued function of convex cones (depending on the frame $U$) by defining $L^U(P)=1$ if $P\ge 0$ and otherwise $L^U(P)=0$. This passes the hyperplane test of Lemma \ref{cone cutlemma}, because if $P\cup Q$ is convex and $P\cap Q\ge 0$ then either $P\ge 0$ or $Q\ge 0$. Moreover, it is multiplicative. Thus it gives an element of $\mathcal G_{\mathbb Z}$. Since $L^U(P)$ is independent of the ambient finite-dimensional vector space, this morphism is absolute; that is, the target $t(L^U)$ is $1$. The source $s(L^U)$ is the object $O^{W^\perp}$ defined in Example \ref{space object}.

Scale $L^U$ by the object $\mathcal O^W$ to obtain a morphism $O^{W}\bullet L^U$ with source $O^{W}O^{W^\perp}=O^0=0$ and target $O^{W}$. Transport $\mathcal V$ by $xO^W\bullet L^U$ to get an $\mathbb R\lbrack x\rbrack$-valued multiplicative polytope invariant 
$$
(xO^W\bullet L^U)\star \mathcal V
$$
located at $xO^W$. 

We leave it as an exercise for the reader to verify that if the $d$-dimensional vector space $V$ contains the $k$-frame $U$ then for any nonempty convex polytope $P\subset V$ the coefficient of $x^k$ in the degree $k$ polynomial 
$$
((xO^W\bullet L^U)\star Vol)_V(P)
$$
is McMullen's $f_U(P)$.

Here is a further exercise, which the author has not succeeded in carrying out: use the $\star$-product to reformulate and perhaps streamline McMullen's proof that the frame invariants separate elements of $\Pi(V)$, and to clarify their relationship with the invariant $\sigma$ of \S\ref{sigma}.

\section{The local Euler invariant and the $\star$-product}\label{e star}
We have seen (Example \ref{minus e}) that, up to scaling, the morphism $e$ is its own inverse. We now show that, up to scaling, $\star$-product with $e$ on the left is given by the interior operator and $\star$-conjugation by $e$ corresponds to dilatation by $-1$.

For every morphism $F$ there is precisely one object $O$ such that $(O\bullet e)\star F$ is defined, namely $-t(F)$, and there is precisely one object $O$ such that $F\star (O\bullet e)$ is defined, namely $s(F)$. Similarly every multiplicative polytope invariant $G$ can be transported by precisely one scaling of $e$, namely $-p(G)\bullet e$. We will give formulas for 
$(-p(G)\bullet e)\star G$, $(-t(F)\bullet e)\star F$, and $F\star (s(F)\bullet e)$.

If $G$ belongs to $\mathcal I_k$, denote by $GI$ the element of $\mathcal I_k$ defined by $(GI)_V=G_V\circ I$. Here we are thinking of $G_V$ as a (translation-invariant) group homomorphism $\mathcal P(V)\to k$ and composing it with $I:\mathcal P(V)\to \mathcal P(V)$. Thus, for example, if $P\subset V$ is a convex polytope then 
$$
(GI)_V(P)=(-1)^{dim (P)}G_V(\text{int}P).
$$
\begin{prop}\label{ep}
For any $G\in \mathcal I_k$ we have 
$$
(-p(G)\bullet e)\star G=-\bullet GI.
$$
Equivalently,
$$
GI=(p(G)\bullet e)\star (-\bullet G)
$$
\end{prop}
\begin{proof}
We prove the first equation. Let $P$ be a nonempty $m$-dimensional convex polytope in a $d$-dimensional vector space $V$. We have 
$$
((-p(G)\bullet e)\star G)_V(P)=\sum_\sigma (-p(G)\bullet e)_{\sigma^\perp}(\nu(\sigma,P))\cdot G_\sigma(\text{int}\sigma),
$$
where the sum is over cells of any cell structure of $P$. Use the cell structure in which the cells are the faces. For every face $\sigma$ with the exception of $P$ itself the corresponding term is zero, 
since $\nu(\sigma,P)$ is an ordinary convex cone of positive dimension and therefore $e(\nu(\sigma, P))$ is zero. The term corresponding to $\sigma=P$ is
$$
(-1)^{d-m} G_{P^\perp}(0)\cdot G_P(\text{int}P)=(-1)^{d-m}G_V(\text{int}P)=(-1)^d(GI)_V(P)=(-\bullet GI)_V(P).
$$
\end{proof}
\begin{rem}
The preceding result is essentially a restatement of Equation (\ref{triang}), since $e(\nu(\sigma,P))=1-\chi \text{Lk}(\sigma,P)$.
\end{rem}
Here is the conical analogue:
\begin{prop}\label{et}
For any $F\in \mathcal G_k$ we have 
$$
(-t(F)\bullet e)\star F=-\bullet FI.
$$
Equivalently, 
$$
FI=(t(F)\bullet e)\star (-\bullet F).
$$
\end{prop}
\begin{proof}
This is so similar to the preceding proof that we omit it.
\end{proof}
We now turn from left $\star$-multiplication by scalings of $e$ to right $\star$-multiplication by scalings of $e$. The key to this will be the following immediate consequence of Lemma \ref{DI}:
\begin{equation}\label{FDIDI}
FDIDI=-\bullet F\Delta(-1).
\end{equation}
Since $FD=F^{-1}$, Proposition \ref{et} implies
$$
FDI=F^{-1}I=(t(F^{-1})\bullet e)\star (-\bullet F^{-1})=(s(F)\bullet e)\star (-\bullet F^{-1}).
$$
Using (\ref{-e}) this implies
$$
FDID=(-\bullet F)\star (-s(F)\bullet e)=-\bullet (F\star (s(F)\bullet e)).
$$
Using Proposition \ref{et} again, we find
$$
FDIDI=(-t(F)\bullet e)\star F\star (s(F)\bullet e).
$$
Thus (\ref{FDIDI}) gives 
\begin{equation}\label{eqn}
-\bullet F\Delta(-1)=(-t(F)\bullet e)\star F\star (s(F)\bullet e).
\end{equation}
This may also be written
$$
F\star (s(F)\bullet e)=(t(F)\bullet e)\star (-\bullet F\Delta(-1).
$$
Using Proposition \ref{et} again, we conclude:
\begin{prop}\label{te}
$$
F\star (s(F)\bullet e)=F\Delta(-1)I.
$$
\end{prop}
\begin{example}\label{e}
In the case when $F$ is the conical volume $\mathcal U$, this means that 
$$
\mathcal U\star e=UI
$$
and therefore by (\ref{VIcone}) 
$$
\mathcal U\star e=-\bullet \mathcal U.
$$
In other words, the $\mathbb Z$-valued morphism $e$ coincides with the $\mathbb R$-valued morphism $\mathcal W\star(-\bullet \mathcal U)$.
\end{example}

\section{Euclidean invariants}\label{Euclideancase}
In the rest of this paper we will be concerned almost exclusively with functions of polytopes and cones that are invariant under isometry. Here we introduce notation for those and explain how to dispense with $\mathbb R^\infty $ when working with them.

The group of invertible linear isometries $\mathbb R^\infty\to \mathbb R^\infty$ acts on the set of finite-dimensional vector subspaces and on the set of cones and the set of polytopes, preserving all relevant structure. Therefore it acts on the sets $\mathcal O_k$, $\mathcal G_k$, and $\mathcal I_k$ in such a way as to preserve $\star$-products, sources, targets, locations, scaling, and so on.

Let $\mathcal G^O_k\subset \mathcal G_k$ be the set of all morphisms fixed by this action, in other words those morphisms $F$ such that $F_V(P)=F_{V'}(P')$ whenever there is a linear isometry $V\cong V'$ taking $P$ to $P'$. Examples of such \emph{Euclidean morphisms} include $\epsilon$, $e$, $\mathcal U$, and $\mathcal W$. 

Define $\mathcal I^O_k\subset \mathcal I_k$ likewise. Examples of \emph{Euclidean multiplicative polytope invariants} include $\chi$ and $\mathcal V$. 

The corresponding set $\mathcal O^O_k\subset \mathcal O_k$ is simply $k$. That is, the only fixed objects are the scalars.

Thus there is a sub-groupoid $\mathcal G^O_k\subset \mathcal G_k$ whose object set is $k$. Its morphisms transport Euclidean polytope invariants from one element of $k$ to another. Elements of $\mathcal G^O_k$ and $\mathcal I^O_k$ can be scaled by elements of $k$.

Note that when $F$ belongs to $\mathcal G^O_k$ then $F_V$ is determined by $F_{\mathbb R^n}$ where $n=\text{dim}V$. In fact, an element $F$ of $\mathcal G^O_k$ can be described more economically as a sequence $\lbrace F_n\rbrace$ where
\begin{itemize}
\item $F_n$ is a $k$-valued valuation of cones in $\mathbb R^n$
\item $F_n$ is invariant under linear isometries of $\mathbb R^n$
\item $F_{p+q}(P\times Q)=F_p(P)F_q(Q)$ for any cones $P\subset \mathbb R^p$ and $Q\subset \mathbb R^q$, where $\mathbb R^p\times \mathbb R^q$ is being identified with $\mathbb R^{p+q}$ in the usual way
\item $F_0(0)=1$.
\end{itemize}
We will often use this description. At the same time we will also regard the function $F_n$ as defined not only on cones in $\mathbb R^n$ but on cones (or constructible conical sets) in any $n$-dimensional vector space (with inner product).

Likewise an element $F$ of $\mathcal I^O_k$ can be described as a sequence $\lbrace F_n\rbrace$ where
\begin{itemize}
\item $F_n$ is a $k$-valued valuation of polytopes in $\mathbb R^n$
\item $F_n$ is invariant under affine isometries (Euclidean symmetries) of $\mathbb R^n$
\item $F_{p+q}(P\times Q)=F_p(P)F_q(Q)$
\item $F_0(0)=1$.
\end{itemize}
Again when describing $F$ in this way we will regard $F_n$ as defined not only on polytopes in $\mathbb R^n$ but on polytopes (or constructible sets) in any $n$-dimensional vector space.

In this notation the formula for the $\star$-product reads
\begin{equation}\label{star Euclidean}
(F\star G)_n(P)=\sum_\sigma F_{n-|\sigma|}(\nu(\sigma,P))\cdot G_{|\sigma|}(\text{int}\sigma),
\end{equation}
where $P$ is a polytope or cone in $\mathbb R^n$ and $\sigma$ ranges over all cells in a cell structure or conical cell structure of $P$, and $|\sigma|$ is the dimension of $\sigma$.

\begin{rem}\label{abstract}
It can be useful to extend the domains of the invariants a little further. Suppose that $F_n$ is a valuation of polytopes in $\mathbb R^n$ invariant under isometry. It is not hard to see that it extends to a valuation of polytopes of dimension at most $n$ in any vector space (with inner product), invariant under isometries between affine subspaces. One can go further and define it on Euclidean polytopes of dimension at most $n$ in the following more abstract sense. A \emph{linear Euclidean structure} on an $m$-simplex $\sigma$ is an equivalence class of affine embeddings $\sigma\to \mathbb R^m$ with respect to the action of isometries. Such a structure restricts to give such a structure on each face, and more generally on each linearly embedded simplex in $\sigma$. A \emph{simplicial Euclidean structure} on a triangulation of a polytope ($=$ compact PL space) $P$ consists of a choice of linear Euclidean structure on each simplex, compatible with restriction to faces. Such a structure determines another such structure on any refinement of the triangulation, which in turn determines it. A choice of such a structure for one, hence any finer, triangulation of $P$ is a \emph{Euclidean structure} on $P$, and $P$ is called a \emph{Euclidean polytope} if it is equipped with such a structure. One can also speak of constructible sets. This whole remark has a conical analogue, which we will not write out.
\end{rem}

\subsection{Intrinsic volume}

There is a canonical way of extending the $n$-dimensional volume $\mathcal V_n$ to a valuation of all polytopes in $\mathbb R^d$ for $d>n$. Let us construct it using the $\star$-product.

Begin with the volume invariant $\mathcal V$, a multiplicative $\mathbb R$-valued Euclidean polytope invariant located at $0$. Introduce a polynomial indeterminate $x$ and regard $\mathcal V$ as taking values in the larger ring $\mathbb R\lbrack x\rbrack$; it is still located at $0$. Scale this by $x$ to obtain the invariant $x\bullet \mathcal V$, again located at $0$. This is shorthand for the following bit of bookkeeping: for a polytope $P\subset \mathbb R^n$,
$$
(x\bullet \mathcal V)_n(P)=\mathcal V_n(P)x^n.
$$

When the dual volume $\mathcal W$, an $\mathbb R$-valued morphism with source $0$ and target $1$, is regarded as taking values in the larger ring $\mathbb R\lbrack x\rbrack$, then $\mathcal W\star (x\bullet \mathcal V)$ becomes defined. It is located at $1$. Explicitly, this absolute multiplicative $\mathbb R\lbrack x\rbrack$-valued polytope invariant is given by 
\begin{equation}\label{absvol}
(\mathcal W\star (x\bullet \mathcal V))(P)=\chi(0,P)+\chi(1,P)x+\chi(2,P)x^2+\dots
\end{equation}
where the coefficient of $x^n$ is a sum over all of the $n$-dimensional simplices of a cell structure of $P$:
\begin{equation}\label{intrinsic}
\chi(n,P)=\sum_{|\sigma|=n}\ \mathcal W(\nu(\sigma,P)) \mathcal V_n(\sigma).
\end{equation}
The statement that $\mathcal W\star (x\bullet \mathcal V)$ is multiplicative means that
$$
\chi(n,P\times Q)=\sum_{p+q=n}\ \chi(p,P)\chi(q,Q).
$$

The quantity $\chi(n,P)$ is called the \emph{intrinsic $n$-dimensional volume} of $P$. It is an $\mathbb R$-valued Euclidean valuation for polytopes. It is absolute (independent of the ambient vector space), and it coincides with $n$-dimensional volume $\mathcal V_n$ if $P$ is at most $n$-dimensional, so that in particular it vanishes on polytopes of dimension $<n$.

It also has the following homogeneity property: for any positive real number $\lambda$, $\chi(n,\lambda P)=\lambda^n\chi(n,P)$. 
\begin{rem}
By a theorem of Hadwiger \cite{Had}, for each choice of $0\le n\le d$ there is, up to multiplication by a constant, precisely one real-valued valuation $F$ of polytopes in $\mathbb R^d$ that depends continuously on the polytope and is homogeneous of degree $n$ in this sense.
\end{rem}
By substituting $0$ for the indeterminate $x$ in (\ref{absvol}) and using (\ref{Vchi}) we find that
$$
\chi(0,P)=\chi(P).
$$
Thus if $P$ is $d$-dimensional then the polynomial $(\mathcal W\star (x\bullet \mathcal V))(P)$ has degree $d$ and has the form
$$
\chi(P)+\dots +\mathcal V_d(P)x^d.
$$
If $M$ is a closed $2$-manifold (and a Euclidean polytope), then by (\ref{intrinsic}) $\chi(0,M)$ is the sum, over vertices $v$ of a triangulation, of $\mathcal W(\nu(v,M))=\frac{2\pi-\theta_v}{2\pi}$, where $\theta_v$ is the total angle around $v$. The fact that this equals $\chi(M)$ can be seen as a kind of Gauss-Bonnet theorem, with all of the curvature of $M$ being concentrated at the vertices.

If $M$ is a manifold of dimension $n+1$ then $\chi(n,M)=\frac{1}{2}\mathcal V_n(\partial M)$, because when $\sigma$ is a codimension one simplex in a manifold $M$ then $\mathcal W(\nu(\sigma,M))$ is $\frac{1}{2}$ or $0$ according to whether $\sigma$ is or is not contained in $\partial M$.

If $M$ is a closed $m$-manifold then $\chi(j,M)=0$ when $m-j$ is odd, because by Corollary \ref{oddbend} the dual volume of the normal cone of an interior simplex of odd codimension is zero. 
In fact, the same applies if low-dimensional singularities are allowed: if $M$ is locally an $m$-manifold except in a set of dimension $<j$ then $\chi(j,M)=0$ if $m-j$ is odd.

Here is a slightly different way of organizing the invariants $\chi(n,-)$. Define
$$
S^{x,y}=(x\bullet \mathcal W)\star (y\bullet \mathcal V).
$$
This is a multiplicative polytope invariant with values in $\mathbb R\lbrack x,y\rbrack$, located at $x$. The absolute invariant considered above is $S^{1,x}$. For $P\subset \mathbb R^d$ we have
$$
S^{x,y}_d(P)=\sum_{i+j=d}\chi(j,P)x^iy^j .
$$
Note that, because $S^{x,y}$ is not absolute, we cannot unambiguously write $S^{x,y}(P)$; the dimension of the ambient vector space is indicated by a subscript.

The following identities hold:
$$S^{0,y}=y\bullet \mathcal V$$
$$S^{x,0}=x\bullet \chi$$
$$c\bullet S^{x,y}=S^{cx,cy}$$
The first of these follows from $0\bullet \mathcal W=1_0$, and merely restates the fact that $\chi(d,P)=\mathcal V_d(P)$ for $P\subset\mathbb R^d$. The second follows from $0\bullet \mathcal V=0=x\bullet 0$, and restates the fact that $\chi(0,-)$ is the Euler characteristic.

\subsection{Intrinsic conical volume}

The conical volume invariant $\mathcal U$ can be given the same treatment. Since $\mathcal U$ is a morphism from $1$ to $0$, its scaling $x\bullet \mathcal U$ is a morphism from $x$ to $0$ and $\mathcal W\star (x\bullet \mathcal U)$ is a morphism from $x$ to $1$. Explicitly,
\begin{equation}\label{abssphvol}
(\mathcal W\star (x\bullet \mathcal U))(P)=\mathcal W(0,P)+\mathcal W(1,P)x+\mathcal W(2,P)x^2+\dots
\end{equation}
where the coefficient of $x^n$ is a sum over all of the $n$-dimensional conical simplices of a conical cell structure of $P$:
$$
\mathcal W(n,P)=\sum_{|\sigma|=n}\ \mathcal W(\nu(\sigma,P))\mathcal U_n(\sigma).
$$
Multiplicativity means 
$$
\mathcal W(n,P\times Q)=\sum_{p+q=n}\ \mathcal W(p,P)\mathcal W(q,Q).
$$
Call $\mathcal W(n,P)$ the \emph{intrinsic $n$-dimensional conical volume} of $P$. As a function of the cone $P$ this is an absolute $\mathbb R$-valued Euclidean valuation. When $P$ has dimension at most $n$ it coincides with $n$-dimensional conical volume. When the cone $P$ is an $(n+1)$-manifold it is one half the $n$-dimensional conical volume of the boundary. 
\begin{rem}\label{thing}
When the cone $M$ is an $m$-manifold without boundary (or even when it is locally an $m$-manifold except in a set of dimension $<j$) then $\mathcal W(j,M)=0$ if $m-j$ is odd. 
\end{rem}
We have
$$
\mathcal W(0,P)=\mathcal W(P).
$$
Define
\begin{equation}\label{defTxy}
T^{x,y}=(x\bullet \mathcal W)\star (y\bullet \mathcal U).
\end{equation}
This is a morphism from $y$ to $x$ with values in $\mathbb R\lbrack x,y\rbrack$. The absolute invariant \newline $\mathcal W\star (x\bullet \mathcal U)$ considered above is $T^{1,x}$. For a cone $P\subset \mathbb R^d$ we have
$$
T^{x,y}_d(P)=\sum_{i+j=d}\mathcal W(j,P)x^iy^j .
$$
The following identities hold:
$$T^{0,y}=y\bullet \mathcal U,$$
$$T^{x,0}=x\bullet \mathcal W,$$
$$c\bullet T^{x,y}=T^{cx,cy}$$
We also have 
$$T^{x,x}=(x\bullet \mathcal W)\star (x\bullet \mathcal U)=x\bullet (\mathcal W \star \mathcal U)=x\bullet 1_1=1_x.$$
This means that
\begin{equation}\label{1}
1=\mathcal W(0,P)+\mathcal W(1,P)+\mathcal W(2,P)+\dots
\end{equation}
for every nonempty cone $P$. We also have
$$T^{x,y}\star T^{y,z}=T^{x,z},$$
since $(y\bullet \mathcal U)\star (y\bullet \mathcal W)=y\bullet (\mathcal U\star \mathcal W)=y\bullet 1_0=1_0$. For the same reason,
$$T^{x,y}\star S^{y,z}=S^{x,z}$$
The inverse of $T^{x,y}$ is $T^{y,x}$, which means that if $DP$ is the dual of the convex cone $P$ in $\mathbb R^d$ then 
$$T^{x,y}_d(DP)=T^{y,x}_d(P).$$
Equating coefficients of $x^ny^{d-n}$, this says that 
\begin{equation}\label{d-n}
\mathcal W(n,DP)=\mathcal W({d-n},P).
\end{equation}
If $k$ is an $\mathbb R$-algebra and $a$ and $b$ belong to $k$, then we may write $T^{a,b}$ for the element of $\mathcal G_k$ obtained from $T^{x,y}$ by specializing the indeterminates $x$ and $y$ to $a$ and $b$, in other words by applying the associated ring homomorphism $\mathbb R\lbrack x,y\rbrack\to k$. (Of course, $T^{a,b}$ depends on the $\mathbb R$-algebra structure of $k$ as well as on $a$ and $b$.) These elements form a sub-groupoid of $\mathcal G^O_k$ that is trivial in the sense that there is a unique isomorphism between any two objects. We do not in fact know how to produce any elements of $\mathcal G^O_{\mathbb R}$ other than the elements $T^{a,b}$ for $a,b\in\mathbb R$. When $k$ is an $\mathbb R$-algebra in more than one way, for example when $k=\mathbb R\otimes_{\mathbb Z} \mathbb R$, then there is a richer supply of morphisms, namely various kinds of generalized Dehn invariants. Before discussing these, we examine some consequences of the results of \S\ref{e star}.

\subsection{Volume and the duality/interior relation}\label{van}

By Example \ref{e}
\begin{equation}\label{Te}
e=T^{1,-1},
\end{equation}
and therefore more generally $x\bullet e=T^{x,-x}$. This is valid for any $x\in k$ when $k$ is any $\mathbb R$-algebra.

From Proposition \ref{ep} 
$$
S^{x,y}I=(x\bullet e)\star (-\bullet S^{x,y})=T^{x,-x}\star S^{-x,-y}=S^{x,-y}.
$$
Comparing coefficients of $x^iy^j$, we see that
$$
\chi(j,I\xi)=(-1)^j\chi(j,\xi)
$$
for every $\xi\in \mathcal P(V)$. This demonstrates again the vanishing of $\chi(m-2k-1,M)$ for a closed $m$-manifold.

Likewise from Proposition \ref{et} we have
$$
T^{x,y}I=T^{x,-y},
$$
so that
$$
\mathcal W(j,I\xi)=(-1)^j\mathcal W(j,\xi)
$$
for every $\xi\in \Sigma(V)$.

Equation (\ref{Te}) itself states that for any conical polytope $P$
$$
\mathcal W(0,P)-\mathcal W(1,P)+\mathcal W(2,P)-\dots =e(P).
$$
Combining this with Equation (\ref{1}) we obtain, for a nonempty cone $P$,
\begin{equation}\label{evencone}
\mathcal W(0,P)+\mathcal W(2,P)+\mathcal W(4,P)+\dots=\frac{1+e(P)}{2}
\end{equation}
and
\begin{equation}\label{oddcone}
\mathcal W(1,P)+\mathcal W(3,P)+\mathcal W(5,P)+\dots =\frac{1-e(P)}{2}.
\end{equation}

\subsection{Dehn invariants}\label{dehn}

Consider the two ring maps from $\mathbb R$ to $\mathbb R\otimes\mathbb R$
$$
a\mapsto a\otimes 1
$$$$
a\mapsto 1\otimes a.
$$
Apply the first of these to $x\bullet \mathcal W$ to make an $(\mathbb R\otimes\mathbb R)\lbrack x\rbrack$-valued morphism with source $0$ and target $x$,
$$
((x\bullet \mathcal W)\otimes 1)_n(P)=(x\bullet \mathcal W)_n(P)\otimes 1=(\mathcal W(P)\otimes 1)x^n.
$$
Apply the second to $y\bullet \mathcal V$ to make an $(\mathbb R\otimes\mathbb R)\lbrack y\rbrack$-valued multiplicative polytope invariant located at $0$,
$$
(1\otimes (y\bullet \mathcal V))_n(P)=1\otimes (y\bullet \mathcal V)_n(P)=(1\otimes \mathcal V_n(P))y^n.
$$
The $\star$-product now yields an $(\mathbb R\otimes\mathbb R)\lbrack x,y\rbrack$-valued multiplicative polytope invariant 
$$
\tilde S^{x,y}=((x\bullet \mathcal W)\otimes 1)\star (1\otimes (y\bullet \mathcal V)).
$$
Thus
$$
\tilde S^{x,y}_n(P)=\sum_{i+j=n}\ \tilde\chi(j,P)x^iy^j,
$$
 where 
 $$
 \tilde\chi(j,P)=\sum_{|\sigma|=j}\ \mathcal W(\nu(\sigma,P))\otimes \mathcal V_j(\sigma)\in \mathbb R\otimes\mathbb R.
 $$
The multiplication map $\mathbb R\otimes\mathbb R\to\mathbb R$ takes $\tilde S^{x,y}$ to $S^{x,y}$. That is, it takes $\tilde\chi(j,P)$ to the intrinsic volume $\chi(j,P)$ .

If $P$ is $d$-dimensional then $\tilde\chi(j,P)=0$ for $j>d$. Both $\tilde\chi(d,P)$ and $\tilde\chi(d-1,P)$ belong to $\mathbb Q\otimes \mathbb R$, because when the normal cone of a cell is at most one-dimensional then its dual volume is rational (in fact, one half of an integer). Therefore
$$
\tilde\chi(d,P)=1\otimes \chi(d,P)=1\otimes \mathcal V_d(P)
$$
and
$$
\tilde\chi(d-1,P)=1\otimes \chi(d-1,P).
$$
Likewise 
$$\tilde\chi(0,P)=\chi(0,P)\otimes 1.$$
Since $\chi(0,P)=\chi(P)$, we may also say:
$$\tilde\chi(0,P)=\chi(P)\in \mathbb Z=\mathbb Z\otimes \mathbb Z\subset \mathbb R\otimes \mathbb R.
$$
In dimension three $\tilde\chi(1,-)$ is essentially the original Dehn invariant: if $P$ is a $3$-dimensional nonempty convex polytope, then $\tilde\chi(1,P)$ is the unique element of $\mathbb R\otimes \mathbb R$ that maps to the negative of the Dehn invariant of $P$ under the quotient map  $\mathbb R\otimes \mathbb R\to (\mathbb R/\mathbb Q)\otimes\mathbb R$ and maps to $\mathcal V_3(P)$ under the multiplication map $\mathbb R\otimes \mathbb R\to \mathbb R$.

Similar definitions in the conical case give:
$$
\tilde T^{x,y}=((x\bullet \mathcal W)\otimes 1)\star (1\otimes (y\bullet \mathcal U))
$$
and
 $$
\tilde T^{x,y}_n(P)=\sum_{i+j=n}\  \tilde {\mathcal W}(j,P)x^iy^j,
$$
 where 
 $$
 \tilde {\mathcal W}(j,P)=\sum_{|\sigma|=j} \mathcal W(\nu(\sigma,P))\otimes \mathcal U_j(\sigma)\in \mathbb R\otimes\mathbb R.
 $$
The multiplication map $\mathbb R\otimes\mathbb R\to\mathbb R$ takes $\tilde T^{x,y}$ to $T^{x,y}$. That is, it takes $ \tilde {\mathcal W}(j,P)$ to the intrinsic conical volume $\mathcal W(j,P)$ . 
 
If $P$ is $d$-dimensional then $ \tilde {\mathcal W}(j,P)=0$ for $j>d$, and
$$
 \tilde {\mathcal W}(d,P)=1\otimes \mathcal W(d,P)=1\otimes \mathcal U_d(P)
$$$$
 \tilde {\mathcal W}(d-1,P)=1\otimes \mathcal W(d-1,P).
$$
We always have
$$ \tilde {\mathcal W}(0,P)=\mathcal W(0,P)\otimes 1=\mathcal W(P)\otimes 1.
$$
The inverse of $\tilde T^{x,y}$ is related to $\tilde T^{y,x}$ by the map $a\otimes b\mapsto b\otimes a$ from $\mathbb R\otimes \mathbb R$ to itself. Therefore if $DP$ is the dual of the convex cone $P$ in $\mathbb R^d$ then $ \tilde {\mathcal W}(j,DP)$ and $ \tilde {\mathcal W}(d-j,P)$ are related by that map.
Thus in three dimensions in addition to
$$
 \tilde {\mathcal W}(3,P)=1\otimes \mathcal W(3,P),
$$ 
$$
\tilde {\mathcal W}(2,P)=1\otimes \mathcal W(2,P),
$$
and
$$
\tilde {\mathcal W}(0,P)=\mathcal W(0,P)\otimes 1
$$
we have
$$
 \tilde {\mathcal W}(1,P)=\mathcal W(2,DP)\otimes 1=\mathcal W(1,P)\otimes 1
$$
The first case in which $\tilde{\mathcal W}(j,-)$ takes values in neither $\mathbb R\otimes \mathbb Q$ nor $\mathbb Q\otimes \mathbb R$ is $ \tilde {\mathcal W}(2,P)$ for $4$-dimensional $P$; this is essentially the usual $3$-dimensional spherical Dehn invariant.

\begin{example}\label{K}
A related construction is $K=(\mathcal U\otimes 1)\star (1\otimes \mathcal W).$ This is an $(\mathbb R\otimes \mathbb R)$-valued morphism from $0$ to $0$. The multiplication map $\mathbb R\otimes \mathbb R\to \mathbb R$ sends it to $\mathcal U\star\mathcal W=0=1_0$. It is related to its inverse by the ring map $a\otimes b\mapsto b\otimes a$. If $\sigma_\theta$ is a conical $2$-simplex with angle $\theta$ then $K_2(\sigma_\theta)= \frac{\theta}{2\pi}\otimes 1-1\otimes \frac{\theta}{2\pi}$.
\end{example}

\subsection{Higher Dehn invariants}\label{total dehn}

Write
$$
T^{x,y,z}=(T^{x,y}\otimes 1)\star  (1\otimes T^{y,z}).
$$
This is the $\star$-product of two morphisms with values in
$(\mathbb R\otimes\mathbb R)\lbrack x,y,z\rbrack$, one from $z$ to $y$ and one from $y$ to $x$. 
Note that $T^{x,0,y}=\tilde T^{x,y}$, and that $T^{0,1,0}=K$. We have also the identity 
$$
T^{x,y,z}=(x\bullet (\mathcal W\otimes 1))\star (y\bullet K)\star (z\bullet (1\otimes U)).
$$ 

Likewise define
$$
T^{x,y,z,w}=(T^{x,y}\otimes 1\otimes 1)\star  (1\otimes T^{y,z}\otimes 1)\star (1\otimes 1\otimes T^{z,w}).
$$
This takes values in $(\mathbb R\otimes\mathbb R\otimes \mathbb R)\lbrack x,y,z,w\rbrack$. It may be written
$$
T^{x,y,z,w}=(x\bullet (\mathcal W\otimes 1\otimes 1))\star (y\bullet (K\otimes 1))\star (z\bullet (1\otimes K))\star (w\bullet (1\otimes 1\otimes U)).
$$
In general we obtain an $(\mathbb R^{\otimes r})\lbrack x_0,\dots ,x_r\rbrack$-valued morphism $T^{x_0,\dots ,x_r}$ with source $x_r$ and target $x_0$. 
The correct convention when $r$ is zero is $T^x=x\bullet\epsilon=1_x$. That is, $T^x$ is the identity morphism of $x$ in $\mathcal G_{\mathbb Z\lbrack x\rbrack}$.

Specializing the indeterminates to elements of a ring $k$, we may also regard $T^{x_0,\dots ,x_r}$ as denoting a $k$-valued morphism, but only when $k$ is also equipped with a ring map from $\mathbb R^{\otimes r}$, that is, an ordered $r$-tuple of $\mathbb R$-algebra structures.

Every Euclidean morphism that can be made from $\mathcal U$ by a combination of scaling, groupoid composition, groupoid inversion, and precomposing an invariant with a ring map can be described as a specialization of one of these invariants $T^{x_0,\dots ,x_r}$ in this way. A central question is whether such invariants are enough to detect everything: if $P$ and $Q$ are cones in $V$ such that $T^{x_0,\dots,x_r}_n(P)=T^{x_0,\dots,x_r}_n(Q)$ for all $r$, does it follow that $F(P)=F(Q)$ for every Euclidean valuation $F$ of cones in $V$?

There are relations between these invariants:

Repeating an indeterminate corresponds to inserting a $1$; for example, $T^{x,y,y,z}$ is obtained from $T^{x,y,z}$ by the ring map $\xi\otimes \eta\mapsto \xi\otimes 1\otimes \eta$. Another way to say this is that setting two consecutive indeterminates equal has the effect of making one of the $\mathbb R$-algebra structures irrelevant.

Contracting two consecutive copies of $\mathbb R$ (that is, identifying two of the $\mathbb R$-algebra structures) has the effect of deleting an indeterminate; for example, the ring map $\xi\otimes \eta\otimes \zeta\mapsto \xi\eta\otimes \zeta$ takes $T^{x,y,z,w}$ to $T^{x,z,w}$.

The fact that $T^{1,-1}=e$ is $\mathbb Z$-valued leads to more relations. Since $T^{x,-x}$ takes values in $\mathbb Z\lbrack x\rbrack$, not merely in $\mathbb R\lbrack x\rbrack$, $T^{x,-x}\otimes 1$ is the same as $1\otimes T^{x,-x}$. Thus, for example, 
since
$$
(T^{x,y}\otimes 1)\star (T^{y,-y}\otimes 1)\star (1\otimes T^{-y,z})
=(T^{x,y}\otimes 1)\star (1\otimes T^{y,-y})\star (1\otimes T^{-y,z}),
$$
we have
$$
T^{x,-y,z}=T^{x,y,z}.
$$
More generally, the polynomial $T^{x_0,\dots ,x_r}$ is even with respect to all variables $x_i$ with $0<i<r$. 

Using these relations, one can work out that when $P$ is a cone in $\mathbb R^n$ then the elements $T_n^{x_0,\dots ,x_r}(P)$ for all $r\ge 0$ are determined by those for which $2r\le n$.

Similarly there are polytope invariants $S^{x_0,\dots ,x_r}$, with values in $(\mathbb R^{\otimes r})\lbrack x_0,\dots ,x_r\rbrack$. For example
$$
S^{x,y,z,w}=(T^{x,y}\otimes 1\otimes 1)\star  (1\otimes T^{y,z}\otimes 1)\star (1\otimes 1\otimes S^{z,w}).
$$
$S^{x,y}$ is as before. $S^{x,0,y}=\tilde S^{x,y}$. There is no $S^x$.

Every multiplicative Euclidean polytope invariant that can be made from $\mathcal V$ by a combination of scaling, transport by $\mathcal U$ or its inverse, and precomposing an invariant with a ring map can be described as a specialization of one of these invariants $S^{x_0,\dots ,x_r}$ in this way. Again we can ask whether such invariants are enough to detect everything.

There are relations among these similar to the relations among the $T^{x_0,\dots,x_r}$.

\section{The graded rings $E$ and $L$}

The sets $\mathcal O_k$, $\mathcal G_k$, and $\mathcal I_k$ depend functorially on the ring $k$; we have already made informal use of this. Moreover, the following maps defined in \S\ref{star} are natural: the multiplication $\mathcal O_k\times \mathcal O_k\to\mathcal O_k$; the source, target, identity, and location maps \newline
$s:\mathcal G_k\to \mathcal O_k$, $t:\mathcal G_k\to \mathcal O_k$, $1:\mathcal O_k\to \mathcal G_k$, $p:\mathcal I_k\to \mathcal O_k$; the two $\star$-products
$$
\mathcal G_k\times _{\mathcal O_k}\mathcal G_k\to \mathcal G_k
$$$$
\mathcal G_k\times _{\mathcal O_k}\mathcal I_k\to \mathcal I_k;
$$
the inversion map $\mathcal G_k\to \mathcal G_k$; and the scaling maps $\mathcal O_k\times \mathcal G_k\to \mathcal G_k$ and $\mathcal O_k\times \mathcal I_k\to \mathcal I_k$.

It is clear that the three functors are representable. For example, to construct a ring $R$ such that there is a natural bijection between ring maps $R\to k$ and elements of $\mathcal G_k$ one may take the abelian group $\bigoplus_V \Sigma(V)$, form its symmetric algebra over $\mathbb Z$, and let $R$ be the quotient of this by an ideal that enforces multiplicativity. 

In algebraic geometry language, we may say that $\mathcal O$ is an affine monoid scheme, that $\mathcal G$ is an affine groupoid scheme with an action of $\mathcal O$, and so on.

We will hardly use this language again. However, in the Euclidean-invariant case we will find it useful to work directly with the representing rings. We now introduce notation for those.

\subsection{The ring $E$}

We describe a graded ring $E=\bigoplus_n E_n$ such that (Proposition \ref{E rep}) elements of $\mathcal I^O_k$ correspond precisely to ring maps $E\to k$.

For each $n\ge 0$, let $E_n$ be the group of coinvariants for the action of the orthogonal group $O_n$ on McMullen's group $\Pi(\mathbb R^n)$ (equivalently, the group of coinvariants for the action of affine isometries on the polytope group $\mathcal P(\mathbb R^n)$). Thus for any abelian group $A$ the group homomorphisms $E_n\to A$ are in natural bijection with the 
$A$-valued Euclidean-invariant valuations for polytopes in $\mathbb R^n$. 

Denote the universal such valuation by $P\mapsto \lbrack P\rbrack_n\in E_n$; in other words, let $ \lbrack P\rbrack_n$ be the image in $E_n$ of $(P)\in \mathcal P(\mathbb R^n)$. We write $\lbrack P\rbrack_n$ more generally when $P$ is a Euclidean polytope (or a constructible subset of one) having dimension at most $n$, either in an arbitrary Euclidean space or in the abstract sense of Remark \ref{abstract}. 
\begin{rem}\label{}
There are many other descriptions of $E_n$. For example, if $d\ge n$ then $E_n$ is the group of coinvariants for the action of Euclidean isometries of $\mathbb R^d$ on $\mathcal P_n(\mathbb R^d)$. It is also the colimit, over the category of Euclidean polytopes of dimension $\le n$ and isometric PL embeddings, of the functor $X\mapsto \mathcal P(X)$. It is also the colimit, over the category of all Euclidean polytopes and isometric PL embeddings, of $X\mapsto \mathcal P_n(X)$.
\end{rem}
The groups $E_n$ form a graded ring $E$ in which multiplication is given by 
$$
\lbrack P\rbrack_i\lbrack Q\rbrack_j=\lbrack P\times Q\rbrack_{i+j}.
$$
$E_0$ is isomorphic (as a ring) to $\mathbb Z$, with the element $1\in E_0$ being given by a point. 

The following is clear:
\begin{prop}\label{E rep}
For every element $F=\lbrace F_n\rbrace\in \mathcal I^O_k$ there is a unique ring map $F:E\to k$ such that for every Euclidean polytope $P$ of dimension $\le n$
$$F_n(P)=F(\lbrack P\rbrack_n).
$$
\end{prop}
Note that we use the same name for the invariant and the corresponding ring map.

We show that the group $E_1$ is isomorphic to $\mathbb Z\times \mathbb R$. Let $p\in E_1$ be the element given by a point. Let $p+q(\lambda)$ be the element given by a closed line segment of length $\lambda$; that is, let $q(\lambda)$ be the element given by a half-open segment of that length. Since $q(\lambda+\mu)=q(\lambda)+q(\mu)$, this gives a group homomorphism $q:\mathbb R\to E_1$. Every element of $E_1$ is uniquely of the form $kp+q(\lambda)$, $k\in\mathbb Z$, $\lambda\in \mathbb R$. In fact, $E_1$ is generated by points and closed line segments, and the uniqueness holds because the maps $\chi:E_1\to\mathbb Z$ and $\mathcal V_1:E_1\to \mathbb R$ take $kp+q(\lambda)$ to $k$ and $\lambda$ respectively:
$$
\chi(p)=1 \hskip .6 in \chi(q(\lambda)=0
$$$$
\mathcal V(p)=0 \hskip .6 in\mathcal V(q(\lambda)=\lambda.
$$

Multiplication by $p$ in the graded ring $E$ is the same as the ``inclusion'' map from $E_n$ to $E_{n+1}$
$$
\lbrack P\rbrack_n\mapsto \lbrack P\rbrack_{n+1}.
$$
The quotation marks will become less important when we see in \S \ref{E split} that this map is an injection, in other words that $p$ is not a zero-divisor in $E$. 

Each group $E_n$ inherits an involution $I$ from $\mathcal P(\mathbb R^n)$, and the total involution \newline
$I:E\to E$ is a graded ring map. It follows from (\ref{intM}) that if $M$ is an $m$-dimensional manifold (and a Euclidean polytope) with $m\le n$ then
$$
I\lbrack M\rbrack_n=(-1)^m\lbrack \text{int}M\rbrack_n.
$$
It follows from (\ref{IVtriang}) that if $P$ is any Euclidean polytope of dimension at most $n$ then
$$
I\lbrack P\rbrack_n=\sum_\sigma\ (-1)^{|\sigma|}\lbrack\sigma\rbrack_n
$$
for any cell structure on $P$.

We record the action of $I$ on named elements of $E$:
$$I(p)=p  \hskip .5 inI(q(\lambda))=-q(\lambda).
$$

The dimension filtration that $E_n$ inherits from $\mathcal P(\mathbb R^n)$ is
$$
p^nE_0\subset p^{n-1}E_1\subset \dots \subset pE_{n-1}\subset E_n.
$$
Denote the top quotient $E_n/pE_{n-1}$ by $\hat E_n$. This is the $n$-dimensional Euclidean scissors congruence group. By Proposition \ref{topI} the involution of $\hat E_n$ induced by $I$ is multiplication by $(-1)^n$.

Note that when a ring map $F:E\to k$ is interpreted as an element of $\mathcal I^O_k$ then its location is $F(p)$.

The groups $\hat E_n$ form a graded ring $\hat E=E/pE$. Ring maps $F:\hat E\to k$ correspond to those elements of $\mathcal I^O_k$ which are located at $0$, i.\ e.\ those multiplicative Euclidean invariants $F=\lbrace F_n\rbrace$ such that $F_n$ vanishes on polytopes of dimension $<n$.

Another quotient of $E$ is the ungraded ring $E_\infty=E/(1-p)E$. As a group, $E_\infty$ can be described as the direct limit of $E_n$ with respect to inclusion. Ring maps $F:E_\infty\to k$ correspond to those elements of $\mathcal I^O_k$ which are located at $1$, i.\ e.\ those multiplicative Euclidean invariants $F=\lbrace F_n\rbrace$ which are absolute (those for which $F_n(P)$ is independent of $n$).

The group $E_2$ can be described using the classical and elementary fact (\cite{Du}, p.\ 3) that two polygons of equal area in the plane are scissors congruent. This means that $\hat E_2$ is generated by rectangles of height $1$, so that the map $\mathbb R\to E_2$ given by
$$
\mu\mapsto q(\mu)q(1)
$$
is surjective modulo $pE_1$. It follows that the map $\mathbb Z\times \mathbb R\times \mathbb R\to E_2$ given by
$$
(k,\lambda,\mu)\mapsto kp^2+pq(\lambda)+q(\mu)q(1)
$$
is surjective. This last map is in fact an isomorphism, because a left inverse is provided by the homomorphisms 
$$
\chi:E_2\to \mathbb Z
$$$$
\chi(1,-):E_2\to \mathbb R
$$$$
\mathcal V_2:E_2\to \mathbb R.
$$
(Here $\chi(1,-)$ is the intrinsic $1$-dimensional volume introduced in \ref{}, which takes a $2$-dimensional convex polygon to one half of its perimeter.)

\subsubsection{Splitting the dimension filtration}\label{E split}
In discussing $E_0$, $E_1$, and $E_2$ we showed how to split the ``inclusion map'' $E_{n-1}\to E_{n}$ for $n\le 2$ using absolute volume invariants. The same method succeeds for $n=3$. To split it for all $n$ we can use the method of \S\ref{M split}.

In fact, the desired splitting follows from McMullen's isomorphism (\ref{Msplitting}); the latter is compatible with the canonical $O_n$-actions, and upon passing to coinvariants it yields an isomorphism
$$ 
E_n\cong\bigoplus_{0\le j\le n}\hat E_j.
$$
It is worth examining this splitting directly in the Euclidean-invariant context.

Let $\tau\in \mathcal I^O_{\hat E}$ be given by the quotient ring map $E\to \hat E$. Thus for a polytope $P\subset \mathbb R^n$ the element $\tau_n(P)\in \hat E_n$ is the Euclidean scissors congruence class of $P$. (This is the universal example of a multiplicative Euclidean polytope invariant located at $0$.) The groups $\hat E_j$, $j>0$, have real vector space structures (inherited from the groups $\hat \Pi(\mathbb R^j)$) such that the multiplication $\hat E_i\times \hat E_j\to \hat E_{i+j}$ is $\mathbb R$-bilinear. Thus the ring $\hat E$ becomes an $\mathbb R$-algebra $\hat E_{\mathbb R}$ when $\hat E_0\cong \mathbb Z$ is replaced by $\mathbb R$. Now the $\star$-product $(x\bullet \mathcal W)\star\tau\in \mathcal I^O_{\hat E_{\mathbb R}\lbrack x\rbrack}$ is defined (by regarding both $x\bullet \mathcal W\in \mathcal I^O_{\mathbb R\lbrack x\rbrack}$ and $\tau\in \mathcal I^O_{\hat E}$ as taking values in $\hat E_{\mathbb R}\lbrack x\rbrack$), and the resulting ring map $E\to \hat E_{\mathbb R}\lbrack x\rbrack$ goes into $\hat E\lbrack x\rbrack$. It becomes a graded ring map if $\hat E$ is graded in the usual way and $x$ is assigned degree one.

Explicitly, the coefficient of $x^{n-i}$ in the polynomial $((x\bullet \mathcal W)\star\tau)\lbrack P\rbrack_n$ is the following element of $\hat E_i$:
$$
\alpha_i(P)=\sum_{|\sigma|=i}\ \mathcal W(\nu(\sigma,P))\tau_i(\sigma),
$$
with summation over the $i$-simplices of a triangulation of $P$. In the exceptional case $i=0$ it is 
$$
\alpha_0(P)=\sum_{|\sigma|=0}\ \mathcal W(\nu(\sigma,P))=\chi(P)\in \hat E_0=\mathbb Z.
$$ 

Note that the element $p$ of $E$ is mapped to the element $x$ of $\hat E\lbrack x\rbrack$.
\begin{prop}\label{alpha}
This graded ring map $E\to \hat E\lbrack x\rbrack$ is an isomorphism.
\end{prop}
\begin{proof}
This amounts to the assertion that for every $n$ the map $E_n\to \bigoplus_{0\le i\le n}\hat E_i$ given by
$$
\xi\mapsto \sum_{i\le n}\alpha_i(\xi)
$$
is a group isomorphism. To verify this assertion, note that these maps constitute a map of direct limit systems from the diagram of ``inclusion maps'' 
$$
E_0\to E_1\to E_2\to \dots
$$
to the diagram of inclusion maps
$$
\hat E_0\to \hat E_0\oplus \hat E_1\to \hat E_0\oplus \hat E_1\oplus \hat E_2\to\dots .
$$
This map is an isomorphism at the $n$th level by induction on $n$, because the induced map from the cokernel of $E_{n-1}\to E_n$ to the cokernel below is the obvious isomorphism. (That is: when $n=i$, the map $\alpha_i=\alpha_n:E_n\to \hat E_n$ is the quotient map.)
\end{proof}
Note that $\alpha_i$ is an absolute Euclidean-invariant valuation with values in $\hat E_i$. The element $\alpha_i(P)\in \hat E_i$ may be thought of as the purely $i$-dimensional part of the polytope $P$.

Passing to the direct limit, we obtain a group isomorphism $E_\infty\to \hat E$. It is the same as the ring isomorphism obtained from the isomorphism $E\to \hat E\lbrack x\rbrack$ by setting $p=1$ in one ring and $x=1$ in the other.
\begin{rem}
By Proposition \ref{ep} and the equation $\mathcal W=e\star -\bullet \mathcal W$ we have
$$
(x\bullet \mathcal W)\star \tau=(x\bullet (e\star -\bullet \mathcal W))\star \tau= (x\bullet e)\star(-x\bullet \mathcal W)\star\tau=((x\bullet \mathcal W)\star(-\bullet\tau))I.
$$
Equating coefficients of $x^n$, this says 
\begin{equation}\label {alpha I}
\alpha_n=(-1)^n\alpha_n\circ I.
\end{equation}
\end{rem}

\subsection{The ring $L$}

We describe a graded ring $L=\bigoplus_n L_n$ such that elements of $\mathcal G^O_k$ correspond precisely to ring maps $L\to k$.

Let $L_n$ be the group of coinvariants for the action of the orthogonal group $O_n$ on the cone group $\Sigma(\mathbb R^n)$. Thus for any abelian group $A$ the homomorphisms $L_n\to A$ are in natural bijection with the $A$-valued Euclidean-invariant valuations of cones in $\mathbb R^n$. 

Denote the universal such valuation by $P\mapsto \langle P\rangle_n\in L_n$; in other words let $\langle P\rangle_n$ be the image in $L_n$ of $(P)\in \Sigma(\mathbb R^n)$. This extends from cones in $\mathbb R^n$ to Euclidean cones (or conical constructible subsets thereof) having dimension at most $n$, either cones in a Euclidean vector space or Euclidean cones in the abstract sense of Remark \ref{}. 
\begin{rem}\label{other}
There are many other descriptions of $L_n$. For example, if $d\ge n$ then it is the group of coinvariants for the action of  isometries of $O_d$ on $\Sigma_n(\mathbb R^d)$. This will be used below with $n=d+1$. It is also the colimit, over the category of pointed Euclidean polytopes of dimension $\le n$ and germs of isometric PL embeddings, of $(X,x)\mapsto \mathcal P(X,X-x)$.
\end{rem}
These groups form a graded ring $L$ in which multiplication is given by 
$$
\langle P\rangle_i\langle Q\rangle_j=\langle P\times Q\rangle_{i+j}.
$$
As a ring $L_0$ is isomorphic to $\mathbb Z$, the element $1$ being given by a point. 

The following is clear:
\begin{prop}
For every element $F=\lbrace F_n\rbrace\in \mathcal G^O_k$ there is a unique ring map $F:L\to k$ such that for every Euclidean cone $P$ of dimension $\le n$ 
$$F_n(P)=F(\langle P\rangle_n).
$$
\end{prop}
It is easy to describe $L_1$. Up to isometry there are three nonempty cones in the line $\mathbb R^1$: the origin, a closed half-line, and the whole line. Denote the corresponding elements of $L_1$ by $t$, $d$, and $s$. They satisfy $t+s=2d$. A $\mathbb Z$-basis for $L_1$ is given by $t$ and $d$. Another useful $\mathbb Z$-basis is given by $d$ together with the element $d'=d-t=s-d$ given by an open half-line (a conical constructible set). Thus $t=d-d'$ and $s=d+d'$.

Multiplication by $t$ gives the ``inclusion'' map $L_n\to L_{n+1}$ 
$$
\langle P\rangle_n \mapsto \langle P\rangle_{n+1}.
$$
(The quotation marks serve to remind us that it is not known whether $t$ is a zero-divisor in $L$.) The group $L_d$ inherits a dimension filtration from $\Sigma(\mathbb R^d)$, which may be written 
$$
t^dL_0\subset t^{d-1}L_1\subset \dots \subset tL_{d-1}\subset L_d.
$$
Denote the top quotient $L_d/tL_{d-1}$ by $\hat L_d$. For $d\ge 1$ this is the $(d-1)$-dimensional \emph{spherical scissors congruence group}. These form a graded ring $L/tL=\hat L=\bigoplus_d\hat L_d$

The quotient map  $L\to \hat L$ gives an element of $\mathcal G^O_{\hat L}$ that is the universal Euclidean morphism with target $0$. The quotient map  $L\to L/(1-t)L=L_\infty$ gives an element of $\mathcal G^O_{L_\infty}$ that is the universal Euclidean morphism with target $1$.

For every $n> 1$ the composed ``inclusion'' map $L_1\to L_n$ given by multiplication by $t^{n-1}$ is injective. In fact, it is split injective; the $\mathbb Z$-valued invariants $\epsilon$ and $e$ provide a group homomorphism $L_n\to \mathbb Z^2$ such that the composed map $L_1\to \mathbb Z^2$ sends a $\mathbb Z$-basis to a $\mathbb Z$-basis:
$$
t\mapsto (\epsilon(t^n),e(t^n))=(\epsilon(t)^n,e(t)^n)=(1,1)
$$
$$
d\mapsto (\epsilon(t^{n-1}d),e(t^{n-1}d))= (\epsilon(t)^{n-1}\epsilon(d),e(t)^{n-1}e(d))=(1,0).
$$
Thus for $n\ge 1$ the group $L_n$ is the direct sum of two subgroups. One is a free abelian group $t^{n-1}L_1$ generated by the point and half-line elements $t^n$ and $t^{n-1}d$, and the other is $ker(\epsilon)\cap ker(e)$.

In the case of $L_2$ this second summand is isomorphic to $\mathbb R$. We parametrize it as follows. A conical $2$-simplex (sector) with angle $\theta$ gives an element of $L_2$ which we call $td+a(\theta)$. That is, for $0<\theta<\pi$ we define $a(\theta)\in L_2$ to be the difference between this sector and a closed half-line. Since $a(\theta+\phi)=a(\theta)+a(\phi)$, this produces an additive homomorphism $a:\mathbb R\to L_2$. Its image is contained in $ker(\epsilon)\cap ker(e)$. It maps onto the quotient $\hat L_2$ because the latter is generated by conical $2$-simplices. The map $a$ is injective because the conical volume satisfies
$$
\mathcal U(a(\theta))=\frac{\theta}{2\pi}.
$$
Thus
$$
L_2=t^2\mathbb Z\oplus td\mathbb Z\oplus a(\mathbb R)\cong \mathbb Z\oplus \mathbb Z\oplus \mathbb R.
$$
If $P$ is the cone corresponding to the germ of a surface (with Euclidean structure) at a point where the total angle is $\theta$, then $\langle P\rangle_2=t^2+a(\theta)$. 

In general the summand $ker(\epsilon)\cap ker(e)$ of $L_n$ is $2$-divisible. To see this, view it as the cokernel of ``inclusion'' $L_1\to L_n$. The result follows by induction from the fact (\cite{Du} p.\ 10) that for $n\ge 2$ the group $\hat L_n=L_n/tL_{n-1}$ is $2$-divisible. (Note that these $2$-divisible groups are not in general known to be uniquely $2$-divisible.) As a result we can describe the graded ring $L/2L=L\otimes (\mathbb Z/2\mathbb Z)$: 
\begin{prop}\label{L/2}
For $n>0$ a $\mathbb Z/2\mathbb Z$-basis for $L_n/2L_n$ is given by the point and the half-line. As a ring $L/2L$ is generated by $t$ and $d$ subject  to the relations $td=d^2$ and $2=0$.
\end{prop}

We record the values of the basic invariants on named elements (some of these have been used already):
$$
\epsilon(t)=1\hskip .2 in\epsilon(d)=1\hskip .2 in\epsilon(s)=1\hskip .2 in\epsilon(d')=0
$$$$
\epsilon(a(\theta))=0
$$
$$
e(t)=1\hskip .2 in e(d)=0 \hskip .2 in e(s)=-1\hskip .2 in e(d')=-1
$$
$$
e(a(\theta))=0
$$
$$
\mathcal U(t)=0\hskip .2 in \mathcal U(d)=1/2 \hskip .2 in \mathcal U(s)=1\hskip .2 in \mathcal U(d')=1/2
$$
$$
\mathcal U(a(\theta))=\frac{\theta}{2\pi}.
$$
$$
\mathcal W(t)=1\hskip .2 in \mathcal W(d)=1/2 \hskip .2 in \mathcal W(s)=0\hskip .2 in \mathcal W(d')=-1/2
$$$$
\mathcal W(a(\theta))=-\frac{\theta}{2\pi}.
$$

$L_n$ inherits an interior operator $I$ and a duality operator $D$ from $\Sigma(\mathbb R^n)$.These are involutions of the graded ring $L$. By Lemma \ref{DI} we have $DIDI\xi=(-1)^d\xi$ for $\xi\in L_d$.

We record the values of the involutions on the named elements of $L_1$ and $L_2$:
$$
It=t\hskip .2 in  Is=-s \hskip .2 in  Id=-d' \hskip .2 in Id'=-d.
$$$$
Ia(\theta)=a(\theta)
$$
$$
Dt=s\hskip .2 in Ds=t\hskip .2 in  Dd=d\hskip .2 in  Dd'=-d'
$$$$
Da(\theta)=-a(\theta).
$$
Note that
$$dd'=d^2-dt=a(\pi/2)$$
$$d(s-t)=2dd'=a(\pi)$$
$$s^2-t^2=(s+t)(s-t)=2d(s-t)=a(2\pi).$$

Each of the rings $\hat L$ and $L_\infty$ inherits an involution $I$ from $L$, because $It=t$. Neither of them inherits a duality involution from $L$, because the element $Dt=s$ does not belong to $tL$ and the element $D(1-t)=1-s$ does not belong to $(1-t)L$. In fact, $D$ produces an isomorphism $\hat L\cong L_\infty$.

The involution $I$ preserves the dimension filtration of $L_n$. The involution $D$ takes the dimension filtration to a different filtration
$$
s^dL_0\subset s^{d-1}L_1\subset \dots \subset sL_{d-1}\subset L_d,
$$
which might be called the suspension filtration. Multiplication by $s$ corresponds to sending a cone $P\subset V$ to the cone $P\times \mathbb R\subset V\times\mathbb R$. The element $s^d\in L_d$ is given by the whole of $\mathbb R^d$.

\subsection{Rigid invariants}

Call a Euclidean invariant of cones (resp. polytopes) \emph{rigid} if it is invariant under all linear (resp. affine) maps, not just the isometries. We may speak of rigid valuations, and we may also speak of rigid multiplicative invariants. 

Of course there are very few rigid invariants. Let us find them all.

The universal rigid valuation for cones in $\mathbb R^n$ (or for all cones of dimension $\le n$) corresponds to the quotient map $L_n\to L^{rig}_n$, where $L_n^{rig}$ is the group of coinvariants for the action of linear automorphisms of $\mathbb R^n$ on $\Sigma(\mathbb R^n)$. The universal rigid valuation for polytopes in $\mathbb R^n$ (or for all polytopes of dimension $\le n$) corresponds to the quotient map $E_n\to E^{rig}_n$, where $E_n^{rig}$ is the group of coinvariants for the action of affine automorphisms of $\mathbb R^n$ on $\mathcal P(\mathbb R^n)$. It is easy to see that these groups form graded rings $L^{rig}=\bigoplus_n L_n^{rig}$ and $E^{rig}=\bigoplus_n E_n^{rig}$, which are the recipients of universal rigid multiplicative invariants.
\begin{prop}\label{rigidE}
For every $n\ge 0$ the group $E_n^{rig}$ is infinite cyclic. More precisely, the universal rigid valuation of polytopes in $\mathbb R^n$ is the Euler characteristic $P\mapsto \chi(P)$. 
\end{prop}
\begin{proof}
$\chi$ is certainly rigid, and the resulting group map $E_n^{rig}\to \mathbb Z$ is certainly surjective for each $n\ge 0$. We must show that it is injective, or equivalently that for any abelian group $A$ the only $A$-valued valuations of polytopes in $\mathbb R^n$ that are rigid are those given by $P\mapsto \chi(P)a$ for some $a\in A$. 

By Remark \ref{d} a valuation $F$ of polytopes in $\mathbb R^n$ is determined by its values on $n$-dimensional simplices. Furthermore, if $F$ is rigid then its values on any two $n$-simplices are equal, so that $F$ is determined by its value on a single $n$-simplex, say $\sigma$. Let $a\in A$ be $F(\sigma)$. Then $F(P)=\chi(P)a$ for every polytope $P\subset \mathbb R^n$ because this is so when $P=\sigma$.
\end{proof}
Proposition \ref{rigidE} implies that $x\bullet \chi$, the scaling of $\chi$ by an indeterminate, gives a ring isomorphism
$$
E^{rig}\cong \mathbb Z\lbrack x\rbrack.
$$
We may also express this as follows:
\begin{prop}
Every rigid multiplicative polytope invariant with values in the ring $k$ is the scaling $c\bullet \chi$ of the Euler characteristic by a unique element $c$ of $k$.
\end{prop}
Turning to the conical case, we have:
\begin{prop}
For every $n\ge 1$ the group $L_n^{rig}$ is free abelian of rank two. More precisely, the universal rigid valuation of cones in $\mathbb R^n$ is $P\mapsto (\epsilon(P),e(P))\in \mathbb Z\times \mathbb Z$. The group $L_0^{rig}$ is infinite cyclic, and the universal rigid valuation of cones in $\mathbb R^0$ is $P\mapsto \epsilon(P)=e(P)\in \mathbb Z.$
\end{prop}
\begin{proof}
$\epsilon$ and $e$ are certainly rigid, and therefore together they give a group homomorphism $L_n^{rig}\to \mathbb Z\times \mathbb Z$. This is surjective for each $n\ge 1$ , since $(\epsilon,e)$ takes the conical $0$-simplex to $(1,1)$ and takes a conical $n$-simplex to $(1,0)$. We must show that it is injective, or equivalently that for any abelian group $A$ the only rigid $A$-valued valuations of cones in $\mathbb R^n$ are those given by $P\mapsto \epsilon(P)a+e(P)b$ for some $a,b\in A$. 

By Remark \ref{0d} a valuation $F$ of cones in $\mathbb R^n$ is determined by its values on $n$-dimensional conical simplices together with its value on the unique conical $0$-simplex $0$. Furthermore, if $F$ is rigid then its values on any two conical $n$-simplices are equal, so that $F$ is determined by its value on $0$ and its value on a single conical $n$-simplex, say $\sigma$. Let $a\in A$ be $F(\sigma)$ and let $b$ be $F(0)-F(\sigma)$. Then $F(P)=\epsilon(P)a+e(P)b$ for every cone $P$ because this is so when $P=\sigma$ and also when $P=0$. 

The assertion about the case $n=0$ is clear.
\end{proof}
Note that the kernel of the surjection $L_n\to L_n^{rig}$ is the subgroup $\text{ker}(\epsilon)\cap\text{ker}(e)$, the $2$-divisible part. Thus if the abelian group $A$ has no nontrivial $2$-divisible subgroup then every $A$-valued Euclidean valuation for cones in $\mathbb R^n$ is rigid. In particular,  if $k$ is a ring (such as $\mathbb Z$ or $\mathbb Z/2\mathbb Z$) whose additive group has no nontrivial $2$-divisible subgroup, then every $k$-valued morphism is rigid.

For the ring structure of $L^{rig}$ we reason as follows. Scaling the invariants $\epsilon$ and $e$ by an indeterminate yields graded ring maps
$$
 x\bullet\epsilon:L^{rig}\to \mathbb Z\lbrack x\rbrack
 $$
 $$
 x\bullet e:L^{rig}\to \mathbb Z\lbrack x\rbrack.
$$
Combining these, we have a graded ring map $L\to \mathbb Z\lbrack x\rbrack\times \mathbb Z\lbrack x\rbrack$. By Proposition \ref{} this is injective and its image has an additive $\mathbb Z$-basis consisting of the element $1=(1,1)$ in degree zero and the elements $(x^n,0)$ and $(0,x^n)$ in all positive degrees. The elements $(x,0)$ and $(0,x)$ are the images of $d$ and $t-d=-d'$. Therefore as a ring $L^{rig}$ is generated by (the images of) $d$ and $t$ subject only to the relation $d(t-d)=0$. We may also express this conclusion as follows:
\begin{prop}\label{domain}A rigid morphism $F$ is determined by its value $F(d)$ on a half-line in $\mathbb R^1$ and its value $F(t)$ on the origin in $\mathbb R^1$. The only constraint on these is
$$
F(d)(F(d)-F(t))=0.
$$
\end{prop}
\begin{rem}
For a rigid morphism $F$, if $F(d)-F(t)$ is zero then $F$ is the scaling $c\bullet \epsilon$ of the yes/no invariant by $c=F(t)=F(d)$; in other words it is the identity morphism of the object $c$. If $F(d)=0$ then $F$ is the scaling $c\bullet e$ of the local Euler invariant by $c=F(t)$. Thus if $k$ is a domain then the only rigid $k$-valued morphisms are the identity morphisms and the scalings of $e$.
\end{rem}
Note that the canonical map $L/2L\to L^{rig}/2L^{rig}$ is an isomorphism. Note also that, while the projection $L_n\to L_n^{rig}$ is a split surjection of groups for each $n$, the projection $L\to L^{rig}$ has no right inverse as a ring map.

\subsection{Relation with the spherical scissors congruence Hopf algebra}

Inside the groupoid $\mathcal G_k$ is the group of morphisms from $0$ to $0$. The representing ring for this functor of $k$ is the quotient $\bar L=L/(tL+sL)$ of $L$ by the ideal generated by $t$ and $s$. Recall that $\hat L=L/tL$ is the spherical scissors congruence ring; $\bar L$ is a further quotient $\hat L/s\hat L$. The group law corresponds to a ring map $\bar L\to \bar L\otimes \bar L$. This makes $\bar L$ a Hopf algebra over $\mathbb Z$, with antipode given by the duality operator $D$. 
The group acts on multiplicative invariants located at $0$; in other words, the ring $\hat E=E/pE$ is a comodule for $\hat L$.

This Hopf algebra structure on a quotient of the spherical scissors congruence ring, with the Euclidean scissors congruence ring as a comodule, appears in \cite{Sah}, in different notation. (Actually Sah's quotient is very slightly smaller; he works with $\hat L/d\hat L$ rather than $\hat L/s\hat L$.)

\section{Manifolds without boundary and Euclidean invariants}
 
Of course the groups $E_n$ and $L_n$ are generated by manifolds with boundary (compact and conical, respectively); the one is generated by convex polytopes and the other by convex cones. In this section we pay special attention to elements of $E_n$ and $L_n$ that are given by manifolds with empty boundary.

We will find that $E_n$ splits as the direct sum $E_n^{ev}\oplus E^{od}_n$ of the subgroup generated by closed even-dimensional manifolds and the subgroup generated by closed odd-dimensional manifolds, and that these coincide with the $+1$ and $-1$ ``eigenspaces'' of the operator $I$. In some sense the most important part of the graded ring $E$ is the subring $$
E^+=E_0^{ev}\oplus E_1^{od}\oplus E_2^{ev}\oplus E_3^{od}\oplus\dots
$$
consisting of elements of $E_n$ fixed by $(-1)^nI$. We will see that $E$ is a free module of rank two over $E^+$.

In the case of $L_n$ there are similar but more complicated statements. Again we introduce a graded subring $L^+$ such that $L_n^+$ turns out to be the subgroup generated by those manifolds without boundary whose dimensions are congruent to $n$ mod $2$. Again $L$ is a free module of rank two over $L^+$. In other ways, results about $L$ are less complete than those about $E$, mainly because less is known about spherical scissors congruence groups than about Euclidean scissors congruence groups. In particular it is not known whether $L$ has any $2$-torsion (or any torsion at all). Modulo $2$-torsion we describe $L$ as a free module of rank four over a graded subring $\Theta$ which, modulo $2$-torsion, coincides with $L^+\cap D(L^+)$. This subring can be described in several ways and seems to be the most important part of $L$.

\subsection{Closed manifold elements in $E$}

\begin{defin}
$E_n^{ev}\subset E_n$ is the subgroup generated by all elements $\lbrack M\rbrack_n$ where $M$ is a closed manifold (with Euclidean structure) of even dimension $\le n$.
$E_n^{od}\subset E_n$ is the subgroup generated by all elements $\lbrack M\rbrack_n$ where $M$ is a closed manifold of odd dimension $\le n$.
\end{defin}
\begin{rem}In fact $E_n^{ev}$ is generated by closed manifolds of dimension $n$ or $n-1$, whichever of these numbers is even, and likewise for $E_n^{od}$. To see this, it suffices to show that the point element $p^2$ in $E_2$ belongs to the subgroup generated by closed $2$-manifolds. In fact, by the description of $E_2$ in \S\ref{}, $p^2$ is the difference between a sphere and a projective plane (each with Euclidean structure) having equal $2$-dimensional volume.
\end{rem}
\begin{thm}
$E_n=E_n^{ev}\oplus E_n^{od}$. Moreover, $E_n^{ev}=ker(\mathbb I-I)$ and $E_n^{od}=ker(\mathbb I+I)=im(\mathbb I-I)$, while the group $im(\mathbb I+I)$ has index two in $E_n^{ev}$ and is generated by closed even-dimensional manifolds of even Euler characteristic.
\end{thm}
\begin{proof}
Because $I\circ I=\mathbb I$, $im(\mathbb I\mp I)$ is contained in $ker(\mathbb I\pm I)$ and the quotient is killed by $2$. Moreover, we have inclusions
$$
im(\mathbb I+I)\subset E_n^{ev}\subset ker(\mathbb I-I),
$$
and
$$
im(\mathbb I-I)\subset E_n^{od}\subset ker(\mathbb I+I)
$$
by Remark \ref{boundary double}. This yields a weak form of the result: after inverting the prime $2$ the inclusions displayed above become equalities and the ``eigenspaces'' of $I$ coincide with $E_n^{ev}$ and $E_n^{od}$. 

We next show that $E_n=ker(\mathbb I-I)\oplus ker(\mathbb I+I)$ and that the quotients \newline
$ker(\mathbb I-I)/im(\mathbb I+I)$ and $ker(\mathbb I+I)/im(\mathbb I-I)$ have order two and one respectively. 

The Euler characteristic map $\chi:E_n\to \mathbb Z$ splits $E_n$ as $\mathbb Zp^n\oplus ker(\chi)$, and the involution $I$ preserves this splitting. On the summand $\mathbb Zp^n$, an infinite cyclic group, $I$ acts trivially, so that for this summand $ker(\mathbb I- I)/im(\mathbb I+I)$ has order two and $ker(\mathbb I+ I)$ is trivial. The other summand $ker(\chi)$ is isomorphic to $\hat E_1\oplus\dots \oplus \hat E_n$, by the results of \S\ref{E split}, and is therefore uniquely $2$-divisible (in fact, a real vector space), which is enough to guarantee that it is the direct sum of $ker(\mathbb I- I)$ and $ker(\mathbb I+I)$, and that in this group $ker(\mathbb I- I)=im(\mathbb I+ I)$ and $ker(\mathbb I-I)=im(\mathbb I-I)$.

The rest of the statement of the Theorem now follows using the inclusions above.
\end{proof}

Let us examine this splitting in the light of the results of  \S\ref{E split}. By (\ref{alpha I}) the isomorphism $E_n\cong\bigoplus_{0\le i\le n}\hat E_i$ in the proof of Proposition \ref{alpha} takes $\ker(\mathbb I- I)$ and $ker(\mathbb I+ I)$ respectively into $\bigoplus\hat E_{2j}$ and $\bigoplus\hat E_{2j+1}$. Therefore under that isomorphism the involution $I:E_n\to E_n$ corresponds to the involution of $\bigoplus_{0\le i\le n}\hat E_i$ that acts like $(-1)^i$ on $\hat E_i$.

Let $E^+_n\subset E_n$ consist of those elements $\xi$ such that $I(\xi)=(-1)^n\xi$, and let $E^-_n$ consist of those such that $I(\xi)=-(-1)^n\xi$. By the theorem above, these coincide with $E_n^{ev}$ and $E_n^{od}$ or vice versa, according to the parity of $n$. Note that $E^+$ is a graded subring of $E$ and that $E^-$ is a module for $E^+$.

\begin{thm}\label{plus}As a module for the subring $E^+$, $E$ has a basis consisting of $1$ and $p$. \end{thm}
\begin{proof}
The element $p$ belongs to $E_1^-$. It remains to prove that the group homomorphism $E^+_{n-1}\to E^-_n$ given by multiplication by $p$ is an isomorphism. It is injective because $p$ is not a zero-divisor in $E$. It is surjective for the simple reason that if the dimension of a manifold is $\le n$ and of the opposite parity from $n$ then it is $\le n-1$ and of the same parity as $n-1$. 
\end{proof}

Note that in the ring isomorphism of Proposition \ref{alpha} the subring $E^+$ of $E$ corresponds to the subring $\hat E\lbrack x^2\rbrack$ of $\hat E\lbrack x\rbrack$.

\subsection{Interior manifold elements in $L$}\label{int elem}
Having examined the subgroup of $E_n$ that is generated by even- or odd-dimensional manifolds without boundary, we now do the same for $L_n$. Call a cone an \emph{interior manifold cone} if it is a nonempty manifold with empty boundary, and call it a \emph{boundary manifold cone} if it is a manifold with nonempty boundary. (The discussion in \S\ref{interior cone} is relevant here.)
\begin{defin}
$L^{ev}_n$ is the subgroup of $L_n$ generated by even-dimensional interior manifold cones. $L^{od}_n$ is the subgroup of $L_n$ generated by odd-dimensional interior manifold cones.
\end{defin} 
\begin{rem}In fact $L_n^{ev}$ is generated by interior manifold cones of dimension $n$ or $n-1$, whichever of these numbers is even, and likewise for $E_n^{od}$. To see this, it suffices to verify that the point element $t^2$ in $L_2$ belongs to the subgroup generated by two-dimensional interior manifold cones. The element $t^2+a(\theta)$ belongs to this subgroup for all $\theta>0$, by Remark \ref{}, and so therefore does $2(t^2+a(\pi))-(t^2+a(2\pi))=t^2$.
\end{rem}

As before, we have inclusions
$$
im(\mathbb I+I)\subset L_n^{ev}\subset ker(\mathbb I-I)
$$
$$
im(\mathbb I-I)\subset L_n^{od}\subset ker(\mathbb I+I).
$$
Thus after inverting the prime $2$ it is true that $L$ becomes the direct sum of $L_n^{ev}$ and $L_n^{od}$ and that these are the $+1$ and $-1$ eigenspaces of $I$. 

Without inverting $2$ there are several differences between the polytope case and the cone case.
One difference is that, whereas $E_n$ is generated by closed manifolds, $L_n$ is not quite generated by interior manifold cones if $n>0$. In fact, the valuation $\epsilon-e:L_n\to \mathbb Z$ takes even values (namely $0$ or $2$) on interior manifold cones and takes the value $1$ on boundary manifold cones such as the half-line element $t^{n-1}d$. On the other hand, this is the only obstruction:
\begin{prop}\label{M}
For every $n>0$ the subgroup $L_n^{ev}+L_n^{od}\subset L_n$ has index two. 
\end{prop}
\begin{proof}
For every $\xi\in L_n$ we have 
$$
2\xi=(\xi +I(\xi))+(\xi-I(\xi))\in L_n^{ev}+L_n^{od}.
$$
Therefore the quotient $L_n/(L_n^{ev}+L_n^{od})$ is killed by two. By Proposition \ref{L/2} it follows that the quotient is generated by the point element $t^n$ and the half-line element $t^{n-1}d$. Since the point element belongs to $L_n^{ev}$, the quotient is generated by $t^{n-1}d$ alone.
\end{proof}

We now define a subring $L^+\subset L$ analogous to $E^+$ and prove a statement similar to Theorem \ref{plus}, using a necessarily different method. 

$L_n^+$ is designed to coincide with $L_n^{ev}$ or $L_n^{od}$ according to parity. We do not define it to be the kernel of $\mathbb I-(-1)^nI$ (as we did $E_n^+$). Instead we use the operator $\delta_n:\mathcal P_n\to \mathcal P_{n-1}$ of \S\ref{boundaries} to introduce a map $\delta_n:L_n\to L_{n-1}$ such that 
\begin{equation}\label{t delta}
t\delta_n(\xi)=\xi-(-1)^nI\xi,
\end{equation}
and we define $L_n^+$ to be the kernel of $\delta_n$. 
\begin{lemma}
There is a group homomorphism $\delta_n:L_n\to L_{n-1}$ such that 
\begin{equation}\label{one}
\delta_n\langle M\rangle_n=\langle \partial M\rangle_{n-1}
\end{equation}
if $M$ is a boundary manifold cone of dimension $n-2k\le n$,
\begin{equation}\label{two}
\delta_n\langle M\rangle_n=\langle\mathcal DM\rangle_{n-1}
\end{equation}
if $M$ is a boundary manifold cone of dimension $n-2k-1\le n$,
\begin{equation}\label{three}
\delta_n\langle M\rangle_n=0
\end{equation}
if $M$ is an interior manifold cone of dimension $n-2k\le n$,
and \begin{equation}
\delta_n\langle M\rangle_n=2\langle M\rangle_{n-1}
\end{equation}
if $M$ is an interior manifold cone of dimension $n-2k-1\le n$.
\end{lemma}
\begin{rem}
There can be at most one such map, since $L_n$ is generated by conical simplices, which are manifold cones. Clearly any such map will satisfy (\ref{t delta}), since this must be the case when $\xi$ is given by a manifold cone, by Remark \ref{boundary double}. If we knew that $t$ was not a zero-divisor in $L$, then we could use (\ref{t delta}) to define $\delta_n$, and we would also know then that the kernel of $\delta_n$ coincides with the kernel of $\mathbb I-(-1)^nI$.
\end{rem}
\begin{proof}
View $L_n$ as the group of coinvariants for the action of $O_n$ on $\Sigma(\mathbb R^n)=\Sigma_n(\mathbb R^n)=\mathcal P_n(\mathbb R^n,\mathbb R^n-0)$, and (using Remark \ref{other}) view $L_{n-1}$ as the group of coinvariants for the action of $O_n$ on $\mathcal P_{n-1}(\mathbb R^n,\mathbb R^n-0)$.
Recall the map 
$$
\delta_n:\mathcal P_n(\mathbb R^n,\mathbb R^n-0)\to \mathcal P_{n-1}(\mathbb R^n,\mathbb R^n-0)
$$
defined in \S\ref{boundaries}. This map is compatible with the $O_n$-action and therefore yields a map $L_n\to L_{n-1}$. The latter behaves as specified on boundary manifold cones because of (\ref{boundary}) and (\ref{double}).
\end{proof}
In addition to (\ref{t delta}) we have
\begin{equation}\label{delta t}
\delta_{n+1}(t\xi)=\xi+(-1)^nI\xi.
\end{equation}
We will write $\delta$ instead of $\delta_n$ when there is no danger of ambiguity. We have
$$
\delta(t)=2\hskip .2 in \delta(d)=1\hskip .2 in \delta(s)=0\hskip .2 in\delta(d')=-1
$$$$
\delta(a(\theta))=0.
$$
Equations (\ref{dd}), (\ref{Idelta_n}), (\ref{delta_nI}), and (\ref{prodrule}) yield 
\begin{equation}\label{sq0}
\delta\circ\delta=0
\end{equation}
\begin{equation}\label{Idelta}
(\mathbb I+(-1)^nI)\circ\delta_n=0
\end{equation}
\begin{equation}\label{deltaI}
\delta_n\circ (\mathbb I+(-1)^nI)=0
\end{equation}
and
\begin{equation}\label{leib}
\delta(\xi\eta)=\delta(\xi)\eta+(-1)^iI(\xi)\delta(\eta)
\end{equation}
when $\xi\in L_i$.
\begin{defin}$L_n^+$ is the kernel of $\delta_n:L_n\to L_{n-1}$.
\end{defin}
By (\ref{leib}), $L^+$ is a graded subring of $L$.
\begin{thm}\label{L+}
$L_n^+$ coincides with the image of $\delta_{n+1}:L_{n+1}\to L_n$. As a module for $L^+$, $L$ has a basis consisting of $1$ and $d$. The group $L_n^+$ is equal to $L_n^{ev}$ if $n$ is even, and it is equal to $L_n^{od}$ if $n$ is odd. If $n>0$ then $L_n^+$ coincides with the image of $\mathbb I+(-1)^nI$.  
\end{thm}
\begin{proof}
First, $\delta(L)\subset L^+$ by (\ref{sq0}). We now use the pair of equations
\begin{equation}\label{key1}
\delta(d'\eta)=-\eta+d\delta(\eta)
\end{equation}
\begin{equation}\label{key2}
\delta(d\eta)=\eta+d'\delta(\eta),
\end{equation}
which follow from (\ref{leib}). Equation (\ref{key1}) shows that $L^+=\delta(L)$: if $\delta(\eta)=0$ then $\eta=\delta(-d'\eta)$. It also shows that $L=L^++dL^+$: any element $\eta$ can be written as $\delta(-d'\eta)+d\delta(\eta)$. To see that $1$ and $d$ are linearly independent over $L^+$, suppose that $\zeta+d\eta=0$ with both $\delta(\zeta)=0$ and $\delta(\eta)=0$. Then $\delta(d\eta)=0$, so that by (\ref{key2}) $\eta=0$ and therefore $\zeta=0$. The statement that $ker(\delta_n)$ contains $L_n^{ev}$ if $n$ is even or $L_n^{od}$ if $n$ is odd follows from (\ref{three}). The statement that the image of $\delta_{n+1}$ is generated by elements of $L_n^{ev}$ if $n$ is even or $L_n^{od}$ if $n$ is odd follows from (\ref{one}) and (\ref{two}). For the final statement, (\ref{deltaI}) gives one inclusion. To obtain the reverse inclusion, recall that Proposition \ref{L/2} implies that any element of $L_{n+1}$ has the form $2\xi+t\eta$, so that any element of $L_n^+$ has the form $\delta_{n+1}(2\xi+t\eta)$. This is in the image of $\mathbb I+(-1)^nI$ because
$$
\delta_{n+1}(2\xi)=\delta_{n+1}\xi+(-1)^nI(\delta_{n+1}\xi)
$$
by (\ref{Idelta}) and
$$ 
\delta_{n+1}(t\eta)=\eta+(-1)^n I(\eta)
$$
by (\ref{delta t}). 
\end{proof}
\begin{rem}Ultimately this proof works because $\delta(d)=1$. This reflects the fact that locally a point is a boundary.
\end{rem}
\begin{rem}Whereas $E$ may be obtained from $E^+$ by adjoining a square root of $p^2$, the relationship between $L$ and $L^+$ is a bit more subtle. The unique monic degree two equation satisfied by $d$ over $L^+$ is
$$
d^2-sd+a(\pi/2)=0.$$
Since $s=d+d'$ and $a(\pi/2)=dd'$, the other root of the same polynomial is $d'$.
\end{rem}
\begin{rem}
One can define a map $\delta_n:E_n\to E_{n-1}$ analogous to $\delta_n:L_n\to L_{n-1}$. The fact that $p$ is not a zero divisor (\S\ref{E split}) implies that this is the \emph{unique} map such that $p\delta_n(\xi)=\xi-(-1)^nIx$ for all $\xi\in E_n$, and also that its kernel coincides with $ker(\mathbb I-(-1)^nI)=E_n^+$.
\end{rem}
\begin{rem}
Both in $E_n$ and in $L_n$ we have the following relations between subgroups:
\begin{equation}\label{string}
im(\mathbb I+ (-1)^nI)\subset im(\delta)\subset ker(\delta)\subset ker(\mathbb I- (-1)^nI),
\end{equation}
where the quotient $ker(\mathbb I- (-1)^nI)/im(\mathbb I+ (-1)^nI)$ is obviously killed by $2$. Let us summarize the more detailed information that we have in the two cases.

For $E_n$ the quotient $ker(\mathbb I- (-1)^nI)/im(\mathbb I+(-1)^n I)$ has order one or two according to whether $n$ is odd or even. The nontrivial subquotient in the even case is $ker(\delta)/im(\delta)$; it is detected by $mod\ 2$ Euler characteristic and generated by the point element $p^n$. The group $E_n^+=ker(\mathbb I-(-1)^n I)=ker(\delta)$ is $E^{ev}_n$ if $n$ is even and $E^{od}_n$ if $n$ is odd.

For $L_n$ the second of the three inclusions in (\ref{string}) is an equality, as well as the first if $n>0$. The group $L^+_n=ker(\delta)= im(\delta)$ is $L^{ev}_n$ if $n$ is even and $L^{od}_n$ if $n$ is odd. It is not clear whether this is the same as $ker(\mathbb I-(-1)^nI)$. It must be the same if the ``inclusion'' $L_{n-1}\to L_n$ is injective. This ends the remark.

\end{rem}

Because $D:L\to L$ is a graded ring map, it takes $L^+$ to a graded subring $D(L^+)\subset L$. The group $D(L_n^+)$ is the kernel of the operator 
$$
\bar\delta=D\circ\delta\circ D:L_n\to L_{n-1}.
$$
Because $Dt=s$ and $D\circ I\circ D=(-1)^nI$ in $L_n$, (\ref{t delta}) and (\ref{delta t}) yield
$$
s\bar\delta(\xi)=\xi- I(\xi)
$$
$$
\bar\delta(s\xi)=\xi+ I(\xi).
$$
\begin{thm}
$D(L_n^+)$ is both the kernel of $\bar\delta_n$ and the image of $\bar\delta_{n+1}$. As a module for $D(L^+)$, $L$ has a basis consisting of $1$ and $Dd=d$. If $n>0$ then $D(L_n^+)$ coincides with the image of $\mathbb I+ I$. For every $n$ we have $D(L_n^+)=L_n^{ev}$.
\end{thm}
\begin{proof}
All of this except the last statement follows directly from Theorem \ref{L+}. To see that $D(L_n^+)=L_n^{ev}$, first suppose that $n$ is even. If $n=0$ then both sides are $L_0$. If $n>0$ then we have
$$
D(L_{n}^+)=D(im(\mathbb I+ I))=im(D\circ (\mathbb I+I))=im((\mathbb I+I)\circ D)=im(\mathbb I+I)=L_{n}^+=L_{n}^{ev},
$$
where the first equality and the last two use Theorem \ref{L+}. 

Now suppose that $n$ is odd. We have
$$
D(L_{n}^+)=D(im(\mathbb I- I))=im(D\circ (\mathbb I-I))=im((\mathbb I+I)\circ D)=im(\mathbb I+I)\subset L_{n}^{ev},
$$
where the first equality uses Theorem \ref{L+}. On the other hand, we also have 
$$
L_n^{ev}=tL_{n-1}^{ev}=tD(L_{n-1}^+)=D(sL_{n-1}^+)\subset D(L_n^+).
$$
The first equality holds simply because an even number less than or equal to $n$ is less than or equal to $n-1$. The second uses the even case of the present result. The final inclusion holds because $s\in L_1^+$.
\end{proof}
\begin{cor}\label{sL}
$L_{2k+1}^{ev}=tL_{2k}^+$ and $L^+_{2k+1}=L_{2k+1}^{od}=D(L_{2k+1}^{ev})=sL_{2k}^+$.
\end{cor}

\subsection{The graded ring $\Theta$}

Let us write 
$$
\Theta_{2k}=L^+_{2k}=L^{ev}_{2k}=D(L^+_{2k})
$$
$$
\Theta_{2k+1}=0.
$$
Then $\Theta$ is a graded subring of $L$ contained in $L^+\cap D(L^+)$ and potentially equal to it. According to Corollary \ref{sL}, $L^+$ is generated as a $\Theta$-module by $1$ and $s$. By Theorem \ref{L+} it follows that $L$ is generated as a $\Theta$-module by $1$, $s$, $d$, and $ds$.

\subsection{Potential $2$-torsion in $L$}

Except for some small values of $n$, we do not know whether $L_n$ has any $2$-torsion (or indeed any torsion at all). This is unfortunate, because the following types of elements of $L_n$ are all killed by $2$:
\begin{itemize}

\item Any element of $ker(\mathbb I-I)\cap ker(\mathbb I+I)$.

\item Any element of $L^+_{2k+1}\cap D(L^+_{2k+1})$. In fact, this is contained in \newline
$ker(\mathbb I-I)\cap ker(\mathbb I+I)$.

\item Any $\xi\in L^+_{n}$ such that $t\xi=0$. Indeed, for such an element \newline $0=\delta(t\xi)=\xi+(-1)^nI(\xi)=2\xi$.

\item Any $\xi\in L^+_{2k}$ such that $s\xi=0$. Indeed, if $\xi$ is such an element then $D\xi\in L^+_{2k}$ and $tD\xi=D(s\xi)=0$.

\item $\delta(\xi)$ if $\xi\in L_{n+1}$ satisfies $I(\xi)=(-1)^n\xi$. Indeed, in that case $2\delta(\xi)=\newline\delta(\xi+(-1)^nI(\xi))=0$.
\end{itemize}
Therefore, \emph{if it happens that $L$ has no $2$-torsion}, the following statements are true:
\begin{itemize}
\item $\Theta=L^+\cap D(L^+)$
\item As a $\Theta$-module, $L^+$ has a basis consisting of $1$ and $s$.
\item As a $\Theta$-module, $L$ has a basis consisting of $1$, $s$, $d$, and $ds$.
\item $\ker(\delta_n)=ker(\mathbb I-(-1)^nI)$ in $L_n$.
\item In $L_{2k+1}$ the subgroup generated by all interior manifold elements is the direct sum of 
$$L^{ev}_{2k+1}=t\Theta_{2k}=L^+_{2k+1}=ker(\delta)=ker(\mathbb I+I)$$ 
and
$$L^{od}_{2k+1}=s\Theta_{2k}=D(L^+_{2k+1})=ker(\bar\delta)=ker(\mathbb I-I).$$
\item In $L_{2k}$ the subgroup generated by all interior manifold elements is the direct sum of 
$$L^{ev}_{2k}=\Theta_{2k}=L^+_{2k}=ker(\delta)=ker(\bar\delta)=ker(\mathbb I-I)$$ 
and
$$L^{od}_{2k}=st\Theta_{2k-2}=ker(\mathbb I+I).$$
\end{itemize}
We repeat that these statements rely on the assumption that $L$ has no $2$-torsion. 

For what it is worth, these statements become true if $L$ is replaced by $L/J$, where $J$ is the ideal of $2$-torsion elements.


\end{document}